\documentclass[11pt]{amsart}
\usepackage{amssymb}
\usepackage{amsmath}
\usepackage{amsthm}
\usepackage{verbatim}
\usepackage[all]{xy}
\usepackage{url}
\usepackage{mathtools}
\usepackage{dsfont}
\usepackage{pdflscape}
\usepackage[normalem]{ulem}
\usepackage{extarrows}

\usepackage[mathscr]{euscript}
\usepackage{calligra}
\usepackage[T1]{fontenc}
\usepackage{tikz-cd}
\usetikzlibrary{calc}
\usepackage[top=2cm, bottom=2.1cm, left=2.5cm, right=2.7cm]{geometry}
\usepackage%[pagebackref=true]
{hyperref}

\DeclareFontEncoding{OT2}{\ha}{\ha} % to enable usage of cyrillic fonts
\DeclareTextFontCommand{\textcyr}{\fontencoding{OT2}
    \fontfamily{wncyr}\fontseries{m}\fontshape{n}\selectfont}

\DeclareSymbolFont{rsfs}{U}{rsfs}{m}{n}
\DeclareSymbolFontAlphabet{\mathrsfs}{rsfs}

\usepackage{color}

\theoremstyle{plain}
\newtheorem{theorem}{Theorem}[section]

\newtheorem{proposition}[theorem]{Proposition}
\newtheorem{lemma}[theorem]{Lemma}
\newtheorem{corollary}[theorem]{Corollary}

\newtheorem{fact}[theorem]{Fact}

\theoremstyle{definition}

\newtheorem{example}[theorem]{Example}
\newtheorem{examples}[theorem]{Examples}
\newtheorem{remark}[theorem]{Remark}
\newtheorem{remarks}[theorem]{Remarks}
\newtheorem{definition}[theorem]{Definition}
\newtheorem{construction}[theorem]{Construction}

\newtheorem*{notation*}{Notation}

\renewcommand{\mathbb}{\mathds}

\newcommand{\Z}{{\mathds Z}}
\newcommand{\Q}{{\mathds Q}}
\newcommand{\R}{{\mathds R}}
\newcommand{\C}{{\mathds C}}

\newcommand{\kbar}{{\bar k}}

\newcommand{\ssc}{{\rm sc}}
\newcommand{\sscp}{{(\ssc)}}
\newcommand{\sss}{{\rm ss}}
\newcommand{\sssp}{{(\sss)}}

\newcommand{\im}{{\rm im}}
\newcommand{\ab}{{\rm ab}}
\newcommand{\coker}{{\rm coker\hs}}
\newcommand{\loc}{{\rm loc}}
\newcommand{\re}{{\rm re}}

\newcommand{\Gal}{{\rm Gal}}
\newcommand{\Hom}{{\rm Hom}}

\newcommand{\cZ}{{\mathcal Z}}
\newcommand{\X}{{\sf X}}
\newcommand{\V}{{\mathcal V}}

\newcommand{\into}{{\,\hookrightarrow\,}}
\newcommand{\onto}{{\,\twoheadrightarrow\,}}
\newcommand{\isoto}{{\,\overset\sim\longrightarrow\,}}

\newcommand{\lra}{\longrightarrow}

\newcommand{\Maps}{{\rm Maps}}

\newcommand{\emm}{\bfseries}

\newcommand{\cok}{{\rm cok}}
\newcommand{\fppf}{{\rm fppf}}

\newcommand{\qt}{{\rm qt}}

\newcommand{\labelt}[1]{\xrightarrow{\makebox[0.8em]{\scriptsize ${#1}$}}}
\newcommand{\labelto}[1]{\xrightarrow{\makebox[1.3em]{\scriptsize ${#1}$}}}

\newcommand{\functor}{\rightsquigarrow}

\newcommand{\hs}{\kern 0.8pt}
\newcommand{\hssh}{\kern 1.2pt}
\newcommand{\hshs}{\kern 1.6pt}
\newcommand{\hssss}{\kern 2.0pt}
\newcommand{\ha}{{\kern 1pt}}

\newcommand{\hm}{\kern -0.8pt}
\newcommand{\hmm}{\kern -1.2pt}

\newcommand{\vk}{{\varkappa}}

\newcommand{\triv}{{\rm triv}}

\newcommand{\SmallMatrix}[1]{\text{{%
\Small
\arraycolsep=0.4\arraycolsep\ensuremath
    {\begin{pmatrix}#1\end{pmatrix}}}}}

\newcommand{\mH}{{\scriptscriptstyle{H}}}
\newcommand{\mG}{{\scriptscriptstyle{G}}}
\newcommand{\mZ}{{\scriptscriptstyle{Z}}}

\newcommand{\tdash}{\text{-}}
\newcommand{\br}{\{\tdash\hs,\hm\tdash\}}

\newcommand{\inn}{{\rm inn}}
\newcommand{\Aut}{{\rm Aut}}
\newcommand{\Out}{{\rm Out}}
\newcommand{\GL}{{\rm GL}}
\newcommand{\Lie}{{\rm Lie}}
\newcommand{\crr}{{\rm cr}}
\newcommand{\Red}{{\mathcal{R}ed}}
\newcommand{\CrMod}{{\mathcal{C}r\hm\mathcal{M}od}}

\newcommand{\cG}{{\mathcal G}}
\newcommand{\cA}{{\mathcal A}}
\newcommand{\Abar}{{\bar A}}

\newcommand{\SAut}{{\rm SAut}}
\newcommand{\SOut}{{\rm SOut}}
\newcommand{\Inn}{{\rm Inn}}

\newcommand{\SO}{{\rm SO}}
\newcommand{\OO}{{\rm O}}

\newcommand{\Grp}{{\mathcal{G}rp}}
\newcommand{\Alg}{{\mathcal{A}lg}}
\newcommand{\Set}{{\mathcal{S}et}}

\newcommand{\PCrM}{{\rm PCrM}}
\newcommand{\CrM}{{\rm CrM}}

\newcommand{\ul}{\underline}
\newcommand{\ol}{\overline}

\newcommand{\upsig}{{\sigma\!}}
\newcommand{\upst}{{\sigma\tau\!}}
\newcommand{\upg}{{g\!}}

\renewcommand{\nu}{\upsilon}

\newcommand{\psitil}{{\tilde\psi}}

\begin{document}

\title[Quasi-connected reductive groups]
{Abelian Galois cohomology of
quasi-connected\\ reductive groups}

\author{Mikhail Borovoi and Taeyeoup Kang}

\address{Raymond and Beverly Sackler School of Mathematical Sciences,
Tel Aviv University, 6997801 Tel Aviv, Israel}
\email{borovoi@tauex.tau.ac.il}
\email{taeyeoupkang@gmail.com}

\thanks{This research was supported
by the Israel Science Foundation (grant 1030/22).}



\date{\today}

\begin{abstract}
In 1999 Labesse introduced quasi-connected reductive groups
and investigated their abelian Galois cohomology
over local and global fields of characteristic 0.
We (1) generalize some of the constructions of Labesse from quasi-connected reductive groups
to arbitrary reductive groups, not necessarily connected or quasi-connected;
(2) generalize results of Labesse on the abelian Galois cohomology of quasi-connected reductive groups
to the case of  local and global fields of arbitrary characteristic;
and (3) investigate the functoriality properties of the abelian Galois cohomology.
In particular, we introduce the notion of a principal homomorphism of quasi-connected reductive groups,
and show that if $G$ is a quasi-connected reductive group
over a local or global field $k$ of  {\em positive} characteristic,
then the first Galois cohomology set $H^1(k,G)$ has a canonical abelian group structure,
which is functorial with respect to  {\em principal} homomorphisms.
\end{abstract}

\begin{comment} Abstract for arXiv

In 1999 Labesse introduced quasi-connected reductive groups
and investigated their abelian Galois cohomology
over local and global fields of characteristic 0.
We (1) generalize some of the constructions of Labesse from quasi-connected reductive groups
to arbitrary reductive groups, not necessarily connected or quasi-connected;
(2) generalize results of Labesse on the abelian Galois cohomology of quasi-connected reductive groups
to the case of  local and global fields of arbitrary characteristic;
and (3) investigate the functoriality properties of the abelian Galois cohomology.
In particular, we introduce the notion of a principal homomorphism of quasi-connected reductive groups,
and show that if G is a quasi-connected reductive group
over a local or global field $k$ of  *positive* characteristic,
then the first Galois cohomology set H^1(k,G) has a canonical structure of abelian group,
which is functorial with respect to *principal* homomorphisms.

\end{comment}

\maketitle

%\tableofcontents
%\setcounter{section}{-1}

\section{Introduction}
\label{s:Intro}

Let $k$ be a field. We write $\kbar$ for a fixed algebraic closure of $k$,
$k_s$ for the separable closure of $k$ in $\kbar$, and $\Gamma=\Gal(k_s/k)$.

By a {\em $k$-group} we mean a linear algebraic group over $k$,
that is, an affine group scheme of finite type over $k$,
not necessarily smooth.
By a {\em $k$-algebra} we mean a commutative associative unital algebra over $k$,
possibly containing nilpotents.

\subsection*{A crossed module from an algebraic group}
Let $G$ be a reductive $k$-group.
By this we mean that $G$ is smooth and that its identity component $G^0$ is reductive.
Unlike \cite{SGA3}, we do not assume that $G$ is connected.
Let $G^\sss\coloneqq[G^0,G^0]$ denote the derived subgroup
of the connected reductive group $G^0$,
and let $\pi\colon G^\ssc\onto G^\sss$ denote the universal covering of $G^\sss$.
The $k$-group $G^\sss$ is semisimple, and $G^\ssc$ is (semisimple) simply connected,
which explains our notations.
We consider the composite homomorphism
\[\rho\colon G^\ssc\onto G^\sss\into G^0\into G.\]
The $k$-groups $G^0$, $G^\sss$, and $G^\ssc$ depend functorially on $G$;
see Proposition \ref{p:GA} for the functoriality of $G^\ssc$.
Therefore, the left action of $G$ on itself by conjugation induces a left action of $G$ on $G^\ssc$.
Thus we obtain an action morphism of $k$-varieties
\[\theta\colon G\times_k G^\ssc\to G^\ssc.\]
We check that $\rho$ and $\theta$ are compatible, that is, that $\CrM(G)\coloneqq (G^\ssc,G,\rho,\theta)$
is a {\em crossed module of $k$-groups};
see Proposition \ref{p:crossed-module}.
We sometimes  write
\[ \CrM(G)=(G^\ssc\labelt\rho G,\theta)\]
to emphasize that we regard $\CrM(G)$ as a {\em complex} of $k$-groups (in degrees $-1$ and $0$).

We define the {\em crossed cohomology pointed set}
\[ H^1_\crr(k,G)\coloneqq H^1\big(k,\CrM(G)\big)\coloneqq H^1\Big(\Gal(k_s/k),\,\big(G^\ssc(k_s)\labelt\rho G(k_s),\theta\big)\hs\Big) \]
where the group (hyper-)cohomology with coefficients in a crossed module was defined in \cite[Section 3]{Borovoi-Memoir};
see also Section \ref{s:coboundary} below.
The inclusion morphism of crossed modules
\[ (1\to G)\,\into\, (G^\ssc\labelt\rho G,\theta)\]
induces the {\em crossing map}
\begin{equation}\label{e:crossing-map}
\crr^1\colon\, H^1(k,G)=H^1(k,1\to G)\,\lra H^1\big(k,\CrM(G)\big)\eqqcolon H^1_\crr(k,G).
\end{equation}

\begin{theorem}[Theorems \ref{t:cr-surjective} and \ref{t:cr-bijective}]
\label{t:local-global-crossing}
Let $G$ be a reductive $k$-group such that the  center $Z(G^\ssc)$ is \'etale, that is, smooth. Then:
\begin{enumerate}
\item[\rm (i)] When $k$ is a local field or a global field, the crossing map \eqref{e:crossing-map} is surjective.
\item[\rm (ii)] Moreover, when $k$ is a {\emm non-archimedean} local field
or a global field {\emm without real places}
(that is, a global function field or a totally imaginary number field),
the crossing map \eqref{e:crossing-map} is bijective.
\end{enumerate}
\end{theorem}

We show in Sections \ref{s:crossed} and \ref{s:cr-cohomology}
that a homomorphism of reductive $k$-groups
$\vk\colon G_1\to G_2$
induces a morphism of crossed modules
\begin{equation} \label{e:vk*}
\vk_*\colon \CrM(G_1)\to\CrM(G_2)
\end{equation}
which in turn induces a morphism of pointed sets
\begin{equation}\label{e:vk-crr}
\vk_\crr\colon H^1_\crr(k,G_1)\to H^1_\crr(k,G_2).
\end{equation}

\subsection*{Quasi-connected reductive groups}
Labesse \cite{Labesse} introduced the notion of quasi-connected reductive groups,
which appear in the Arthur-Selberg trace formula as the stabilizers of outer automorphisms of the
connected reductive group under consideration.

\begin{definition}
\label{d:Labesse}
A {\em quasi-connected reductive $k$-group}
is a {\em smooth} linear algebraic group $G$ over $k$ that is isomorphic to the kernel
of a surjective homomorphism $H\to S$
where $H$ is a {\em connected} reductive $k$-group and $S$ is a $k$-torus.
\end{definition}

This definition is the same as the definition of Labesse \cite[Definition 1.3.1]{Labesse},
except that Labesse does not require $G$ to be smooth.
Our working definition is Definition \ref{d:q-c-r} below, which is equivalent to Definition \ref{d:Labesse}.

Alternatively, a linear algebraic $k$-group is quasi-connected reductive
if and only if it is isomorphic to the quotient by a central finite $k$-subgroup
of the product of a {\em $k$-quasi-torus}
and a connected reductive $k$-group; see \cite[Theorem 2.13]{BGR}.
Here, following \cite[Section 3.2.3]{OV}, we use the short term ``$k$-quasi-torus''
instead of the somewhat clumsy term ``smooth $k$-group of multiplicative type''.

For a {\em quasi-connected} reductive $k$-group $G$,
we define, following an idea of Deligne \cite[\S2.0.2]{Deligne},
a canonical {\em Picard braiding} of $\CrM(G)$, that is, a canonical morphism of $k$-varieties
\[ \br\colon G\times_k G\to G^\ssc\quad\
\text{with the property}\quad\
\rho\big(\{g_1,g_2\}\big)=[g_1,g_2]\coloneqq g_1 g_2 g_1^{-1} g_2^{-1}\]
and having other good properties; see Construction \ref{cons-braiding} and Proposition \ref{p:br} below.
Thus we obtain a Picard crossed module
\[\PCrM(G)=\big(\CrM(G),\br\big)= \big(G^\ssc\labelt\rho G,\theta,\br\big).\]
Following an idea of Breen \cite{Breen},
we use Deligne's Picard braiding $\br$
to define in Appendix \ref{app:gr-coh-Picard} a canonical abelian group structure
on the pointed set $H^1_\crr(k,G)$.
Thus we obtain an abelian group $H^1_\ab(k,G)$.

For a quasi-connected reductive $k$-group $G$, we consider the crossing map \eqref{e:crossing-map},
which now takes values in the abelian group $H^1_\ab(k,G)$,
and so we call it the {\em abelianization map} and denote it by $\ab^1$:
\begin{equation}\label{e:abel-map}
\ab^1=\crr^1\colon\, H^1(k,G)\,\lra\,H^1_\crr(k,G)\eqqcolon H^1_\ab(k,G).
\end{equation}

\begin{corollary}\label{c:local-global-abel}
Let $G$ be a quasi-connected reductive $k$-group
such that the center $Z(G^\ssc)$ is \'etale. Then:
\begin{enumerate}
\item[\rm (i)] When $k$ is a local or global field, the abelianization map \eqref{e:abel-map} is surjective.
\item[\rm (ii)] Moreover, when $k$ is a {\emm non-archimedean} local field
or a global field {\emm without real places}
(that is, a global function field or a totally imaginary number field),
the abelianization map \eqref{e:abel-map} is bijective.
\end{enumerate}
\end{corollary}

The corollary follows immediately from Theorem \ref{t:local-global-crossing}.
This corollary was earlier proved by Labesse \cite[Proposition 1.6.7]{Labesse}
in the case when $k$ is a local or global field {\em of characteristic} 0.
Note that when ${\rm char}(k)\ne 2$,
this corollary applies to the non-connected quasi-connected reductive $k$-group
$G=\OO_{n,F}$ of Example \ref{ex:OOnF} below (the orthogonal group of the quadratic form
given by a non-degenerate symmetric $n\times n$-matrix $F$ for {\em odd} $n\ge 3$).
Here $Z(G^\ssc)\cong\mu_2$, and so $Z(G^\ssc)$ is \'etale when ${\rm char}(k)\ne 2$.

By Proposition \ref{p:q-ab} below, when $G$ is a quasi-connected reductive $k$-group,
the crossed module $\CrM(G)=(G^\ssc\labelt\rho G,\theta)$ is a {\em quasi-abelian}
crossed module in the sense of \cite{GA-12}; see Proposition \ref{p:q-ab} below for the definition.
Therefore, our Corollary \ref{c:local-global-abel}
is a direct analog of \cite[Corollary 4.5(i) and Theorem 5.8(i)]{GA-12}
in the context of Springer's Galois non-abelian cohomology of \cite{Springer}.
Note that, since we work with {\em Galois} cohomology, we have to assume that $Z(G^\ssc)$ is \'etale, whereas Gonz\'alez-Avil\'es \cite{GA-12},
who works with Giraud-Douai flat non-abelian cohomology,
does not assume this.
On the other hand, in \cite[Theorem 5.8(i)]{GA-12} the reductive $k$-group $G$ is assumed to be connected,
whereas we consider the more general case of a quasi-connected reductive $k$-group.

In the case where $G$ is a quasi-connected reductive $k$-group with $k$ as in Corollary
\ref{c:local-global-abel}(ii), the bijection \eqref{e:abel-map}
induces a canonical abelian group structure on the pointed set $H^1(k,G)$.
This abelian group structure is functorial in $G$ with respect to
{\em principal} homomorphisms $G\to G'$; see Definition \ref{d:principal-intro} below.

\subsection*{Functoriality}
We investigate the subtle question of the functoriality of the assignment
$G\rightsquigarrow  H^1_\ab(k,G)$.
Let $\vk\colon G_1\to G_2$ be a homomorphism
of quasi-connected reductive $k$-groups.
Then we have the induced morphism of crossed modules \eqref{e:vk*}
and the induced morphism of pointed sets \eqref{e:vk-crr}.
One might expect \eqref{e:vk*} to preserve the canonical braiding
and \eqref{e:vk-crr} to preserve the group structure.
However, in general these expected properties do not hold.

\begin{example}\label{ex:intro}
Let $k=\R$ and let $B$ and $G$ be quasi-connected
reductive $\R$-groups where $B=\mu_2\times \mu_2$
and $G={\rm SU}_2/\mu_2$.
Let  $\iota\colon B\into G$ be a certain injective homomorphism
whose image is not contained in any maximal torus of $G$;
see Example \ref{ex:PU2} and Remark \ref{r:crr-not-hom} below.
One checks that the map
$\iota_\crr\colon  H^1_\ab(k,B)\to H^1_\ab(k,G)$
sends the three non-trivial elements of the group $H^1_\ab(\R,B)=B(\R)$ of order 4
to the non-trivial element $-1\in H^1_\ab(\R,G)=H^2(\R,\mu_2)=\{\pm1\}$.
Thus $\iota_\crr$ is not a group homomorphism.
\end{example}

We introduce a class of morphisms of quasi-connected reductive $k$-groups
that do preserve the canonical braiding of $\CrM(G)$
and the abelian group structure on $H^1_\crr(k,G)$.

For a quasi-connected reductive $k$-group $G$, let $T^\sscp\subset G^\ssc$ be a maximal torus,
and write $T=\cZ_G(\rho(T^\sscp))$, the centralizer in $G$ of $\rho(T^\sscp)$.
We show that $T$ is a quasi-torus, that $T=Z(G)\cdot \rho(T^\sscp)$, and that $T^\sscp=\rho^{-1}(T)$.
We say that $T$ is a {\em principal quasi-torus in $G$}.
When $G$ is connected, a principal quasi-torus in $G$ is the same as a maximal torus.

\begin{definition}\label{d:principal-intro}
A homomorphism of quasi-connected reductive $k$-groups
$\vk\colon G_1\to G_2$ is called {\em principal} if for some principal quasi-torus $T_1\subseteq G_1$
there exists a principal quasi-torus $T_2\subseteq G_2$ such that $\vk(T_1)\subseteq T_2$.
\end{definition}

Note that in Definition \ref{d:principal-intro}, instead of ``for some principal quasi-torus'',
we may write ``for every principal quasi-torus''; see Lemma \ref{l:some-any}.

When $G_1$ is connected, every homomorphism $\vk\colon G_1\to G_2$ is principal, whereas
the homomorphism $\iota\colon B\into G$ of Example \ref{ex:intro} (with non-connected $B$) is not principal.

\begin{theorem}[Corollary \ref{c:principal-induces}]
\label{t:principal-intro}
A {\emm principal} homomorphism $\vk\colon G_1\to G_2$ of quasi-connected reductive $k$-groups
preserves the canonical braiding, that is, for every $k$-algebra $R$ and
for all  $g,g'\in G_1(R)$  we have
\[ \big\{\vk(g),\vk(g')\big\}_{G_2} = \vk^\ssc\big(\big\{g,g'\big\}_{G_1}\big). \]
\end{theorem}

\begin{corollary}[Theorem \ref{t:functor}]
\label{c:principal-intro}
For a {\emm principal} homomorphism $\vk\colon G_1\to G_2$\hs,
the induced map
\[\vk_\ab\coloneqq\vk_\crr\colon\, H^1_\ab(k,G_1)\to  H^1_\ab(k,G_2)\]
is a homomorphism of abelian groups.
\end{corollary}

We see that the assignment $G\rightsquigarrow H^1_\ab(k,G)$ naturally extends to a functor
from the category of quasi-connected reductive $k$-groups with {\em principal} homomorphisms
to the category of abelian groups.

\begin{corollary}[from Corollaries \ref{c:local-global-abel}(ii) and \ref{c:principal-intro}]
For a quasi-connected reductive group $G$ over a field $k$ as in Corollary \ref{c:local-global-abel}(ii),
the first Galois cohomology set $H^1(k,G)$ has a canonical abelian group structure,
which is functorial with respect to principal homomorphisms.
\end{corollary}

\subsection*{Using principal quasi-tori}
For a quasi-connected reductive group $G$, let $T^\sscp\subset G^\ssc$ be a maximal torus,
and let $T=Z(G)\cdot\rho(T^\sscp)\subseteq G$
be the corresponding principal quasi-torus.
We consider the complex of quasi-tori $(T^\sscp\labelt\rho T)$ in degrees $-1$ and $0$
and its first hyper-cohomology group $H^1(k,T^\sscp\labelt\rho T)$.
Moreover, we consider the natural morphism of Picard crossed modules
\begin{equation*}
(T^\sscp\labelt\rho T, \theta_\triv, \br_\triv)\,\lra\, (G^\ssc\labelt\rho G, \theta,\br)
\end{equation*}
(where the action $\theta_\triv$ and the braiding $\br_\triv$ are trivial),
which induces a homomorphism of abelian groups
\begin{equation}\label{e:H1-T-ab}
H^1(k,T^\sscp\labelt\rho T)\to H^1_\ab(k,G).
\end{equation}
We prove the following theorem:

\begin{theorem}[Theorem \ref{t:f-iso}]
\label{t:f-iso-intro}
For a quasi-connected reductive $k$-group $G$, a maximal torus $T^\sscp\subset G^\ssc$,
and the principal quasi-torus $T=Z(G)\cdot\rho(T^\sscp)$,
the homomorphism \eqref{e:H1-T-ab} is an isomorphism of abelian groups,
which is functorial with respect to {\em principal} homomorphisms $G_1\to G_2$.
\end{theorem}

Here functoriality means the following. Let $\vk\colon G_1\to G_2$ be a principal homomorphism
of quasi-connected reductive $k$-groups.
Let $T_1\subseteq G_1$ be a principal quasi-torus,
and let $T_2\subseteq G_2$ be a principal quasi-torus such that $\vk(T_1)\subseteq T_2$
(then $\vk(T_1^\sscp)\subseteq T_2^\sscp$).
Then we have a commutative diagram of abelian groups
\begin{equation*}
\begin{aligned}
\xymatrix@C=15mm{
H^1(k, T_1^\sscp\to T_1)\ar[r]^-{\vk_*}\ar[d]_-\sim   &H^1(k, T_2^\sscp\to T_2)\ar[d]^-\sim\\
H^1_\ab (k,G_1)\ar[r]^-{\vk_\ab}             &H^1_\ab(k,G_2).
}
\end{aligned}
\end{equation*}
in which the vertical arrows are isomorphisms.

\subsection*{Idea of proof of Theorem \ref{t:principal-intro}}
Following \cite{BGA}, for a quasi-connected reductive $k$-group $G$,
we introduce the notion of  {\em $t$-extension
of $G$}  to be a short exact sequence
\[ 1\to S\to H\to G\to 1\]
where $S$ is a $k$-torus and $H$ is a quasi-connected reductive $k$-group
for which $H^\sss$ is simply connected.
It follows easily from \cite[Proposition 2.2]{BGA}
that every quasi-connected reductive $k$-group admits a $t$-extension.
Moreover, following \cite{BGA},
for  a homomorphism of quasi-connected reductive $k$-groups
$\vk\colon G\to G'$,
we introduce the notion of {\em $t$-extension of $\vk$}
to be a commutative diagram
\begin{equation*}
\begin{aligned}
\xymatrix@R=5mm{
H\ar[d]\ar[r] & H'\ar[d]\\
G\ar[r]^-\vk   & G'
}
\end{aligned}
\end{equation*}
in which the vertical arrows $H\to G$ and $H'\to G'$ are $t$-extensions.
In general, a homomorphism $\vk\colon G\to G'$ may not admit a $t$-extension.
For instance, the homomorphism $B\into G$ of Example \ref{ex:intro} does not admit a $t$-extension;
see Remark \ref{r:not-admit-t}.
In Section \ref{s:t-ext-principal} we prove:

\begin{theorem}[Theorem \ref{t:t-ext-principal}, not easy]
\label{t:t-ext-principal-intro}
Every {\emm principal} homomorphism of quasi-connected reductive $k$-groups  admits a $t$-extension.
\end{theorem}

Using Theorem \ref{t:t-ext-principal-intro}, in Section \ref{s:Picard}
we easily prove Theorem \ref{t:principal-intro} and Corollary \ref{c:principal-intro}.

\subsection*{Plan of the paper}
The plan of the rest of the paper is as follows.
In Section \ref{s:crossed} we construct the  crossed module $\CrM(G)$ from a reductive $k$-group $G$.
In Section \ref{s:coboundary},
using second non-abelian cohomology,
we  extend a hyper-cohomology exact sequence
corresponding to a crossed module of $\Gamma$-groups where $\Gamma$ is a profinite group,
and thus generalize results of Borel and Serre in \cite[Proposition 41 in Section I.5.6]{Serre}
and of Springer \cite[Proposition 1.28]{Springer}.
In Section \ref{s:Galois} we consider a version of these results for a crossed module of $k$-groups.
In Section \ref{s:Douai} we introduce the notion of a field of Galois-Douai type
and prove that for a simply connected semisimple group $G$ over such  a field $k$,
all elements of $H^2(k,G)$ are neutral.
In Section \ref{s:cr-cohomology} we define the crossed cohomology $H^1_\crr(k,G)$
and prove Theorems \ref{t:cr-surjective} and \ref{t:cr-bijective},
which together constitute Theorem \ref{t:local-global-crossing}.
Starting from Section \ref{s:quasi-connected}, we consider quasi-connected reductive $k$-groups
and principal homomorphisms.
Since $Z(G)$ may be non-smooth, we work with $R$-points rather than with $\kbar$-points,
where $R$ runs over commutative unital $k$-algebras.
In Section \ref{s:Picard} we construct the canonical Picard braiding on
the crossed module $\CrM(G)$ for a quasi-connected reductive $k$-group $G$,
and prove modulo Theorem \ref{t:t-ext-principal-intro}
that a principal homomorphism preserves the canonical braiding.
In Section \ref{s:t-ext-principal} we prove Theorem \ref{t:t-ext-principal-intro}.
In Section \ref{s:Picard-to-abelian} we consider the abelian Galois cohomology group $H^1_\ab(k,G)$
and prove Theorem \ref{t:f-iso-intro}.
We also give an example of computing $H^1(k,G)$
over a non-archimedean local field $k$ of positive characteristic,
using Theorem \ref{t:f-iso-intro}.

The paper contains two appendices.
In Appendix \ref{app:universal} we gather results concerning
the functoriality of the universal covering of a semisimple $S$-group scheme
over a non-empty base scheme $S$.
In Appendix \ref{app:gr-coh-Picard}, following \cite{Noohi}, we construct the abelian group structure
on the first group cohomology with coefficients in a Picard braided crossed module.

{\sc Acknowledgements.}
We thank  Remy van Dobben de Bruyn, Cristian D. Gonz\'alez-Avil\'es, and Derek Holt
for answering the first-named author's questions
in MathOverflow and Mathematics Stack Exchange.
We are grateful to Cristian D. Gonz\'alez-Avil\'es,  Jean-Pierre Labesse,
Nguyêñ Quôć Thǎńg, and Zev Rosengarten
for helpful email correspondence.
We thank the anonymous referee for careful reading and most inspiring comments,
which helped us to improve the paper.
The first-named author worked on this paper during his visit
to the Max Planck Institute for Mathematics, Bonn, in January 2026,
and he thanks the institute for hospitality, support, and excellent working conditions.
\bigskip\bigskip

\section{A crossed module from a non-connected reductive group}
\label{s:crossed}

In this section, $G$ is a (smooth) reductive group, {\em not necessarily connected}, over a field $k$.
We use the notations \,$G^0$, \,$G^\sss=[G^0,G^0]$, \,$\pi\colon G^\ssc\onto G^\sss$, \,and $\rho\colon G^\ssc\to G$ of the Introduction.

Let $\Set$ denote the category of sets, $\Grp$ denote the category of groups,
and $\Alg_R$ denote the category of
commutative unital $R$-algebras.
In particular, by  $\Alg_k$ we denote the category of
commutative unital $k$-algebras.
We denote by $\ul G$ the group functor represented by $G$:
\[\ul G\colon \Alg_k\to \Grp,\quad\ R\functor G(R)\quad\ \text{where $R$ is a $k$-algebra.}\]
For a $k$-group $H$, we consider the automorphism group functor
\[\ul\Aut(H)\colon \Alg_k\to \Grp,\quad\ R\functor \Aut_R(H_R),\]
where $H_R=H\times_k R$ denotes the $R$-group scheme obtained from the $k$-group $H$ by base change.
Observe that we have a homomorphism
\[ \ul\Aut(H)(R)=\Aut_R(H_R)\,\lra\, \Aut\,H_R(R)=\Aut\, H(R),\quad \ul a \mapsto a\ \ \text{for $\ul a\in\ul\Aut(H)(R)$,}\]
whence we obtain an action
\[\ul\Aut(H)(R)\times H(R)\to H(R), \quad\ (\ul a,h)\mapsto  a(h).\]

\begin{construction}\label{cons:theta}
We define an action $\theta^\ssc \colon G \times_k G^\ssc \to G^\ssc$.
It suffices to construct a homomorphism of group functors on the category $\Alg_k$ of commutative unital $k$-algebras
\[
\eta \colon \ul G \to \underline{\Aut}(G^\ssc).
\]
We construct $\eta$ by composing several homomorphisms of group functors.
For a $k$-algebra $R$, let $G_R^\sss=G^\sss\times_k R$ and $G_R^\ssc=G^\ssc\times_k R$
be the associated semisimple and simply connected group schemes over $R$, respectively;
see Definition \ref{d:red_ss_S-grp} and Definition \ref{d:sc} for these notions over $R$.
Furthermore, the induced morphism $G^\ssc_{R} \to G^\sss_{R}$ is a universal covering over $R$;
see Definition \ref{d:universal} and Lemma \ref{l:uni_cov_bc}.
By Proposition \ref{p:GA}(i), there is a homomorphism of groups
\[
{\Aut(G_R^\sss)} \to {\Aut(G_R^\ssc)}
\]
which is compatible with any base change $R \to R'$.
This shows that we have a homomorphism of group functors
\[
\underline{\Aut}(G^\sss) \xrightarrow{} \underline{\Aut}(G^\ssc).
\]
By \cite[Proposition 1.52, Corollary 6.19(e)]{Milne-AG},
the algebraic $k$-subgroup $G^\sss$ of $G$ is characteristic,
that is, $G^\sss_R$ is preserved by every automorphism of $G_R$ for every $k$-algebra $R$;
see \cite[Definition 1.51(b)]{Milne-AG}.
This yields a restriction homomorphism of (abstract) groups
\[
{\Aut_R(G_R)} \to {\Aut_R(G_R^\sss)}
\]
which is clearly compatible with any base change $R \to R'$.
Thus we obtain the restriction homomorphism of group functors
\[
\ul\Aut(G) \to \ul\Aut(G^\sss).
\]

By \cite[Expos\'e XXIV, \S 1.1]{SGA3}, we have a homomorphism of group functors
\[
\ul G \to \underline{\Aut}(G), \quad\ g\in G(R)=G_R(R)\, \longmapsto\ \inn(g)\in \Aut_R(G_R)=\ul\Aut(G)(R).
\]
Here $\inn(g)$ acts on the functor $\Alg_R\to \Grp$ represented by $G_R$ by
\[g'\mapsto g\hs g'g^{-1}\colon\, G_R(R')\to G_R(R') \quad\ \text{for every $R$-algebra $R'$ and every element $g'\in G_R(R')$}.\]
By composing all these homomorphisms, we obtain the desired homomorphism of group functors
\[
\ul \eta \colon \ul G \to \underline{\Aut}(G) \to \underline{\Aut}(G^\sss) \to \underline{\Aut}(G^\ssc),
\]
and thus a morphism of functors  $\Alg_k\to \Set$:
\[ \ul\theta^\ssc\colon \ul G\times \ul G^\ssc\to \ul G^\ssc,\quad(g,s)\mapsto \eta(g)(s)
    \quad\text{for}\ \, g\in G(R),\, s\in G^\ssc(R).\]
Applying the Yoneda lemma, we obtain the desired morphism of affine $k$-varieties
\[ \theta^\ssc\colon  G\times_k G^\ssc\to G^\ssc.\]
We  denote $\theta^{\ssc}_g (s)\coloneqq \theta^\ssc(g,s)$. We write
\[^g\hm s=\theta^\ssc_g(s)=\theta^\ssc(g,s).\]
\end{construction}

\begin{proposition}\label{p:crossed-module}
The quadruple $(G^\ssc,G,\rho,\theta^\ssc)$ is a {\emm crossed module of $k$-groups},
that is, for every $k$-algebra $R$ and for all $s,s'\in G^\ssc(R)$, $g\in G(R)$, we have
\begin{align}
^{\rho(s)}\hm s'  &= s\hs s' s^{-1},\tag{\sf CM1}\label{e:CM1}\\
\rho(\hs^g\hm s') &= g\rho(s') g^{-1}. \tag{\sf CM2}\label{e:CM2}
\end{align}
\end{proposition}

\begin{proof}
We must prove the equalities
\begin{align*}
\theta^\ssc\circ\rho&=\inn\colon\,  \ul G^\ssc\to \ul\Aut(G^\ssc), \tag{\sf CM1'}\label{e:CM1'}\\
\rho\circ\theta_g^\ssc&=\inn_g\circ\rho\colon\, G_R^\ssc\to  G_R\hs.  \tag{\sf CM2'}\label{e:CM2'}
\end{align*}
They follow from the commutativity of the following two diagrams:
\[
\xymatrix@C=20mm@R=7mm{
\ul G^\ssc\ar[r]^-{\inn}\ar@{->>}[d]_(.42)\pi  &\ul\Aut(G^\ssc)\ar@{=}[d]
      &G_R^\ssc\ar[r]^-{\theta^\ssc_g}\ar@{->>}[d]_(.42)\pi &G_R^\ssc\ar@{->>}[d]^(.42)\pi \\
\ul G_{\phantom{R}}^\sss\ar[r]^-{\theta^\ssc}\ar@{_(->}[d]  &\ul\Aut(G^\ssc)\ar@{=}[d]
      &G_R^\sss\ar[r]^-{\inn(g)}  \ar@{_(->}[d] &G_R^\sss\ar@{_(->}[d]\\
\ul G\ar[r]^-{\theta^\ssc}  &\ul\Aut(G^\ssc)      &G_R\ar[r]^-{\inn(g)} &G_R.
}
\]
In each of these two diagrams, the commutativity of the bottom rectangle is obvious,
and the commutativity of the top rectangle follows from the definition of the map $\theta^\ssc$.
\end{proof}

\begin{definition}
We write
$\CrM(G) = (G^\ssc,G, \rho,\theta)$ where we write $\theta$ for $\theta^\ssc$,
and we say that $\CrM(G)$ is
the {\em crossed module obtained from $G$}.
\end{definition}

Let $\varphi\colon G_1\to G_2$ be a homomorphism
of (not necessarily connected) reductive $k$-groups.
Using Proposition \ref{p:GA}(i), we obtain a commutative diagram
\[
\xymatrix@C=18mm{
G_1^\ssc\ar[r]^-{\varphi^\ssc}\ar[d]_-{\rho_1}  &G_2^\ssc\ar[d]^-{\rho_2} \\
G_1\ar[r]^-{\varphi}                            &G_2\, .
}
\]
\begin{proposition}\label{p:functor-CrM}
The pair
\[\CrM(\varphi)\coloneqq (\varphi^\ssc,\varphi)\colon\hs \CrM(G_1)\to \CrM(G_2)\]
is a morphism of crossed modules, that is, for every $k$-algebra $R$ we have
\begin{equation}\label{e:morphism-cr}
\varphi^\ssc(\hs^g\hm s)=\hs^{\varphi(g)}\hm \varphi^\ssc(s)\quad\ \text{for all}\ \,g\in G(R),\ s\in G^\ssc(R).
\end{equation}
\end{proposition}

\begin{proof}
Since $\varphi$ is a homomorphism, for every $g\in G_1(R)$ we have  commutative diagrams
\[
\xymatrix@C=20mm{
 G_1(R)\ar[r]^-{\inn(g)} \ar[d]_-\varphi & G_1(R)\ar[d]^-\varphi
    & G_1^\sss(R)\ar[r]^-{\inn(g)} \ar[d]_-{\varphi^\sss} & G_1^\sss(R)\ar[d]^-{\varphi^\sss} \\
 G_2(R)\ar[r]^-{\inn\,\varphi(g)}        & G_2(R)
    & G_2^\sss(R)\ar[r]^-{\inn(\varphi(g))}
    & G_2^\sss(R)
}
\]
where $\varphi^\sss$ is the restriction of $\varphi$ from $G_1$ to $G_1^\sss$.
By Proposition \ref{p:GA}(ii), we obtain from the diagram at right above that the following diagram commutes:
\begin{equation}\label{e:theta-varphi(g)}
\begin{aligned}
\xymatrix@C=18mm{
 G_1^\ssc(R)\ar[r]^-{\theta^\ssc_g} \ar[d]_-{\varphi^\ssc} & G_1^\ssc(R)\ar[d]^-{\varphi^\ssc} \\
 G_2^\ssc(R)\ar[r]^-{\theta^\ssc_{\varphi(g)}}             & G_2^\ssc(R)
}
\end{aligned}
\end{equation}
Diagram \eqref{e:theta-varphi(g)} means that \eqref{e:morphism-cr} holds, as desired.
\end{proof}

\begin{proposition}
Consider homomorphisms of (smooth) reductive $k$-groups, not necessarily connected:
\[ G_1\labelto{\varphi_{12}} G_2\labelto{\varphi_{23}} G_3\hs.\]
Then
\[\CrM(\varphi_{23}\circ\varphi_{12}) =
   \CrM(\varphi_{23})\circ\CrM(\varphi_{12})\hs\colon\, \CrM(G_1)\to \CrM(G_3) . \]
\end{proposition}

\begin{proof}
The proposition can be shown using Proposition \ref{p:GA}(ii).
\end{proof}

\section{Coboundary of a hyper-cocycle for a crossed module of \texorpdfstring{$\Gamma$}{Gamma}-groups}
\label{s:coboundary}

This section contains extensive calculations.
We recommend skipping most of it on a first reading.

In this section we work with crossed modules endowed with an action of a profinite group.
Throughout the section, we denote by $\operatorname{Maps}(-, -)$
the set of locally constant maps from a profinite group to a group.

Let $\Gamma$ be a profinite group, and
    let $(A\labelt{\rho} G,\theta)$ be a crossed module of $\Gamma$-groups, where $\theta \colon G \to \Aut(A)$ corresponds to the $G$-action on $A$.
    Consider the set of $0$-(hyper-)cochains
    $$C^0(\Gamma,A \to G) \coloneqq  \operatorname{Maps}(\Gamma,A) \times G.$$
    Following \cite[Section 3.3.1]{Borovoi-Memoir}, it has a group structure defined by
    \begin{equation}\label{e:grp_C0_Bor}
    (\varphi^1,g_1) \cdot (\varphi^2,g_2)\coloneqq  (\varphi',g_1 g_2)
    \quad \textit{where} \quad \varphi'_\sigma = {\ha}^{g_1} \varphi^2_\sigma \cdot \varphi^1_\sigma
    \end{equation}
    for $(\varphi^1,g_1), (\varphi^2,g_2) \in C^0(\Gamma,A \to G)$.
    Consider the set of 1-(hyper-)cochains
    \[ C^1(\Gamma,A \to G) \coloneqq  \operatorname{Maps}(\Gamma \times \Gamma,A) \times \operatorname{Maps}(\Gamma,G).\]
    The set of 1-cocycles $Z^1(\Gamma,A \to G)$ is the subset of $C^1(\Gamma,A \to G)$
    consisting of the elements $(u,\psi) \in \operatorname{Maps}(\Gamma \times \Gamma,A) \times \operatorname{Maps}(\Gamma,G)$ satisfying
    \begin{align}
        \rho(u_{\sigma,\tau}) \cdot \psi_\sigma \cdot {\ha}^\upsig \psi_\tau
        &=
        \psi_{\sigma\tau}\label{eq:def1_1-cocycle} \\
        u_{\sigma,\tau \nu} \cdot {\ha}^{\psi_\sigma \sigma}u_{\tau,\nu}
        &= u_{\sigma\tau,\nu} \cdot u_{\sigma,\tau}\label{eq:def2_1-cocycle}
    \end{align}
    for all $\sigma, \tau, \nu \in \Gamma$.
    We define a right action of $C^0(\Gamma,A \to G)$ on $Z^1(\Gamma,A \to G)$ by
    \begin{align*}
        (u,\psi) * (\varphi,g) \coloneqq  (u',\psi')
    \end{align*}
    where
    \begin{align}
        u'_{\sigma,\tau}
        &= {\ha}^{g^{-1}}(\varphi_{\sigma \tau} \cdot u_{\sigma,\tau} \cdot {\ha}^{\psi_\sigma \sigma} \varphi_\tau^{-1} \cdot \varphi_\sigma^{-1}) \label{eq:action_C0-1}, \\
        \psi'_\sigma &= g^{-1} \cdot \rho(\varphi_\sigma)
        \cdot
        \psi_\sigma
        \cdot
        {\ha}^\upsig\hmm g \label{eq:action_C0-2}.
    \end{align}
    This action is well defined; in particular, the pair $(u',\psi')$ lies in $Z^1(\Gamma, A\to G)$.
\begin{definition}[{\cite[Section 3.3]{Borovoi-Memoir}}]
    The pointed set $H^1(\Gamma,A \to G)$ is the set of orbits of $Z^1(\Gamma,A \to G)$ under this $C^0(\Gamma,A \to G)$-action,
    with the distinguished unit element $[1,1]$.
\end{definition}

\begin{remark}
For a crossed module of $\Gamma$-groups $(A\to G)$ as above,
we can write an element of $C^0(\Gamma, A\to G)$ as $(\tilde\varphi, g)$ with $\tilde\varphi_\sigma=\varphi_\sigma^{-1}$.
Then formula \eqref{e:grp_C0_Bor} changes to a nicer formula
\begin{equation*}
    (\tilde\varphi^1,g_1) \cdot (\tilde\varphi^2,g_2)\coloneqq  (\tilde\varphi',g_1 g_2)
    \quad \textit{where} \quad \tilde\varphi'_\sigma = \tilde\varphi^1_\sigma  \cdot{\ha}^{g_1}\hm\tilde \varphi^2_\sigma\hs,
\end{equation*}
and formulas \eqref{eq:action_C0-1} and \eqref{eq:action_C0-2} will change to
\begin{align*}
        u'_{\sigma,\tau}
        &= {\ha}^{g^{-1}}(\tilde\varphi_{\sigma \tau}^{-1} \cdot u_{\sigma,\tau} \cdot {\ha}^{\psi_\sigma \sigma} \tilde\varphi_\tau \cdot \tilde\varphi_\sigma) \\
        \psi'_\sigma &= g^{-1} \cdot \rho(\tilde\varphi_\sigma^{-1})
        \cdot
        \psi_\sigma
        \cdot
        {\ha}^\upsig\hmm g \label{eq:action_C0-2}.
    \end{align*}
Now it follows from the formulas of \cite[Appendix A]{Borovoi-16-arXiv} that indeed $(u',\psi')\in Z^1(\Gamma, A\to G)$.
\end{remark}

    Consider the short exact sequence of complexes of $\Gamma$-groups
    \[
    1\to(1 \to G) \labelt{i} (A \labelt{\rho} G) \labelt{ } (A \to 1)\to 1
    \]
    (where $i$ is a morphism of crossed modules), and let
    \[
        H^1(\Gamma,A) \labelto{\rho_*} H^1(\Gamma,G) \labelto{i_*} H^1(\Gamma, A \to G)
    \]
    be the induced hyper-cohomology exact sequence; see \cite[Corollary 3.4.3]{Borovoi-Memoir}.
    We extend this exact sequence to the right by introducing non-abelian $H^2$;
    see Corollary \ref{c:extend_ext_seq}(2) below.

    We say that a hyper-cocycle $(u,\psi)\in Z^1(\Gamma, A\to G)$ is {\em $2$-neutral} if $u=1$.
    We say that a class $[u,\psi]\in H^1(\Gamma, A\to G)$ is $2$-neutral
    if it is represented by a $2$-neutral hyper-cocycle.

\begin{lemma}\label{l:2-neutral-1}
A class $[u,\psi]\in H^1(\Gamma, A\to G)$ is $2$-neutral if and only if
there exists $w\in {\rm Maps}(\Gamma, A)$ such that
\begin{equation}\label{e:2-neutral-1}
w_{\sigma \tau} \cdot u_{\sigma,\tau} \cdot {\ha}^{\psi_\sigma \sigma} w_\tau^{-1} \cdot w_\sigma^{-1}=1
   \quad\ \text{for all}\ \,\sigma,\tau\in \Gamma.
\end{equation}
\end{lemma}

\begin{proof}
If $[u,\psi]$ is 2-neutral, then there exists a pair $(w,g)\in {\rm Maps}(\Gamma,A) \times G$
such that $(u,\psi) * (w,g) = (1,\psi')$ for some $\psi'\in {\rm Maps}(\Gamma, G)$.
Then by \eqref{eq:action_C0-1} we obtain that
\[{\ha}^{g^{-1}}\big(w_{\sigma \tau} \cdot u_{\sigma,\tau} \cdot
   {\ha}^{\psi_\sigma \sigma} w_\tau^{-1} \cdot w_\sigma^{-1}\big)
   =1\quad\ \text{for all}\ \,\sigma,\tau\in \Gamma,\]
and hence  \eqref{e:2-neutral-1} holds.
Conversely, assume that there exists $w\in {\rm Maps}(\Gamma, A)$
such that \eqref{e:2-neutral-1} holds.
Then
\begin{equation*}
(u,\psi)*(w,1) =(1,\psi')\quad\ \text{for some}\ \, \psi'\in {\rm Maps}(\Gamma, G),
\end{equation*}
and hence $[u,\psi]=[1,\psi']$ is 2-neutral, as desired.
\end{proof}

\begin{lemma}\label{l:2-neutral-2}
A class $[u,\psi]\in H^1(\Gamma, A\to G)$ is $2$-neutral if and only if
it lies in the image of $i_*\colon H^1(\Gamma, G)\to H^1(\Gamma, A\to G)$.
\end{lemma}

\begin{proof}
Let $[c]\in H^1(\Gamma, G)$, where $c\in Z^1(\Gamma, G)$.
Then we have $i_*[c]=[1,c]$, which is clearly $2$-neutral.
Conversely, let $[u,\psi]\in H^1(\Gamma, A\to G)$ be a $2$-neutral class.
Then we have  $[u,\psi]=[1,\psi']$ for some $\psi'\in {\rm Maps}(\Gamma, G)$,
and since $(1,\psi')\in Z^1(\Gamma, A\to G)$, it follows from
\eqref{eq:def1_1-cocycle} that $\psi'$ is a 1-cocycle.
We obtain that $[u,\psi]=[1,\psi']=i_*[\psi']$, which completes the proof.
\end{proof}

Let $\Gamma$ be a profinite group and let $A$ be an abstract group (endowed with the discrete topology, without given $\Gamma$-action).
Consider a $\Gamma$-{\em band} ($\Gamma$-kernel), that is, a locally constant homomorphism
$\beta: \Gamma \to  \operatorname{Out}(A)$
where $\operatorname{Out}(A) \coloneqq  \operatorname{Aut}(A)/\operatorname{Inn}(A)$.
We say that a $\Gamma$-band $\beta \colon \Gamma \to \Out(A)$ is {\em trivial} if it lifts to a locally constant homomorphism $\Gamma \to \Aut(A)$.
Following \cite[Section 1.14]{Springer}, \cite[Section 1.5]{Borovoi-Duke}, and \cite[Section (1.17)]{FSS},
the set of 2-cocycles $Z^2(\Gamma,A,\beta)$ is the set of pairs $(u,f)$,
where $u \in \operatorname{Maps}(\Gamma\times \Gamma,A)$ and $f \in \operatorname{Maps}(\Gamma,\operatorname{Aut}(A))$
satisfying the following conditions:
\begin{align}
    \operatorname{inn}(u_{\sigma,\tau}) \circ f_\sigma \circ f_\tau
    &= f_{\sigma \tau}, \label{eq:def1_2-cocycle} \\
    u_{\sigma,\tau \nu} \cdot f_\sigma(u_{\tau,\nu})
    &= u_{\sigma \tau,\nu} \cdot u_{\sigma,\tau}, \label{eq:def2_2-cocycle} \\
    f_\sigma \!\!\!\! \mod \operatorname{Inn}(A)
    &= \beta_\sigma. \label{eq:def3_2-cocycle}
\end{align}
The group $\operatorname{Maps}(\Gamma,A)$ acts on $Z^2(\Gamma,A,\beta)$ on the left as follows:
\[
    w * (u,f) = (u',f')\quad\ \text{for}\ \,w\in{\rm Maps}(\Gamma, A),
\]
where
\begin{equation}\label{eq:coboundary_2-cycle}
    u'_{\sigma,\tau} = w_{\sigma \tau} \cdot u_{\sigma,\tau} \cdot f_{\sigma}(w_\tau)^{-1} \cdot w_{\sigma}^{-1}, \quad
    f'_\sigma = \operatorname{inn}(w_\sigma) \circ f_\sigma
\end{equation}
for all $\sigma, \tau \in \Gamma$.
This action is well defined; in particular, $(u',f')\in Z^2(\Gamma, A,\beta)$.

\begin{definition}\label{d:nonab-H2}
With the above assumptions and notation,
\[H^2(\Gamma,A,\beta) \coloneqq \operatorname{Maps}(\Gamma,A) \backslash Z^2(\Gamma,A,\beta).\]
\end{definition}

The second cohomology set $H^2(\Gamma,A,\beta)$ has a distinguished subset of {\em neutral elements}.
A 2-cocycle $(u,f)\in Z^2(\Gamma,A,\beta)$ is called neutral if $u=1$.
Then it follows from \eqref{eq:def1_2-cocycle} that $f$ is a homomorphism.
A cohomology class $[u,f]\in H^2(\Gamma, A,\beta)$ is called neutral if it is the class of a neutral 2-cocycle.
By \eqref{eq:coboundary_2-cycle},
a class $[u,f]\in H^2(\Gamma, A,\beta)$ is neutral if and only if there exists
$w\in \Maps(\Gamma, A)$ satisfying
\begin{equation}\label{e:B-neutral}
w_{\sigma \tau} \cdot u_{\sigma,\tau} \cdot f_{\sigma}(w_\tau)^{-1} \cdot w_{\sigma}^{-1}=1.
\end{equation}

    Let $(A \labelt{\rho} G,\theta)$ be a crossed module with $\Gamma$-action. Then
    for all $a \in A$, $g \in G$, and $\sigma \in \Gamma$
    we have
    \begin{align}
        \theta_{\rho(a)} &= \operatorname{inn}(a) \in \operatorname{Aut}(A), \label{eq:prop1_crossed_mod} \\
        \sigma \circ \theta_g &= \theta_{{\ha}^\upsig\hm g} \circ \sigma \in \operatorname{Aut}(A), \label{eq:prop2_crossed_mod}\\
        \rho({\ha}^\upsig\! a) &= {\ha}^\upsig(\rho(a)) \in G, \label{eq:prop3_crossed_mod} \\
        \rho({\ha}^\upg a) &= g \cdot \rho(a) \cdot g^{-1} \in G, \label{eq:prop4_crossed_mod}
    \end{align}
    where, by abuse of notation, $\sigma$ in
    \eqref{eq:prop2_crossed_mod} is identified with its image in $\Aut(A)$ under the $\Gamma$-action on $A$.
    See \cite[Section 3.2]{Borovoi-Memoir}.

    In Construction \ref{construction: H2} below we shall introduce the {\em relative cohomology set}
     $H^2(\Gamma,A \mathrel{\mathrm{rel}} G)$.
     For this end, we now consider certain $\Gamma$-bands associated with 1-cocycles in $Z^1(\Gamma,A \to G)$.

    The natural morphism of crossed modules $(A \labelto{\rho} G) \to (1 \to \coker{\rho})$ induces a map of 1-cocycles $Z^1(\Gamma, A \to G) \to Z^1(\Gamma, \coker \rho)$.
    Since by \eqref{eq:prop1_crossed_mod}, the subgroup $\rho(A)\subseteq G$, when acting on $A$ via $\theta$,
    acts on $A$ by inner automorphisms,  we have a well-defined homomorphism $\bar{\theta} \colon \coker \rho \to \Out(A)$ making the following diagram  with exact rows commutative:
    \begin{equation}\label{e:bar_theta}
        \begin{tikzcd}
         1 \ar[r]    &A \ar[r,"\rho"] \ar[d,"\inn"] &G \ar[d,"\theta"] \ar[r] &\coker \rho \ar[r] \ar[d,"\bar{\theta}"] &1
         \\
         1 \ar[r] &\Inn(A) \ar[r] &\Aut(A) \ar[r] &\Out(A) \ar[r] &1
        \end{tikzcd}.
    \end{equation}
    If we consider the $\Gamma$-action on $\Aut(A)$ defined by $\sigma \cdot \phi \coloneqq \sigma \circ \phi \circ \sigma^{-1}$
    for $\sigma\in\Gamma$, then by \eqref{eq:prop2_crossed_mod}
    the homomorphism $\theta$ in diagram \eqref{e:bar_theta} is $\Gamma$-equivariant.
    Furthermore, the induced $\Gamma$-action on $\Inn(A)$ by restriction is well defined,
    and  by \eqref{eq:prop1_crossed_mod} and \eqref{eq:prop3_crossed_mod},
    the homomorphism $\inn$ in  diagram \eqref{e:bar_theta} is $\Gamma$-equivariant.
    This implies that $\bar{\theta}$ is $\Gamma$-equivariant as well, that is,
    \begin{equation}\label{e:bar_theta_equivar}
        \sigma \circ \bar{\theta}_{\bar{g}} \circ \sigma^{-1} = \bar{\theta}_{{\ha}^\upsig\bar{g}}
        \textit{ in } \Out(A)
    \end{equation}
    for all $\bar{g} \in \coker \rho$ and $\sigma \in \Gamma$.

    For $\bar\psi \in Z^1(\Gamma,\coker \rho)$, the map $\beta \colon \Gamma \to \Out(A)$ defined by $\beta_\sigma \coloneqq \bar{\theta}_{\bar{\psi}_\sigma} \circ \sigma$ is a locally constant homomorphism, where by abuse of notation the rightmost $\sigma$ is identified with its image in $\Out(A)$.
    Indeed, we have
    \[
    \beta_{\sigma} \circ \beta_\tau
            = \bar{\theta}_{\bar{\psi}_\sigma} \circ \sigma \circ \bar{\theta}_{\bar{\psi}_\tau} \circ \tau
            = \bar{\theta}_{\bar{\psi}_\sigma} \circ \bar{\theta}_{{\ha}^\upsig\hm\bar{\psi}_\tau} \circ \sigma \circ \tau
            = \bar{\theta}_{\bar{\psi}_{\sigma\tau}} \circ \sigma \circ \tau
            =\beta_{\sigma \tau},
    \]
    where the second equality follows from \eqref{e:bar_theta_equivar} and the third equality follows from the fact that $\bar{\psi}$ is a 1-cocycle.
    This allows us to define a map
    \begin{equation}\label{e:def_bar_theta}
    \bar{\theta}_* \colon Z^1(\Gamma, \coker \rho)\, \to\, \Hom_{\hs\mathit{l.c.}}(\Gamma,\Out(A)), \quad\
    \bar{\psi} \mapsto (\sigma \mapsto \bar{\theta}_{\bar{\psi}_\sigma} \circ \sigma)
    \end{equation}
    where $\Hom_{\hs\mathit{l.c.}}(\Gamma,\Out(A))$ denotes the set of locally constant homomorphisms, which are precisely the $\Gamma$-bands for $A$.

    Let $\Phi$ denote the image of the composite map
    \begin{equation}\label{e:image_Z1}
        Z^1(\Gamma, A \to G) \to Z^1(\Gamma, \coker \rho) \labelto{\bar{\theta}_*} \Hom_{\hs\mathit{l.c.}}(\Gamma, \Out(A)).
    \end{equation}

    The following lemma gives an explicit description of the elements of $\Phi$.

    \begin{lemma}\label{l:form_band_Phi}
        A $\Gamma$-band $\beta\in \Hom_{\hs\mathit{l.c.}}(\Gamma,\Out(A))$
        is the image of $(u, \psi) \in Z^1(\Gamma, A \to G)$ under the composite map \eqref{e:image_Z1}
        if and only if $\beta = f_*$ for
    \begin{align}
    f&\colon \Gamma \to \Aut(A),
    \ \, \sigma \mapsto \theta_{\psi_\sigma} \circ \sigma,  \label{e:1-to-2-f}\\
     f_* &\colon \Gamma \labelto{f} \Aut(A) \to \Out(A),
     \label{e:1-to-2-beta}
    \end{align}
    where by abuse of notation the rightmost $\sigma$ in \eqref{e:1-to-2-f} is identified with its image in $\Aut(A)$.
    \end{lemma}

    \begin{proof}
        Let $(u, \psi) \in Z^1(\Gamma, A \to G)$ be an arbitrary 1-cocycle.
        The first map in the composition \eqref{e:image_Z1} sends $(u,\psi)$ to $\bar{\psi} \in Z^1(\Gamma, \coker \rho)$, where $\bar{\psi}_\sigma$ denotes the image of $\psi_\sigma$ in $\coker \rho$.
        The second map $\bar{\theta}_*$ then sends $\bar{\psi}$ to the $\Gamma$-band $\beta \in \Hom_{\hs\mathit{l.c.}}(\Gamma, \Out(A))$ where $\beta_\sigma = \bar{\theta}_{\bar{\psi}_\sigma} \circ \sigma$.
        By the commutativity of diagram \eqref{e:bar_theta}, $\bar{\theta}_{\bar{\psi}_\sigma}$ is precisely the image of $\theta_{\psi_\sigma}$ in $\Out(A)$. Therefore, we obtain
        \[
        \beta_\sigma = \bar{\theta}_{\bar{\psi}_\sigma} \circ \sigma
        = (\theta_{\psi_\sigma} \circ \sigma)\cdot\Inn(A)\in\Out(A).
        \]
        Setting $f_\sigma = \theta_{\psi_\sigma} \circ \sigma \,\in \Aut(A)$, we see that $\beta_\sigma$ is exactly the image of $f_\sigma$ in $\Out(A)$,  that is, $\beta = f_*$.
        This completes the proof, since $\Phi$ is defined to be the image of the composite map \eqref{e:image_Z1}.
    \end{proof}

    We wish to show that the composite map \eqref{e:image_Z1} induces a map on cohomology.
    Recall that $Z^1(\Gamma,\coker \rho)$ admits a right action of $\coker \rho$, defined by
    \begin{equation}\label{e:1-cocycle_action}
        (\bar{\psi} \cdot \bar{g})_\sigma = \bar{g}^{-1} \cdot \bar{\psi}_\sigma \cdot {\ha}^\upsig\bar{g}
        \quad\ \text{for}\ \, \bar{g} \in \coker \rho.
    \end{equation}
    The natural map
    \[
    Z^1(\Gamma,A\to G) \to Z^1(\Gamma,\coker\rho),
    \quad
    (u,\psi)\mapsto\bar\psi,
    \]
    is equivariant with respect to the action of $C^0(\Gamma,A\to G)$ on
    $Z^1(\Gamma,A\to G)$ (defined in \eqref{eq:action_C0-1} and
    \eqref{eq:action_C0-2}) and the action of $\coker\rho$ on
    $Z^1(\Gamma,\coker\rho)$ (defined in \eqref{e:1-cocycle_action}), via the homomorphism
    \[
    C^0(\Gamma,A\to G)\to\coker\rho,
    \quad
    (\varphi,g)\mapsto\bar g.
    \]
    Hence it descends to a well-defined map
    \[
    H^1(\Gamma,A\to G)\to H^1(\Gamma,\coker\rho).
    \]

    There is also a natural right action of $\coker\rho$ on the set of
    $\Gamma$-bands
    $\Hom_{\hs\mathit{l.c.}}(\Gamma,\Out(A))$,
    defined by
    \begin{equation}\label{e:act_on_bands}
    (\beta\cdot\bar g)_\sigma
    \coloneqq
    \bar\theta_{\bar g^{-1}}
    \circ
    \beta_\sigma
    \circ
    \bar\theta_{\bar g}
    \quad\text{in }\Out(A),
    \qquad
    \bar g\in\coker\rho.
    \end{equation}

    \begin{lemma}
        The map $\bar{\theta}_* \colon Z^1(\Gamma,\coker \rho) \to \Hom_{\hs\mathit{l.c.}}(\Gamma,\Out(A))$ is $\coker \rho$\,-equivariant.
    \end{lemma}
\begin{proof}
    Let $\bar{\psi} \in Z^1(\Gamma, \coker \rho)$ and $\bar{g} \in \coker \rho$.
    Evaluating $\bar{\theta}_*(\bar{\psi} \cdot \bar{g})$ at $\sigma \in \Gamma$, we obtain
    \begin{alignat*}{3}
        \bar{\theta}_*(\bar{\psi} \cdot \bar{g})_\sigma
        &= \bar{\theta}_{(\bar{\psi} \cdot \bar{g})_\sigma} \circ \sigma
        &&\quad \text{in } \Out(A)
        &&\quad \textit{by } \eqref{e:def_bar_theta} \\
        &= \bar{\theta}_{\bar{g}^{-1}} \circ \bar{\theta}_{\bar{\psi}_\sigma} \circ \bar{\theta}_{{\ha}^\upsig  \bar{g}} \circ \sigma
        &&\quad \text{in } \Out(A)
        &&\quad \textit{by } \eqref{e:1-cocycle_action} \\
        &= \bar{\theta}_{\bar{g}^{-1}} \circ \bar{\theta}_{\bar{\psi}_\sigma} \circ \sigma \circ \bar{\theta}_{\bar{g}}
        &&\quad \text{in } \Out(A)
        &&\quad \textit{by } \eqref{e:bar_theta_equivar} \\
        &= \bar{\theta}_{\bar{g}^{-1}} \circ \bar{\theta}_*(\bar{\psi})_\sigma \circ \bar{\theta}_{\bar{g}}
        &&\quad \text{in } \Out(A)
        &&\quad \textit{by } \eqref{e:def_bar_theta} \\
        &= (\bar{\theta}_*(\bar{\psi}) \cdot \bar{g})_\sigma
        &&\quad \text{in } \Out(A)
        &&\quad \textit{by } \eqref{e:act_on_bands},
    \end{alignat*}
    which completes the proof.
\end{proof}

We see that the composite map \eqref{e:image_Z1} induces a composite map
\begin{equation}\label{e:image_H1}
H^1(\Gamma, A \to G) \to H^1(\Gamma, \coker \rho) \labelto{\bar{\theta}_*} \Hom_{\hs\mathit{l.c.}}(\Gamma,\Out(A))/\coker \rho,
\end{equation}
where $\Hom_{\hs\mathit{l.c.}}(\Gamma,\Out(A))/\coker \rho$ denotes the set of $\coker \rho$\,-orbits in $\Hom_{\hs\mathit{l.c.}}(\Gamma,\Out(A))$.
Consequently, the image of the composite map \eqref{e:image_H1} is canonically identified with the set of $\coker \rho$\,-orbits in $\Phi$.

\begin{lemma} \label{l:Z2(Gamma-A-beta)}
Assume that $(u,\psi) \in Z^1(\Gamma,A \to G)$.
Let $\beta \in \Phi$ be the image of $(u,\psi)$ under the composite map \eqref{e:image_Z1}, with explicit description
\begin{gather*}
    f \colon \Gamma \to \Aut(A),
    \quad \sigma \mapsto \theta_{\psi_\sigma} \circ \sigma,  \\
     \beta = f_* \colon \Gamma \labelto{f} \Aut(A) \to \Out(A)
\end{gather*}
due to Lemma \ref{l:form_band_Phi}, where by abuse of notation the rightmost $\sigma$ is identified with its image in $\Aut(A)$.
Then we have $(u,f) \in Z^2(\Gamma,A,\beta)$.
\end{lemma}

\begin{proof}
    We must show that $f$ satisfies the three conditions \eqref{eq:def1_2-cocycle}--\eqref{eq:def3_2-cocycle}.
    Condition \eqref{eq:def1_2-cocycle} is satisfied as follows:
    \begin{align*}
        \operatorname{inn}(u_{\sigma,\tau}) \circ f_\sigma \circ f_\tau
        &= \theta_{\rho(u_{\sigma,\tau})} \circ f_\sigma \circ f_\tau  && \textit{by \eqref{eq:prop1_crossed_mod}}\\
        &= \theta_{\rho(u_{\sigma,\tau})} \circ \theta_{\psi_{\sigma}} \circ (\sigma \circ \theta_{\psi_{\tau}}) \circ \tau &&\textit{by \eqref{e:1-to-2-f}}\\
        &= \big(\theta_{\rho(u_{\sigma,\tau})} \circ \theta_{\psi_{\sigma}} \circ \theta_{{\ha}^\upsig \psi_{\tau}}\big) \circ \sigma \circ \tau &&\textit{by \eqref{eq:prop2_crossed_mod}} \\
        &= \theta_{\psi_{\sigma  \tau}} \circ \sigma \circ \tau &&\textit{by \eqref{eq:def1_1-cocycle}} \\
        &= f_{\sigma\tau}\hs.
    \end{align*}
    Condition \eqref{eq:def2_2-cocycle} follows from \eqref{eq:def2_1-cocycle} directly:
    \begin{align*}
        u_{\sigma,\tau \nu} \cdot f_{\sigma}(u_{\tau,\nu})
        = u_{\sigma,\tau \nu} \cdot {\ha}^{\psi_{\sigma} \sigma}u_{\tau,\nu}
        = u_{\sigma \tau,\nu} \cdot u_{\sigma,\tau}\hs.
    \end{align*}
    Condition \eqref{eq:def3_2-cocycle} follows from \eqref{e:1-to-2-beta}.
\end{proof}

Let $\beta$ be a $\Gamma$-band.
For $g \in G$ and a class $[u, f] \in H^2(\Gamma, A, \beta)$, we define
\begin{equation*}
    u'_{\sigma, \tau} \coloneqq {\ha}^{g^{-1}}\!u_{\sigma, \tau} \quad\ \text{and} \quad\ f'_\sigma \coloneqq \theta_{g^{-1}} \circ f_\sigma \circ \theta_g.
\end{equation*}
Then the band $\beta' \coloneqq f'_*\colon \Gamma \labelto{f'} \Aut(A) \to  \Out(A)$ is conjugate to $\beta = f_*$ by $\bar{\theta}_{\bar{g}^{-1}} \in \Out(A)$,
hence $\beta'$ and $\beta$ lie in the same $\coker \rho$\,-orbit under the action \eqref{e:act_on_bands}.
Consequently, there is a bijection
\begin{equation}\label{e:act_H2}
    H^2(\Gamma,A,\beta) \isoto H^2(\Gamma,A,\beta'), \quad
    [u,f]  \mapsto [u',f'].
\end{equation}
It is clear that this bijection preserves the subsets of neutral elements.
Furthermore, if $\beta$ is contained in $\Phi$ (the image of the map in \eqref{e:image_Z1}), then so is $\beta'$, since $\Phi$ is closed under the $\coker \rho$\,-action (defined in \eqref{e:act_on_bands}).

\begin{construction}\label{construction: H2}
    By pulling back the $\coker \rho$\,-action on $\Phi$ to $G$ via the natural projection $G \to \coker \rho$,
    the bijection \eqref{e:act_H2} yields a right $G$-action on $\bigsqcup_{\beta \in \Phi} H^2(\Gamma,A,\beta)$ by $[u,f] \cdot g \coloneqq [u',f']$.
    We define
    \[
    H^2(\Gamma,A \mathrel{\mathrm{rel}} G) \coloneqq \bigsqcup_{\beta \in \Phi} H^2(\Gamma,A,\beta)/G,
    \]
    the set of $G$-orbits in $\bigsqcup_{\beta \in \Phi} H^2(\Gamma,A,\beta)$ under the given action.
    For $[u,f] \in H^2(\Gamma,A,\beta)$ with $\beta \in \Phi$, we denote by $[u,f]_G$ its image in $H^2(\Gamma,A \mathrel{\mathrm{rel}} G)$, that is, the $G$-orbit of $[u,f]$.

    We define a \textit{neutral element} in $H^2(\Gamma,A \mathrel{\mathrm{rel}} G)$ as a $G$-orbit containing a neutral element of $H^2(\Gamma,A,\beta)$ for some $\beta \in \Phi$.
    Since the bijection \eqref{e:act_H2} preserves the neutral elements, if an orbit $[u,f]_G$ is neutral, then every representative in this orbit is a neutral element in $H^2(\Gamma,A,\beta')$ for a suitable choice of $\beta'$.
\end{construction}

Note that Construction \ref{construction: H2} is similar to that of  \cite[Section 1.20]{Springer}.
For $(u,\psi) \in Z^1(\Gamma,A \to G)$, we set $\Delta(u,\psi) \coloneqq [u,f]\in H^2(\Gamma,A,\beta)$ for $(u,f)$ constructed in Lemma \ref{l:Z2(Gamma-A-beta)}.

\begin{lemma}\label{l:Delta_H2}
    The map
    \[
    \Delta \colon H^1(\Gamma,A \to G) \to H^2(\Gamma,A \mathrel{\mathrm{rel}} G), \quad
    [u,\psi] \mapsto \Delta(u,\psi)_G
    \]
    is well defined.
\end{lemma}
\begin{proof}
    We must show that $\Delta(u,\psi)_G$ is independent of the choice
    of a representative cocycle $(u,\psi)\in Z^1(\Gamma, A\to G)$
    for a given class $[u,\psi]\in H^1(\Gamma, A \to G)$.
    Since by \eqref{e:grp_C0_Bor} every $(\varphi, g) \in C^0(\Gamma, A \to G)$
    can be decomposed as $(\varphi, g) = (\varphi, 1) \cdot (1, g)$, it suffices to examine the action of $(\varphi, 1)$ and $(1, g)$ separately.

    Let $(u',\psi') = (u,\psi) * (\varphi,1)$ and let $[u,f] = \Delta(u,\psi) \in H^2(\Gamma,A,\beta)$ for $(u,f)$ in Lemma \ref{l:Z2(Gamma-A-beta)}.
    By \eqref{eq:action_C0-1}, we have
    \[
    u'_{\sigma, \tau} = \varphi_{\sigma \tau} \cdot u_{\sigma, \tau} \cdot f_\sigma(\varphi_\tau)^{-1} \cdot \varphi_\sigma^{-1}
    \]
    since $f_\sigma = \theta_{\psi_\sigma} \circ \sigma$ as defined in Lemma~\ref{l:Z2(Gamma-A-beta)}.
    From \eqref{eq:action_C0-2} and \eqref{eq:prop1_crossed_mod}, it follows that
    \[
    f'_\sigma = \theta_{\psi'_\sigma} \circ \sigma = \theta_{\rho(\varphi_\sigma)} \circ \theta_{\psi_\sigma} \circ \sigma = \inn(\varphi_\sigma) \circ f_\sigma.
    \]
    By \eqref{eq:coboundary_2-cycle}, we see that $(u', f')$ is a 2-cocycle with the same band $\beta$, and moreover, $[u', f'] = [u, f]$ in $H^2(\Gamma, A, \beta)$, hence in $H^2(\Gamma,A \mathrel{\mathrm{rel}} G)$ as well.

    Let $(u'',\psi'') = (u,\psi) * (1,g)$.
    By \eqref{eq:action_C0-1} and \eqref{eq:action_C0-2}, we have
    \[
    u''_{\sigma, \tau} = {\ha}^{g^{-1}}\!u_{\sigma, \tau} \quad\ \text{and} \quad\ \psi''_\sigma = \theta_{g^{-1}} \circ \psi_\sigma \circ \theta_{{\ha}^\upsig\hm g}.
    \]
    Take  $f''_\sigma = \theta_{\psi''_\sigma} \circ \sigma$. From \eqref{eq:prop2_crossed_mod} we obtain
    \[ f''_\sigma = \theta_{g^{-1}} \circ \theta_{\psi_\sigma} \circ \theta_{{\ha}^\upsig\hm g} \circ \sigma
    = \theta_{g^{-1}} \circ \theta_{\psi_\sigma} \circ \sigma \circ \theta_g
    = \theta_{g^{-1}} \circ f_\sigma \circ \theta_g\hs. \]
    This implies that $\Delta(u'', \psi'') = [u'', f'']$ lies in $H^2(\Gamma, A, \beta'')$ where $\beta'' = f''_*$ with the notation of Lemma \ref{l:Z2(Gamma-A-beta)},
    and satisfies $[u'',f''] = [u,f] \cdot g$. Hence, $[u'',f'']_G = [u,f]_G$ in $H^2(\Gamma,A \mathrel{\mathrm{rel}} G)$.
\end{proof}

\begin{proposition}\label{p:kang}
A class $[u,\psi] \in H^1(\Gamma,A \to G)$ lies in the image of $H^1(\Gamma,G)$
if and only if the class $\Delta(u,\psi) = [u,f] \in H^2(\Gamma,A,\beta)$ is neutral.
\end{proposition}

\begin{proof}
By Lemma \ref{l:2-neutral-2}, the class $[u,\psi]$ lies in the image of $H^1(\Gamma,G)$ if and only if it is $2$-neutral.
By Lemma \ref{l:2-neutral-1}, this is equivalent to the existence of $w\in\operatorname{Maps}(\Gamma,A)$
satisfying \eqref{e:2-neutral-1}.
Under the identification
$f_\sigma=\theta_{\psi_\sigma}\circ\sigma$
given by \eqref{e:1-to-2-f},
condition \eqref{e:2-neutral-1} is exactly  condition \eqref{e:B-neutral},
which completes the proof.
\end{proof}

Since every representative of a neutral orbit $\Delta(u,\psi)_G$ is itself neutral (see Construction \ref{construction: H2}), Proposition \ref{p:kang} immediately implies the following corollary:

\begin{corollary}\label{c:extend_ext_seq} \ \nopagebreak[4]
\begin{enumerate}
    \item A class $[u,\psi] \in H^1(\Gamma,A \to G)$ lies in the image of $H^1(\Gamma,G)$
    if and only if $\Delta(u,\psi)_G  \in H^2(\Gamma,A \mathrel{\mathrm{rel}} G)$ is neutral.
    \item There is an exact sequence
    \[
    H^1(\Gamma,A) \labelto{\rho_*} H^1(\Gamma,G) \labelto{i_*} H^1(\Gamma, A \to G) \labelto{\Delta} H^2(\Gamma,A \mathrel{\mathrm{rel}} G),
    \]
    where exactness at $H^1(\Gamma, A \to G)$ means that the image of $i_*$ coincides with the preimage of the subset of neutral elements in $H^2(\Gamma,A \mathrel{\mathrm{rel}} G)$ under $\Delta$.
\end{enumerate}
\end{corollary}

\begin{remark}
    The exact sequence in Corollary \ref{c:extend_ext_seq} extends those found in \cite[Corollary 3.4.3]{Borovoi-Memoir} and \cite[Proposition 11.5]{Noohi} by appending the term $H^2(\Gamma,A \mathrel{\mathrm{rel}} G)$ to the right, yielding the following long exact sequence:
\[
\begin{tikzcd}[row sep=1.0em, column sep=1.2em]
1 \arrow[r] & H^{-1}(\Gamma, A \to G) \arrow[r, ""] & H^0(\Gamma,A) \arrow[r, ""] & H^0(\Gamma,G) \arrow[r, ""] & H^0(\Gamma, A \to G) \arrow[dll, rounded corners=8pt, overlay, to path={
    -- ([xshift=2.5ex]\tikztostart.east)
    |- ($([xshift=2.5ex]\tikztostart.east)!0.5!([xshift=-2.5ex]\tikztotarget.west)$)
    -| ([xshift=-2.5ex]\tikztotarget.west)
    \tikztonodes -- (\tikztotarget)
}] \\
& & H^1(\Gamma,A) \arrow[r, "\rho_*"] & H^1(\Gamma,G) \arrow[r, "i_*"] & H^1(\Gamma, A \to G) \arrow[r, "\Delta"] & H^2(\Gamma,A \mathrel{\mathrm{rel}} G).
\end{tikzcd}
\]
Compare with the exact sequence in \cite[Theorem 4.2]{GA-12},
in which $(A\to G)$ is a {\em quasi-abelian} crossed module
(see Proposition \ref{p:q-ab} below for the definition), and the analog of our $H^1(\Gamma, A\to G)$ is an abelian group.
\end{remark}

\begin{remarks} \

\begin{enumerate}
\item[\rm (a)]Assume in Lemma \ref{l:Z2(Gamma-A-beta)} that  $\rho$ is injective.
We identify $A$ with $\rho(A)\subseteq G$. Then it follows from \eqref{eq:prop4_crossed_mod}
that $A$ is a normal subgroup of $G$.
Writing $C=G/A$, we have a short exact sequence of $\Gamma$-groups
\[ 1\to A\to G\to C\to 1,\]
and we identify $H^1(\Gamma, A\into G)$ with $H^1(\Gamma, C)$;
see \cite[Example 3.3.4(3)]{Borovoi-Memoir}.
Every $1$-cocycle $c\in Z^1(\Gamma, C)$ can be lifted
to a $1$-hyper-cocycle $(u,\psi)\in Z^1(\Gamma, A\into G)$,
and by Proposition \ref{p:kang} the cohomology class $[c]\in H^1(\Gamma, C)$
comes from $H^1(\Gamma,G)$ if and only if $\Delta(u,\psi)\in H^2(\Gamma, A,\beta)$ is neutral.
This result was used in the GAP program of \cite{BdG} computing $H^1(\R,G)$ for a real linear algebraic group $G$
(not necessarily connected or reductive).

\item[\rm (b)] Moreover, assume that the normal subgroup $A\subseteq G$ is abelian.
Then $c$ defines a twisted form $_c A=(A,\beta)$ of the $\Gamma$-group $A$,
and the second non-abelian cohomology set $H^2(\Gamma, A,\beta)$ in this case
can be identified with the second {\em abelian} cohomology group $H^2(\Gamma, \,_c A)$.
Furthermore, then $H^2(\Gamma, A,\beta)$  has exactly one neutral element,
which corresponds to $1\in H^2(\Gamma, \,_c A)$.
Under this identification, our  $\Delta(u,\psi)$ corresponds to the class inverse to
$\Delta(c)\in H^2(\Gamma, \,_c A)$ in Serre's book \cite[Section I.5.6]{Serre},
and our $\Delta(u,\psi)$ is neutral if and only if $\Delta(c)=1$.
Thus our Proposition \ref{p:kang} generalizes  \cite[Proposition 41 in Section I.5.6]{Serre},
which says that the class $[c]\in H^1(\Gamma, C)$ comes from $H^1(\Gamma,G)$
if and only if $\Delta(c)=1\in H^2(\Gamma, \,_c A)$.
\end{enumerate}
\end{remarks}

\section{Coboundary of a Galois hyper-cocycle for a crossed module of \texorpdfstring{$k$}{k}-groups}
\label{s:Galois}

In this section we apply the results of the previous section to crossed modules of $k$-groups.
We change our notation and denote by  $(A\labelt\rho G,\theta)$
a crossed module of smooth linear algebraic groups over a field $k$.
Let $k_s$ be a fixed separable closure of $k$,
and let $\Gamma=\Gal(k_s/k)$.
We consider the crossed module of $\Gamma$-groups
\[ (\cA\labelt{ } \cG)\coloneqq \big(A(k_s)\labelt{\rho}G(k_s)\big).\]
Write $\Abar\coloneqq A\times_k k_s$.
Let $\Aut(\Abar)$ denote the group of automorphisms of $\Abar$ and
let $\Inn(\Abar)$ denote the group of inner automorphisms of $\Abar$.
Let $\SAut(\Abar)$ denote the group of $k_s$-semilinear automorphisms of $\Abar$;
see \cite[Section 1.2]{Borovoi-Duke} or \cite[Section (1.2)]{FSS}.
We set $\Out(\Abar)\coloneqq\Aut(\Abar)/\Inn(\Abar)$ and $\SOut(\Abar)\coloneqq\SAut(\Abar)/\Inn(\Abar)$.
There are exact sequences
\begin{align}
    1 \to \Aut(\Abar) \to \SAut(\Abar) \to \Gamma, \label{e:SAut_exact_seq} \\
    1 \to \Out(\Abar) \to \SOut(\Abar) \to \Gamma. \label{e:SOut_exact_seq}
\end{align}
See \cite[Sections (1.2) and (1.5)]{FSS}.
Given the (homomorphic) splitting
$s \colon \Gamma \to \SAut(\Abar)$ of the extension \eqref{e:SAut_exact_seq}
corresponding to a $k$-form $A$ (see \cite[Section (1.4)]{FSS}), we note that
\begin{equation}\label{e:k-form_section_action}
^\upsig\! \phi = s_\sigma \circ \phi \circ s_\sigma^{-1}
\end{equation}
for every $\sigma \in \Gamma$ and every $\phi \in \Aut(\bar A)$.
We say that a set-theoretic section $f \colon \Gamma \to \SAut(\Abar)$ of \eqref{e:SAut_exact_seq}
is \textit{continuous} if, for every $\sigma \in \Gamma$, the map
\[
\Gamma \to \Aut(\Abar), \quad \tau \mapsto s_\tau^{-1} \circ f_{\sigma}^{-1} \circ f_{\sigma \tau}
\]
is locally constant; see \cite[Definition (1.10) and Proposition (1.7)(ii)]{FSS}.
A $k$-\textit{band} for $\Abar$ is defined to be a group homomorphism $\beta \colon \Gamma \to \SOut(\Abar)$ that splits \eqref{e:SOut_exact_seq} and lifts to a continuous map $f \colon \Gamma \to \SAut(\Abar)$; see \cite[Definition (1.11)]{FSS}.

Let
\[(u,\psi)\in Z^1(k,A\to G)\coloneqq Z^1(\Gamma, \cA\to\cG)\]
be a 1-cocycle.
Write
\[\tilde f_\sigma= \theta_{\psi_\sigma}\circ\sigma\in \SAut(\Abar). \]
Then $\tilde f_\sigma$ is induced from a $\sigma$-semilinear automorphism $f_\sigma \coloneqq \theta_{\psi_\sigma} \circ s_\sigma$
in the sense of \cite[Section (1.2)]{FSS}, and the map $f\colon \Gamma \to \SAut(\Abar)$ is a set-theoretic section of \eqref{e:SAut_exact_seq}.
We show that $f$ is continuous.
Indeed, we must prove that the map
$\Gamma \to \Aut(\Abar)$, $\tau \mapsto s_\tau^{-1} \circ f_\sigma^{-1} \circ f_{\sigma \tau}$ is locally constant.
By the construction of $f$ and the fact that $s$ is a splitting homomorphism, we have
\begin{align*}
    s_\tau^{-1} \circ f_\sigma^{-1} \circ f_{\sigma \tau}
    = s_\tau^{-1} \circ s_\sigma^{-1} \circ \theta_{\psi_\sigma}^{-1} \circ \theta_{\psi_{\sigma \tau}} \circ s_{\sigma \tau}
    = s_{\sigma\tau}^{-1} \circ \theta_{\psi_\sigma}^{-1} \circ \theta_{\psi_{\sigma \tau}} \circ s_{\sigma \tau}
    ={}^{\tau^{-1} \sigma^{-1}}\!\big(\theta_{\psi_\sigma}^{-1} \circ \theta_{\psi_{\sigma \tau}}\big),
\end{align*}
where the last equality follows from \eqref{e:k-form_section_action}.
Since $\psi\colon\Gamma\to G(k_s)$ is locally constant, so is the map
$\tau\mapsto\theta_{\psi_\sigma}^{-1}\circ\theta_{\psi_{\sigma\tau}}$.
As the $\Gamma$-action on $\Aut(\Abar)$ is continuous, it follows that
$\tau\mapsto
{}^{\tau^{-1} \sigma^{-1}}\!\big(\theta_{\psi_\sigma}^{-1}\!\circ\theta_{\psi_{\sigma\tau}}\big)$
is also locally constant, hence $f$ is continuous.

We set
\begin{equation}\label{l:k-band}
    \beta_\sigma=f_\sigma\cdot \Inn(\Abar)\in \SOut(\Abar).
\end{equation}
It follows from \eqref{eq:def1_1-cocycle} that $\beta\colon \Gamma \to \SOut(\Abar)$
is a homomorphism, and since $f$ is continuous, $\beta$ is a $k$-band.
We define a map
\begin{equation}\label{e:image_Z1'}
Z^1(k,A \to G) \to \operatorname{Band}(k,\Abar), \quad (u,\psi) \mapsto \beta,
\end{equation}
where $\beta$ is a $k$-band given by \eqref{l:k-band}, and $\operatorname{Band}(k,\Abar)$ denotes the set of $k$-bands for $\Abar$.
This map is analogous to \eqref{e:image_Z1} (see Lemma \ref{l:form_band_Phi}), and we denote its image by $\Phi$.

Following \cite[Section (1.17)]{FSS}, for a given $k$-band $\beta$,
the set of 2-cocycles $Z^2(k,\Abar,\beta)$ is the set of pairs $(u,f)$,
where $u \colon \Gamma\times \Gamma \to \Abar(k_s)$ is a locally constant map
and $f \colon \Gamma \to \SAut(\Abar)$ is a continuous set-theoretic section of \eqref{e:SAut_exact_seq}
such that
\begin{align*}
    \operatorname{inn}(u_{\sigma,\tau}) \circ f_\sigma \circ f_\tau
    &= f_{\sigma \tau}, \\
    u_{\sigma,\tau \nu} \cdot f_\sigma(u_{\tau,\nu})
    &= u_{\sigma \tau,\nu} \cdot u_{\sigma,\tau},  \\
    f_\sigma \!\!\!\! \mod \operatorname{Inn}(\Abar)
    &= \beta_\sigma.
\end{align*}
Here the term $f_\sigma(u_{\tau,\nu})$ is defined via the natural action of $\SAut(\Abar)$ on $\Abar(k_s)$;
see \cite[Section (1.3)]{FSS}.
The group $\operatorname{Maps}(\Gamma,\Abar(k_s))$ acts on $Z^2(k,\Abar,\beta)$ on the left as follows:
\[
    w * (u,f) = (u',f')\quad\ \text{for}\ \,w\in{\rm Maps}(\Gamma, \Abar(k_s)),
\]
where
\begin{equation*}
    u'_{\sigma,\tau} = w_{\sigma \tau} \cdot u_{\sigma,\tau} \cdot f_{\sigma}(w_\tau)^{-1} \cdot w_{\sigma}^{-1}, \quad
    f'_\sigma = \operatorname{inn}(w_\sigma) \circ f_\sigma
\end{equation*}
for all $\sigma, \tau \in \Gamma$.

\begin{definition}\label{d:H2-k}
We define
$H^2(k,\Abar,\beta)\coloneqq Z^2(k,\Abar,\beta)/\operatorname{Maps}\big(\Gamma,\bar A(k_s)\big)$.
A 2-cocycle $(u,f)\in Z^2(\Gamma,A,\beta)$ is called neutral if $u=1$, and a cohomology class $[u,f]\in H^2(\Gamma, A,\beta)$ is called neutral if it is the class of a neutral 2-cocycle.
\end{definition}

Let $\beta$ be a $k$-band.
If $f\colon\Gamma\to\SAut(\bar A)$ is a continuous lift of $\beta$, then the composition $f_* \colon \Gamma \xrightarrow{f} \SAut(\bar A) \to \Aut(A(k_s))$
is locally constant,
where the second homomorphism is constructed in \cite[Section (1.3)]{FSS}.
Hence composing $f_*$ with the natural homomorphism $\Aut(A(k_s)) \to \Out(A(k_s))$ yields an abstract $\Gamma$-band $\beta_*$.

By \cite[Lemma (1.19)]{FSS} and
\cite[Proposition 1.15]{Springer},
both $H^2(k,\bar A,\beta)$ and
$H^2(\Gamma,A(k_s),\beta_*)$
are naturally identified with the set of equivalence classes of extensions of topological groups
\[
1\to A(k_s)\xrightarrow{i} E \xrightarrow{p}\Gamma \to 1
\]
where $A(k_s)$ (resp. $\Gamma$) is equipped with the discrete (resp. profinite) topology,  $i$ and $p$ are open onto their respective images, and the associated band is $\beta_*$.
Hence there is a canonical bijection $H^2(k,\bar A,\beta) \simeq H^2(\Gamma,A(k_s),\beta_*)$.
Consequently, we may apply the results of Section \ref{s:coboundary} in the present setting.

Let $\beta$ be an arbitrary $k$-band.
For $g \in G(k_s)$ and a class $[u, f] \in H^2(k, \Abar, \beta)$, we consider
\begin{equation*}
    u'_{\sigma, \tau} \coloneqq {\ha}^{g^{-1}}\!u_{\sigma, \tau} \quad\ \text{and} \quad\ f'_\sigma \coloneqq \theta_{g^{-1}} \circ f_\sigma \circ \theta_g.
\end{equation*}
Then $\beta' \colon \Gal(k_s/k) \labelto{f'} \SAut(\Abar) \to  \SOut(\Abar)$ is a $k$-band, and there is a bijection
\begin{equation*}
    H^2(k,\Abar,\beta) \isoto H^2(k,\Abar,\beta'), \quad
    [u,f]  \mapsto [u',f'],
\end{equation*}
which preserves the subsets of neutral elements.
Furthermore, if $\beta$ is contained in $\Phi$ (the image of the map in \eqref{e:image_Z1'}), then so is $\beta'$.
Thus this bijection yields a right $G(k_s)$-action on $\bigsqcup_{\beta \in \Phi} H^2(k,\Abar,\beta)$ by $[u,f] \cdot g \coloneqq [u',f']$.
Similarly to Construction \ref{construction: H2}, we define
\begin{equation*}
H^2(k,A \mathrel{\mathrm{rel}} G) \coloneqq \bigsqcup_{\beta \in \Phi} H^2(k,\Abar,\beta)/G(k_s).
\end{equation*}
For $[u,f] \in H^2(k,\Abar,\beta)$ with $\beta \in \Phi$, we denote its image in $H^2(k,A \mathrel{\mathrm{rel}} G)$ by $[u,f]_G$, that is the $G(k_s)$-orbit of $[u,f]$.
A \textit{neutral element} in $H^2(k,A \mathrel{\mathrm{rel}} G)$ is defined as a $G(k_s)$-orbit containing a neutral element of $H^2(k,\Abar,\beta)$ for some $\beta \in \Phi$.
If an orbit $[u,f]_G$ is neutral, then every representative in this orbit is a neutral element in $H^2(k,\Abar,\beta')$ for a suitable choice of $\beta'$.

By Lemma \ref{l:Z2(Gamma-A-beta)} we have $(u,f)\in Z^2(k,\Abar,\beta)$.
We set $\Delta(u,\psi)=[u,f]\in H^2(k,\Abar,\beta)$.
By Lemma \ref{l:Delta_H2}, we have a well-defined map
\[
\Delta \colon H^1(k,A \to G) \to H^2(k,A \mathrel{\mathrm{rel}} G), \quad
[u,\psi] \mapsto \Delta(u,\psi)_G.
\]

\begin{proposition}\label{c:kang}
Let $A\labelt\rho G$ be a crossed module of smooth $k$-groups.
Then for a 1-cocycle $(u,\psi) \in Z^1(k, A \to G)$, its
class $[u,\psi] \in H^1(k,A \to G)$ lies in the image of $H^1(k,G)$
if and only if the second cohomology class $\Delta(u,\psi)\in H^2(k, \Abar,\beta)$ is neutral.
\end{proposition}

\begin{proof}
This is a particular case of Proposition \ref{p:kang}.
\end{proof}

\begin{proposition}\label{c:ext-to-H2} \
\begin{enumerate}
    \item A class $[u,\psi] \in H^1(k,A \to G)$ lies in the image of $H^1(k,G)$
    if and only if $\Delta(u,\psi)_G  \in H^2(k,A \mathrel{\mathrm{rel}} G)$ is neutral.
    \item There is an exact sequence
    \[
    H^1(k,A) \labelto{\rho_*} H^1(k,G) \labelto{i_*} H^1(k, A \to G) \labelto{\Delta} H^2(k,A \mathrel{\mathrm{rel}} G),
    \]
    where exactness at $H^1(k, A \to G)$ means that the image of $i_*$ coincides with the preimage of the subset of neutral elements in $H^2(k,A \mathrel{\mathrm{rel}} G)$ under $\Delta$.
\end{enumerate}
\end{proposition}

\begin{proof}
This is a particular case of Corollary \ref{c:extend_ext_seq}.
\end{proof}

\section[Non-abelian H\textsuperscript{2} over fields of Galois--Douai type]{Non-abelian \texorpdfstring{$H^2$}{H2} over fields of Galois--Douai type}
\label{s:Douai}

\begin{definition}\label{d:Douai}
Following an idea of Gonz\'alez-Avil\'es \cite[Definition 5.2]{GA-12},
we say that a field $k$ is {\em of Galois--Douai type} if
for every {\em simply connected semisimple} $k$-group $G$ with étale center $Z$,
the coboundary map for Galois cohomology
\begin{equation}\label{e:Galois}
\delta^1\colon H^1(k,G/Z)\to H^2(k, Z)
\end{equation}
is surjective.
\end{definition}

\begin{remark}\label{r:smooth}
In our case of an étale center, the map \eqref{e:Galois} is surjective
if and only if the coboundary map in flat cohomology
\begin{equation*}
\delta^1_\fppf\colon H^1_\fppf(k,G/Z)\to H^2_\fppf(k, Z)
\end{equation*}
is surjective;  see \cite{GA-MO}. For details, see  \cite[Theorem 11.7 and Remark 11.8(3)]{Grothendieck}
or \cite[Theorem III.3.9, Corollary III.4.7, and Remark III.4.8(a)]{Milne-EC}.
\end{remark}

\begin{examples} \
\begin{enumerate}
\item[\rm (i)] All local and global fields of characteristic 0 are of Galois--Douai type;
see the references in the proof of \cite[Lemma  5.7]{Borovoi-Duke}.

\item[\rm (ii)]
All local and global fields of positive characteristic are of Galois--Douai type;
see Remark \ref{r:smooth} and
 \cite[Theorem A]{Thang}, or \cite[Proposition 4.5]{Rosengarten}.

\item[\rm (iii)] The geometric fields of types (gl), (ll), and (sl) of \cite{CTGP} are of Galois--Douai type;
see \cite[Examples 5.4(vi)]{GA-12}.

\end{enumerate}

\end{examples}

\begin{theorem}
\label{t:Douai}
Consider the second non-abelian Galois cohomology set
$H^2(k, \ol G,\beta)$ of Definition \ref{d:H2-k},
where $\ol G$ is a simply connected semisimple $k_s$-group and
$\beta$ is any $k$-band.
We assume that $Z(\ol G)$ is \'etale.
If $k$ is a {\em field of Galois--Douai type},
then all elements of $H^2(k,\ol G,\beta)$ are neutral.
\end{theorem}

\begin{remark}\label{r:B-Duke}
In \cite[Corollary 5.6]{Borovoi-Duke}, Theorem \ref{t:Douai} was proved
over local and global fields of characteristic 0.
\end{remark}

\begin{remark}\label{r:DLA}
Let $L=L_{\rm Gir}(\ol G,\beta)$ denote the band (lien)
of Giraud \cite{Giraud} corresponding to $\ol G$ and $\beta$; see \cite[Subsection 1.2.1]{DLA}.
Let $H^2_{\rm Gir}(k,L)$ denote the second non-abelian cohomology of Giraud \cite{Giraud}; see \cite[Subsection 2.2.2]{DLA}.
There is a canonical bijection $H^2(k,\ol G,\beta)\cong H^2_{\rm Gir}(k,L)$
preserving the sets of neutral elements; see \cite[Proposition 2.1]{DLA}.
(In \cite{DLA} the authors assume that the field $k$ is perfect, but they do not use this assumption in
the proof of their Proposition 2.1.)
In \cite[Theorem 1.1]{Douai-local} and \cite[Proposition 4.1 and Theorem 5.1]{Douai-global},
it was proved that, when $G$ is simply connected semisimple and $k$ is a non-archimedean local field or a global field admitting no real embeddings,  all elements of $H^2_{\rm Gir}(k,L)$ are neutral.
See also \cite[Theorems VI.1.4, VI.1.5, and VIII.1.4]{Douai-thesis}.
This gives Theorem \ref{t:Douai} for such fields (even without the assumption that $Z(\ol G)$ is \'etale).
\end{remark}

\begin{proof}[Proof of Theorem \ref{t:Douai}\hs]
Since our simply connected semisimple group $\ol G$ is connected and reductive,
the cohomology set $H^2(k,\ol G,\beta)$ contains a neutral element $(1,f)$;
see \cite[Lemma 1.1]{Douai-local}, or \cite[Proposition V-3.2]{Douai-thesis}, or \cite[Proposition 3.1]{Borovoi-Duke}.
Then the splitting homomorphism $f\colon \Gal(k_s/k)\to \SAut(\ol G)$
of \eqref{e:SAut_exact_seq} defines, by descent, a $k$-form $G$ of $\ol G$; see \cite[Remark (1.15)]{FSS}.
We write $H^2(k,G)$ for $H^2(k,\ol G,\beta)$.

We consider the composite map
\begin{equation}\label{e:Delta-i*}
\chi\colon H^1(k,G/Z)\labelt{\delta^1} H^2(k,Z)\labelt{i_*} H^2(k,G),
\end{equation}
where $\delta^1$ is the coboundary map and where the map
\[i_*\colon H^2(k,Z)\to H^2(k,G)\]
comes from the canonical  action of $H^2(k,Z)$ on $H^2(k,G)$
and is given by
\[ [z]\mapsto [z,f]\quad\ \text{where}\ \,z\in Z^1(k,Z).\]
Here $f=f_G\colon \Gal(k_s/k)\to\SAut(\ol G)$ is as above.
See \cite[Proposition 1.17]{Springer};
by this proposition the map $i_*$ is bijective, hence surjective.
On the other hand, the map $\delta^1$ sends the class $[\psi]$
of $\psi\in Z^1(k,G/Z)$ to the class of
\[(\sigma,\tau)\mapsto \psitil_\sigma \cdot f_\sigma(\psitil_\tau)\cdot \psitil_{\sigma\tau}^{-1}\]
where $\psitil\colon\Gal(k_s/k)\to G(k_s)$ is a locally constant lifting of $\psi$.
There exists such a lifting by the surjectivity of the map $G(k_s)\to (G/Z)(k_s)$,
which holds because $H^1(k_s,Z)=1$ by the smoothness of $Z$.
We note that $\delta^1$ is independent of the choice of a lifting.
Then
\[\chi[\psi]=[u,f]\quad\ \text{where}\ \,u_{\sigma,\tau}
=\psitil_\sigma \cdot f_\sigma(\psitil_\tau)\cdot \psitil_{\sigma\tau}^{-1}.\]
Note that $u_{\sigma,\tau} \in Z(k_s)$ by construction.
Let $w=\psitil^{-1}\colon \Gal(k_s/k)\to G(k_s)$ and define
\[
(u',f)= w * (1,f')\quad\ \text{where} \ \, f'_\sigma=\inn(\psitil_\sigma)\cdot f_\sigma
\]
using the notation of \eqref{eq:coboundary_2-cycle}.
Then we have
\[
u'_{\sigma,\tau}
= \psitil_{\sigma\tau}^{-1} \cdot \psitil_\sigma \cdot f_\sigma(\psitil_\tau)
= \psitil_{\sigma\tau}^{-1} \cdot (\psitil_\sigma \cdot f_\sigma(\psitil_\tau) \cdot \psitil_{\sigma\tau}^{-1}) \cdot \psitil_{\sigma\tau}
= \psitil_{\sigma\tau}^{-1} \cdot u_{\sigma,\tau} \cdot \psitil_{\sigma\tau}.
\]
Since $u_{\sigma,\tau} \in Z(k_s)$, it commutes with $\psitil_{\sigma\tau}$, which yields $u'_{\sigma,\tau} = u_{\sigma,\tau}$.
This implies that $(u,f) = w * (1,f')$, hence $\chi[\psi]$ is a neutral element.

Now for a field $k$ of Galois--Douai type,
the coboundary map $\delta^1$ is surjective.
Thus the composite map $\chi$ of  \eqref{e:Delta-i*} is surjective,
and all elements in its image are neutral.
We conclude that all elements of $H^2(k,G)$ are neutral, as desired.
\end{proof}

\section{Crossed cohomology and the crossing map}
\label{s:cr-cohomology}

For a field $k$, let $\Red_k$ denote the category of smooth reductive $k$-groups,
not necessarily connected, with  $k$-homomorphisms.
Let $\CrMod_k$ denote the category of crossed modules of $k$-groups.
In Section \ref{s:crossed} we defined the functor
\begin{equation}\label{e:functor-cr}
G\rightsquigarrow \CrM(G) = (G^\ssc,G,\rho,\theta)
\end{equation}
from $\Red_k$ to $\CrMod_k$\hs.
For simplicity, we also write $(G^\ssc\to G)$ for $\CrM(G)$.
We regard $\CrM(G)$ as a complex in degrees $-1$ and $0$.

In \cite[Sections 3.3 and 3.6]{Borovoi-Memoir}, in the case of characteristic 0, the first-named author
constructed functors of Galois hyper-cohomology $H^0$ and $H^1$
\[ (A\to G)\rightsquigarrow H^i(k,A\to G)\coloneqq
    H^i\big(\Gal(k_s/k), A(k_s)\to G(k_s)\big)\quad\ \text{for}\ \, i=0,1\]
 where $(A\to G)$ is a crossed module of $k$-groups.
Here $H^0(k, -)$ is a functor from $\CrMod_k$ to the category of groups,
and $H^1(k,-)$ is a functor from $\CrMod_k$ to the category of pointed sets.
These constructions can be naturally generalized to the case of
a crossed module of smooth $k$-groups over a field $k$ of arbitrary characteristic.

Combining the functors $H^i$ with the functor $\CrM$ of \eqref{e:functor-cr},
we obtain functors of {\em crossed cohomology}
\[ G\rightsquigarrow  H^i_\crr(k,G)\coloneqq H^i(k, G^\ssc\to G) \quad\ \text{for}\ \, i=0,1.\]
Here $H^0_\crr(k,-)$ is a functor from $\Red_k$ to the category of groups,
and $H^1_\crr(k,-)$ is a functor from $\Red_k$ to the category of pointed sets.
The morphism of crossed modules of reductive $k$-groups
\[(1\to G)\to (G^\ssc\to G)\]
induces the {\em crossing homomorphism}
\[\crr^0\colon H^0(k,G)=H^0(k, 1\to G)\to H^0(k,G^\ssc\to G)\eqqcolon H^0_\crr(k,G)\]
and the {\em crossing map}
\begin{equation}\label{e:cr1}
\crr^1\colon H^1(k,G)=H^1(k, 1\to G)\to H^1(k,G^\ssc\to G)\eqqcolon H^1_\crr(k,G).
\end{equation}

For a  homomorphism of reductive $k$-groups (not necessarily connected)
\[\vk\colon G_1\to G_2\]
we have a commutative diagram of crossed modules of $k$-groups
\[
\xymatrix{
(1\to G_1)\ar[r]\ar[d] &(1\to G_2)\ar[d]\\
\CrM(G_1)\ar[r] &\CrM(G_2),
}
\]
which induces a commutative diagram of pointed sets
\begin{equation}\label{e:commut-cr}
\begin{aligned}
\xymatrix@C=17mm{
H^1(k, G_1)\ar[r]^-{\vk_*}\ar[d] &H^1(k, G_2)\ar[d]\\
H^1_\crr(k,G_1)\ar[r]^-{\vk_\crr} &H^1_\crr(k,G_2).
}
\end{aligned}
\end{equation}

The short exact sequence of complexes of $k$-groups
\[ 1\to (1\to G)\labelt{i} (G^\ssc\to G)\to (G^\ssc\to 1)\to 1\]
(in which the arrow $i$ is a morphism of crossed modules)
gives rise to a hyper-cohomology exact sequence containing the maps $\crr^0$ and $\crr^1$\,:
\begin{multline}\label{e:exact-crossed}
H^0(k,G^\ssc)\labelto{\rho_*} H^0(k,G)\labelto{\crr^0} H^0(k,G^\ssc\to G) \\
    \labelto{\delta} H^1(k,G^\ssc)\labelto{\rho_*} H^1(k,G)\labelto{\crr^1} H^1(k, G^\ssc\to G)
    \labelto\Delta H^2(k, G^\ssc\textrm{\ rel\ }G);
\end{multline}
see  Proposition \ref{c:ext-to-H2} and \cite[Corollary 3.4.3]{Borovoi-Memoir}.

\begin{theorem}
\label{t:cr-surjective}
Let $k$ be a field of Galois--Douai type, for example, a local or global field, and let $G$ be a reductive $k$-group, not necessarily connected or quasi-connected.
Assume that the finite commutative  $k$-group $Z^\sscp\coloneqq Z(G^\ssc)$ is \'etale.
Then the crossing map
\[\crr^1\colon H^1(k,G)\to H^1(k,G^\ssc\to G)\]
is surjective.
\end{theorem}

This was earlier proved by Labesse \cite[Proposition 1.6.7]{Labesse}
in the case of a quasi-connected reductive group over a  local or global field of characteristic 0.

\begin{proof}
For all bands $\beta$ for  $\ol G^\ssc$, it follows from Theorem \ref{t:Douai}
that all elements of  $H^2(k,\ol G^\ssc,\beta)$ are neutral.
It follows that all elements of $H^2(k, G^\ssc\textrm{\ rel\ }G)$ are neutral.
Now the theorem follows from the exactness of the sequence \eqref{e:exact-crossed}
at $H^1(k, G^\ssc\to G)$.
\end{proof}

\begin{theorem}
\label{t:cr-bijective}
In Theorem \ref{t:cr-surjective}, if moreover,
$k$ is a {\em non-archimedean} local field
or a global field {\em without real places}
(that is, a totally imaginary number field or a global function field),
then in \eqref{e:exact-crossed} the crossing homomorphism $\crr^0$ is surjective
and the crossing map  $\crr^1$ is bijective.
\end{theorem}

\begin{proof}
When  $k$ is a non-archimedean local field
or a global field  without real places,
we have $H^1(k,G^\ssc)=1$ (Kneser, Bruhat and Tits, Harder, Chernousov),
and we see from \eqref{e:exact-crossed} that $\crr^0$ is surjective
and that the kernel of $\crr^1$ is trivial.
Using twisting, we obtain the injectivity of $\crr^1$,
which together with Theorem \ref{t:cr-surjective}
gives the bijectivity of $\crr^1$.
\end{proof}

\begin{construction}
Let $G$ be a reductive group, not necessarily connected, over a {\em number field} $k$.
Let $\V_\infty$ denote the set of infinite (that is, archimedean) places of $k$,
and for $v\in \V_\infty(k)$, let $k_v$ denote the completion of $k$ at $v$.
Consider the commutative diagram
\[
\xymatrix{
H^1(k,G) \ar[r]^-{\crr^1}\ar[d]_-{\loc_\infty}  & H^1_\crr(k,G)\ar[d]^-{\loc_\infty} \\
H^1(k_\infty, G)\ar[r]^-{\crr^1}               & H^1_\crr(k_\infty,G)
}
\]
where we write
\[H^1(k_\infty, G)=\prod_{v\in\V_\infty(k)} H^1(k_v,G)\quad\ \text{and}
     \quad\ H^1_\crr(k_\infty, G)=\prod_{v\in\V_\infty(k)} H^1_\crr(k_v,G).\]
This diagram defines a map
\begin{equation}\label{e:fiber-product}
H^1(k,G)\,\lra\, H^1_\crr(k,G)\underset{H^1_\crr(k_\infty,G)}{\boldsymbol{\times}} H^1(k_\infty, G).
\end{equation}
\end{construction}

\begin{theorem}\label{t:Labesse}
Let $G$ be a reductive group, not necessarily connected, over a {\em number field} $k$. Then:
\begin{enumerate}
\item[\rm (i)] The map \eqref{e:fiber-product} is surjective.
\item[\rm (ii)] When $G$ is connected, the map \eqref{e:fiber-product} is bijective.
\item[\rm (iii)] In general, when $G$ is not connected, the map \eqref{e:fiber-product} may not be injective.
\end{enumerate}
\end{theorem}

\begin{proof}
Assertion (ii) was proved in \cite[Theorem 5.11(i)]{Borovoi-Memoir}.
For assertion (iii), see the proof after Construction \ref{con:not-injective} below.

We prove assertion (i), which  is a generalization of \cite[Theorem 1.6.10]{Labesse}, where {\em quasi-connected}
reductive groups were considered.

From the exact sequence \eqref{e:exact-crossed}
we obtain a commutative diagram with exact rows
\begin{equation}\label{e:diag-Gsc-G-cr}
\begin{aligned}
\xymatrix{
H^1(k,G^\ssc)\ar[r]\ar[d] &H^1(k,G)\ar[r]\ar[d] &H^1_\crr(k,G)\ar[d] \\
H^1(k_\infty,G^\ssc)\ar[r] &H^1(k_\infty,G)\ar[r] &H^1_\crr(k_\infty,G)
}
\end{aligned}
\end{equation}
Let $(\xi_\crr,\xi_\infty)\in H^1_\crr(k,G)\times H^1(k_\infty,G)$
be an element of the fiber product in \eqref{e:fiber-product}.
By Theorem \ref{t:cr-surjective}, $\xi_\crr$ is the image of some element $\xi\in H^1(k,G)$.
We choose a 1-cocycle $c\in Z^1(k,G)$ representing $\xi$,
and twist the diagram \eqref{e:diag-Gsc-G-cr} by $c$;
see Serre \cite[Section I.5.3]{Serre} for twisting.
We obtain the following diagram:
\begin{equation}\label{e:c-diag-Gsc-G-cr}
\begin{aligned}
\xymatrix{
H^1(k,\hs{}_cG^\ssc)\ar[r]\ar[d] &H^1(k,\hs{}_cG)\ar[r]\ar[d] &H^1_\crr(k,\hs{}_cG)\ar[d] \\
H^1(k_\infty,\hs{}_cG^\ssc)\ar[r] &H^1(k_\infty,\hs{}_cG)\ar[r] &H^1_\crr(k_\infty,\hs{}_cG)
}
\end{aligned}
\end{equation}

Under the twisting bijection
\[\tau_c\colon H^1(k,\hs_c G)\to H^1(k,G),\quad\ [c']\mapsto [c'c] \]
of \cite[I.5.3, Proposition 35 bis]{Serre},
the cohomology class $\xi=[c]\in H^1(k,G)$ corresponds to $1\in H^1(k,\hs_c G)$.
Moreover, the class $\xi_\infty\in H^1(k_\infty, G)$ corresponds to
$\xi'_\infty =\tau_{\infty,c}^{-1}(\xi_\infty)\in H^1(k_\infty, \hs_c G)$.
Since $\xi'$ corresponds to $1\in H^1(k,\hs_c G)$,
we see from the diagram  \eqref{e:c-diag-Gsc-G-cr}
that the image of $\xi'_\infty$ in $H^1_\crr(k_\infty,\hs_c G)$ is 1.
Thus $\xi'_\infty$ comes from some $\xi_\infty^{\ssc\,\prime}\in H^1(k_\infty,\hs_c G^\ssc)$.

Since $_c G^\ssc$ is a connected $k$-group, by \cite[Satz 5.5.1]{Harder}
the localization map
\[\loc_\infty\colon  H^1(k,\hs_c G^\ssc)\to H^1(k_\infty, \hs_c G^\ssc)\]
is surjective (it is even bijective, but we will not use this).
Let $\xi^{\ssc\,\prime}\in H^1(k,\hs_c G^\ssc)$ be a preimage of $\xi_\infty^{\ssc\,\prime}$.
Let $\xi'\in H^1(k,\hs_c G)$ denote the image of $\xi^{\ssc\,\prime}$.
Then $\crr^1(\xi')=1$
and $\loc_\infty(\xi')=\xi'_\infty$.
Returning to $G$ from $_c G$, we set $\xi''=\tau_c(\xi')\in H^1(k,G)$.
Then $\crr^1(\xi'')=\xi_\crr$ and $\loc_\infty(\xi'')=\xi_\infty$,
which completes the proof of (i).
\end{proof}

\begin{construction}\label{con:not-injective}
We construct a non-connected reductive group $G$ over a number field $k$
for which the map \eqref{e:fiber-product} is not injective.

Let $k=\Q$, let $L_\im/k$ be an imaginary quadratic extension with Galois group $\Gamma_\im\coloneqq \Gal(L_\im/k)=\{1,\sigma\}$,
and let $L_\re/k$ be a real quadratic extension with Galois group $\Gamma_\re\coloneqq \Gal(L_\re/k)=\{1,\tau\}$.
Write $L=L_\im\cdot L_\re\subset \kbar\subset\C$, and write $\Gamma=\Gal(L/k)=\Gamma_\im\times\Gamma_\re=\{1,\sigma,\tau,\sigma\tau\}$.
Consider the simply connected $k$-group $G_0={\rm SU}(2,L_\im/k)$.

Following Holt \cite{DerekHolt}, we consider a short exact sequence of finite  $\Gamma$-modules
\begin{equation}\label{e:ABC}
0\to A\labelt{i} B\labelt{j} C\to 0,
\end{equation}
where $B=\langle x\rangle \cong\Z/8$, $A=\langle4x\rangle\cong \Z/2$, and $C=B/A\cong\Z/4$.
The group $\Gamma=\langle\sigma,\tau\rangle$ acts on $B$ as follows:
\[\sigma(b)=5b,\quad\ \tau(b)=-b\quad\ \text{for}\ \,b\in B.\]
Then $\Gamma$ acts trivially on $A$, and so it naturally acts on $C$.
Note that $\sigma$ acts on $C$ trivially.

We regard \eqref{e:ABC} as a short exact sequence of finite algebraic $k$-groups.
We identify $A$ with the center $\{1,-1\}=Z(G_0)(k)$, and we set $G=(G_0\times_k B)/A$
with $A$ embedded diagonally. We have a short exact sequence
\begin{equation}\label{e:G0GC}
1\to G_0\labelto\iota G\labelto\lambda C\to 1.
\end{equation}
Since $G_0$ is simply connected, we have
\[ H^i_\crr(k,G)\coloneqq H^i(k, G_0\to G)=H^i(k,G/G_0)=H^i(k,C)\quad\ \text{for}\ \, i=0,1.\]

Consider the homomorphism
$z_\im\colon \Gamma_\im=\{1,\sigma\}\to Z(G_0)(k)=\{\pm1\}$ sending $\sigma$ to $-1$, and let
\[z_0\colon \Gal(\kbar/k)\onto \Gal(L_\im/k)\labelto{z_\im} G_0(k)\]
be the composite homomorphism. Then $z_0$ is a 1-cocycle,
$z_0\in Z^1(k, G_0)$. Let $z=\iota\circ z_0\in Z^1(k,G)$.
Set $\xi_0=[z_0]\in H^1(k,G_0)$, $\xi=\iota_*(\xi_0)=[z]\in H^1(k,G)$.
We write $\loc$ for the localization map $\loc_\infty\colon H^1(k,G_0)\to H^1(\R,G_0)$.
Then $\loc(\xi_0)=[\sigma\mapsto -1]\in H^1(\R,G_0)$.
It follows that $\loc(\xi_0)\neq 1$, whence $\xi_0\neq 1$.
\end{construction}

\begin{proof}[Proof of Theorem \ref{t:Labesse}(iii)]
\label{ex:counter-ex-Labesse}
We show that for  $k$, $G$, and $\xi\in H^1(k,G)$ as in Construction \ref{con:not-injective},
the map \eqref{e:fiber-product} is not injective.
Namely, we have
$\crr^1(\xi)=1\in H^1_\crr(k,G)=H^1(k,C)$, $\loc(\xi)=1\in H^1(\R,G)$, but $\xi\neq 1\in H^1(k,G)$.

Indeed, the short exact sequence \eqref{e:G0GC} gives rise to a commutative diagram with exact rows
\[
\xymatrix@C=13mm{
C^\Gamma\ar[r]^-\delta\ar[d]         & H^1(k,G_0)\ar[r]^-{\iota_*}\ar[d] & H^1(k,G)\ar[r]^-{\lambda_*}\ar[d] & H^1(k,C)\ar[d] \\
C^{\Gamma_\im}\ar[r]^-{\delta_\infty} & H^1(\R,G_0)\ar[r]^-{\iota_\infty} & H^1(\R,G)\ar[r]^-{\lambda_\infty} & H^1(\R,C)
}
\]
We have $\crr^1(\xi)=\lambda_*(\xi)=1\in H^1(k,C)$ because $\xi=\iota_*(\xi_0)\in\im\,\iota_*$.

We show that $\loc(\xi)=1$.
Consider $[x]\in C=C^{\Gamma_\im}$. We have $\sigma(x)-x=5x-x=4x\neq 0\in A$.
It follows that $\delta_\infty[x]=[\sigma\mapsto -1]=\loc(\xi_0)\in H^1(\R,G_0)$.
Thus $\loc(\xi_0)\in\im\, \delta_\infty$, whence $\loc(\xi)\in \im(\iota_\infty\circ\delta_\infty)=\{1\}$.

We show that $\xi\neq 1$. It suffices to show that $\xi_0\notin \im\,  \delta$.
We have $C^\Gamma=\{0, [2x]\}$. Clearly, $\delta(0)=1\neq \xi_0$.
We have $\loc[2x]=[2x]\in C=C^{\Gamma_\im}$, and $\delta_\infty[2x]=[\sigma(2x)-2x]=[10x-2x]=[8x]=0$.
Since $\loc(\xi_0)\neq 1$, we see that $\delta_\infty(\loc[2x])=1\neq \loc(\xi_0)$,
whence $\delta[2x]\neq \xi_0$. We see that $\xi_0\notin \im\,\delta$; hence, indeed $\xi\neq 1$.
This completes the proof of Theorem \ref{t:Labesse}(iii).
\end{proof}

\section{Quasi-connected reductive groups and principal homomorphisms}
\label{s:quasi-connected}

\begin{definition}\label{d:q-c-r}
A linear algebraic group $G$ over a field $k$
is called {\em quasi-connected reductive} if
\begin{enumerate}
\item $G$ is smooth and its identity component $G^0$ is reductive;
\item the center $Z(G)$ of $G$ is a $k$-group of multiplicative type (not necessarily smooth);
\item $G=Z(G)\cdot G^\sss$.
\end{enumerate}
\end{definition}
Here, for two $k$-subschemes $X,Y\subseteq G$, we write $G=X\cdot Y$
if for every $k$-algebra $R$ and every $g\in G(R)$, there exist a faithfully flat homomorphism $R\into R'$
and elements $x\in X(R')$, $y\in Y(R')$ such that $g=x y$,
where by abuse of notation we also write $g$ for the image of $g$ in $G(R')$.

\begin{example}\label{ex:OOnF}
Assume that ${\rm char}(k)\neq 2$.
Let  $n\ge 2$, and let $F\in M_{n\times n}(k)$ be a non-degenerate symmetric square matrix.
Consider $G=\OO_{n,F}\subset\GL_{n,k}$,
the orthogonal group of $F$ given by the equation $gFg^t=F$ for $g\in\GL(n,R)$,
$R$ being a $k$-algebra.
Then $G^0=\SO_{n,F}$, whence $G^\ssc=G^\sss=\SO_{n,F}$\hs.
Moreover, $Z(G)=\mu_2=\{\pm 1\}$.
If $n$ is odd, then
\[Z(G)\cdot G^\sss=\mu_2\cdot\SO_{n,F}=\OO_{n,F}=G,\]
and so $\OO_{n,F}$ is quasi-connected reductive.
However, if $n$ is even, then
\[Z(G)\cdot G^\sss=\mu_2\cdot\SO_{n,F}=\SO_{n,F}=G^0\neq G,\]
and so $\OO_{n,F}$ is not quasi-connected reductive.
\end{example}

We compare Definition \ref{d:q-c-r} with the following definition from \cite{BGR}:

\begin{definition}[{\cite[Definition 2.3]{BGR}}]
\label{d:q-c-r-bis}
A linear algebraic group $G$ over a field $k$
is called {\em quasi-connected reductive} if
\begin{enumerate}
\item $G$ is smooth and its identity component $G^0$ is reductive;
\item[$(2')$] $Z(G)(\kbar)$ consists of semisimple elements, where $\kbar$ is an algebraic closure of $k$;
\item[$(3')$] $G(\kbar)=Z(G)(\kbar)\cdot G^\sss(\kbar)$.
\end{enumerate}
\end{definition}

\begin{lemma}\label{l:equivalence}
Definitions \ref{d:q-c-r} and \ref{d:q-c-r-bis} are equivalent.
\end{lemma}

\begin{proof}
Let $G$ be as in Definition \ref{d:q-c-r}.
It follows from (2) that $Z(G)$ is isomorphic
to a $k$-subgroup of a $k$-torus, from which (2$'$) follows.
Moreover, the morphism of $k$-schemes
\begin{equation*}
\psi\colon Z(G)\times_k G^\sss\to G,\quad\ (z,g^\sss)\mapsto z\cdot g^\sss
\end{equation*}
is clearly a homomorphism, and by Fact \ref{f:W-ff}, it follows from (3)
that this homomorphism is faithfully flat, whence we obtain $(3')$.

Conversely, from (1), (2$'$), and $(3')$, it follows that $Z(G)$ is of multiplicative type;
see \cite[Proposition 2.9]{BGR}.
Since by $(3')$ the homomorphism $\psi$ is surjective on $\kbar$-points,
it is surjective by \cite{vDdB}; see also \cite[Tag 02KV]{stacks-project}.
Since, moreover, $G$ is smooth by (1), we conclude
from \cite[Proposition 1.70]{Milne-AG} that $\psi$ is faithfully flat,
whence by Fact \ref{f:W-ff} below we have the equality $G=Z(G)\cdot G^\sss$ of (3).
\end{proof}

Observe that the equivalent Definitions \ref{d:q-c-r} and \ref{d:q-c-r-bis}
are also equivalent to Definition \ref{d:Labesse} (due to Labesse);
see \cite[Theorems 2.14 and 2.18]{BGR}.

\begin{fact}
\label{f:W-ff}
Let $\varphi \colon H\to G$ be a homomorphism of affine $k$-groups.
Then $\varphi$ is faithfully flat if and only if for every $k$-algebra $R$ and element $g\in G(R)$ there exists
a faithfully flat homomorphism of $k$-algebras  $R\into R'$ and an element $h\in H(R')$
such that $g=\varphi(h)$.
\end{fact}

\begin{proof}
We denote the coordinate rings of $H$ and $G$ by $\mathcal{O}_H$ and $\mathcal{O}_G$, respectively.
By \cite[Theorem 3.31]{Milne-AG}, $\varphi$ is faithfully flat if and only if the corresponding ring homomorphism $\mathcal{O}_G \to \mathcal{O}_H$ is injective.
By \cite[Section 15.5, Theorem]{Waterhouse}, this injectivity is equivalent to the desired condition, which concludes the proof.
\end{proof}

\begin{fact}[{\cite[Section 15.6(b)]{Waterhouse}}]
\label{f:W-inj}
Let $G$ be a $k$-group and $R\into R'$ a faithfully flat homomorphism of $k$-algebras.
Then the natural homomorphism $G(R)\to G(R')$ is injective.
\end{fact}

\begin{lemma}\label{l:preimage-center}
Let $\pi\colon H\to G$ be a faithfully flat homomorphism of $k$-groups.
Then $\pi\big(Z(H)\big)\subseteq Z(G)$.
\end{lemma}

\begin{proof}
Let $z\in Z(H)(R)$ for some commutative unital $k$-algebra $R$.
Consider $\pi(z)\in G(R)$ and any element $g\in G(R')$ for any $R$-algebra $R'$.
Since the homomorphism $\pi\colon H\to G$ is faithfully flat,
by Fact \ref{f:W-ff}
there exists a faithfully flat homomorphism $R'\into R''$ and an element $h\in H(R'')$ such that $g=\pi(h)$.
Since $z\in Z(H)(R)$, the element $z$ commutes with $h$ in $H(R'')$.
It follows that $\pi(z)$ commutes with $g=\pi(h)$ in $G(R'')$,
and by Fact \ref{f:W-inj} it commutes with $g$ in $G(R')$.
Thus $\pi(z)\in Z(G)(R)$, as desired.
\end{proof}

\begin{lemma}\label{l:AN}
Let
\[ 1\to N\to H\labelt\varphi G\to 1\]
be a short exact sequence of $k$-groups (then $\varphi$ is faithfully flat).
Let $A\subset H$ be a $k$-subgroup.
Then $H=A\cdot N$ if and only if the restriction $\varphi|_A\colon A\to G$ is faithfully flat.
\end{lemma}

\begin{proof}
Assume that $H=A\cdot N$.
Let $g\in G(R)$ for some $k$-algebra $R$. Since $\varphi$ is faithfully flat, by Fact \ref{f:W-ff}
there exist a faithfully flat homomorphism $R\into R'$ and an element $h\in H(R')$
such that $g=\varphi(h)$. Since  $H=A\cdot N$, then there exist a faithfully flat homomorphism $R'\into R''$
and elements $a\in A(R'')$, $n\in N(R'')$ with $h=an$. Since $N=\ker \varphi$,
we have $\varphi(n)=1$, whence  $g=\varphi(h)=\varphi(an)=\varphi(a)\varphi(n)=\varphi(a)=\varphi|_A(a)$.
By Fact \ref{f:W-ff} this means that $\varphi|_A$ is faithfully flat.

Conversely, assume that  $\varphi|_A$ is faithfully flat.
Let $h\in H(R)$ for some $k$-algebra $R$, and set $g=\varphi(h)$.
Since  $\varphi|_A$ is faithfully flat, by Fact \ref{f:W-ff}
there exist a faithfully flat homomorphism $R\into R'$
and an element $a\in A(R)$ such that $g=\varphi|_A(a)=\varphi(a)$.
Set $n=a^{-1}h\in H(R')$; then $\varphi(n)=\varphi(a)^{-1}\varphi(h)=1\in G(R')$, whence $n\in N(R')$.
Clearly, we have $h=an$, which by Fact \ref{f:W-ff}  means that $H=A\cdot N$.
\end{proof}

\begin{lemma}\label{l:cZ_G}
Let $G$ be a quasi-connected reductive $k$-group.
Then $Z(G)=\cZ_G(G^\sss)$ where $\cZ_G$ denotes the centralizer in $G$.
\end{lemma}

\begin{proof}
The inclusion $Z(G) \subseteq \mathcal{Z}_G(G^\sss)$ is trivial.
To prove the reverse inclusion,
suppose that $x \in \mathcal{Z}_G(G^\sss)(R)$.
We must show that for every (commutative unital) $R$-algebra $R'$
and every $g \in G(R')$ we have $x g x^{-1} = g$ in $G(R')$;
see \cite[Proposition 1.92]{Milne-AG}.
Since $G=Z(G)\cdot G^\sss$ by Definition \ref{d:q-c-r}(3),
there exist a faithfully flat homomorphism $R' \into R''$ and elements
$z \in Z(G)(R'')$,  $g^\sss \in G^\sss(R'')$
such that $g = z\cdot g^\sss $ in $G(R'')$.
Since $x \in \mathcal{Z}_G(G^\sss)(R)$ and $z\in Z(G)(R'')$, we have
$$
x g x^{-1} = x(z\cdot g^\sss) x^{-1} = xzx^{-1}\cdot x g^\sss x^{-1}=z \cdot g^\sss = g
$$
in $G(R'')$.
By Fact \ref{f:W-inj} we conclude that $x g x^{-1} = g$ in $G(R')$, as desired.
\end{proof}

\begin{corollary}\label{c:cZ_G}
For a quasi-connected reductive $k$-group $G$, write  $Z^\sssp=Z(G^\sss)$.
Then we have
$Z(G)\cap G^\sss=Z^\sssp$.
\end{corollary}

\begin{proof}
By Lemma \ref{l:cZ_G} we have
\[ Z(G)\cap G^\sss=\cZ_G(G^\sss)\cap G^\sss
=\cZ_{G^\sss}(G^\sss)=Z(G^\sss)=Z^\sssp.\qedhere\]
\end{proof}

\begin{proposition}\label{p:q-ab}
For a quasi-connected reductive $k$-group $G$,
the crossed module $\CrM(G)=(G^\ssc\labelt\rho G,\theta)$
is {\emm quasi-abelian} in the sense
of Gonz\'alez-Avil\'es \cite[Definition 3.2]{GA},
that is, the following assertions hold:
\begin{enumerate}
\item[\rm (i)] $Z(G)$, when acting on $G^\ssc$ via $\theta$, acts trivially,
\item[\rm (ii)] $G=(\im\hs\rho)\cdot Z(G)$, and
\item[\rm (iii)] the map $Z(G^\ssc) \to Z(\im\hs\rho)$ is surjective.
\end{enumerate}
\end{proposition}

\begin{proof}
(i) Since $Z(G)$ acts trivially on $G$ by conjugation, it acts trivially on $G^\ssc$ via $\theta$
(by the functoriality of the assignment $G\functor G^\ssc$).

(ii) We have $\im\hs\rho=G^\sss$,
and the equality $G=Z(G)\cdot G^\sss$ holds by Definition \ref{d:q-c-r}(3).

(iii) We must show that for the universal cover $G^\ssc\to G^\sss$, we have
$\rho\big(Z(G^\sscp)\big)=Z(G^\sssp)$, which is well known.
\end{proof}

\begin{definition}
A {\em quasi-torus} over a field $k$ is a smooth $k$-group of multiplicative type.
\end{definition}

\begin{lemma}\label{l:qt}
Let $G$ be a quasi-connected reductive $k$-group.
Set $G^\qt=G/G^\sss$ and consider the natural quotient homomorphism
$q\colon G\to G^\qt$.
\begin{enumerate}
\item[\rm (i)] The restriction $q_\mZ\colon Z(G)\to G^\qt$ of $q$ is faithfully flat.
\item[\rm (ii)] $G^\qt$ is a quasi-torus (which explains the notation).
\end{enumerate}
\end{lemma}

\begin{proof}
By Lemma \ref{l:AN}, assertion (i) follows from Definition \ref{d:q-c-r}(3).
We see that $G^\qt$ is a quotient of the group of multiplicative type $Z(G)$,
and it follows from \cite[Theorem 12.23]{Milne-AG}
that $G^\qt$ is of multiplicative type itself.
Since $G$ is smooth, by \cite[Proposition 1.62]{Milne-AG}(b) so is $G^\qt$,
and (ii) follows.
\end{proof}

\begin{construction}\label{con:+}
Let $G$ be a quasi-connected reductive $k$-group and let $T^\sssp\subset G^\sss$
be a maximal torus.
Consider the group
\[ T=\cZ_G(T^\sssp).\]
\end{construction}

\begin{lemma}\label{l:t-product}
For $G$ and $T$ as in Construction \ref{con:+}, we have:

\begin{enumerate}
\item[\rm (i)] $G^\sss\cap T=T^\sssp$;
\item[\rm (ii)] $T=Z(G)\cdot T^\sssp$;
\item[\rm (iii)] $T$ fits into a short exact sequence
\,$1\to  T^\sssp\to T\to G^\qt\to 1$;
\item[\rm (iv)] $T$ is a quasi-torus.
\end{enumerate}
\end{lemma}

\begin{proof}
We have $G^\sss\cap T=G^\sss\cap \cZ_G(T^\sssp)=\cZ_{G^\sss}(T^\sssp)=T^\sssp$, which proves (i).

Clearly, $Z(G)\subseteq T$ and $T^\sssp\subseteq T$, whence $Z(G)\cdot T^\sssp\subseteq T$.
Conversely, let $t\in T(R)\subseteq G(R)$ for some $k$-algebra $R$.
Since $G=Z(G)\cdot G^\sss$ by Definition \ref{d:q-c-r}(3), we may write $t= z \cdot g^\sss$
for some faithfully flat homomorphism $R\into R'$ and some $z\in Z(G)(R')$,  $g^\sss\in G^\sss(R')$.
Since $z\in Z(G)(R')\subseteq T(R')$ and $t\in T(R)\subseteq T(R')$, we see that
$g^\sss\in T(R')\cap G^\sss(R')$,
and by (i) $g^\sss\in T^\sssp(R')$,
which proves (ii).

Since $T\supseteq Z(G)$,
it follows from Lemma \ref{l:qt}(i) and Fact \ref{f:W-ff}
that the restriction $q|_T\colon T\to G^\qt$ is faithfully flat.
Using (i), we obtain
$\ker q|_{T} =G^\sss\cap T= T^\sssp$.
Thus (iii) holds.

Since both $T^\sssp$ and $G^\qt$ are smooth, by \cite[Proposition 1.62(a)]{Milne-AG}
it follows from (iii) that $T$ is smooth.
By assertion (ii), the $k$-group $T$ is commutative.
Since both $T^\sssp$ and $G^\qt$ are of multiplicative type, and $T$ is commutative,
by \cite[Corollary 12.22]{Milne-AG} the $k$-group $T$ is of multiplicative type.
Thus $T$ is a quasi-torus, which proves (iv).
\end{proof}

\begin{definition}
A quasi-torus $T\subseteq G$ as in Construction \ref{con:+} is called a
{\em principal quasi-torus in $G$}.
\end{definition}

\begin{remark}
For a {\em connected} reductive $k$-group $G$,
a principal quasi-torus in $G$ is the same as a maximal torus in $G$.
\end{remark}

\begin{lemma}\label{l:l0-short-exact}
Consider a short exact sequence of quasi-connected reductive $k$-groups
\begin{equation}\label{e:SHG}
1\to S\to H\labelt\pi G\to 1
\end{equation}
where $S$ is a $k$-torus. Then $S$ is central in $H$.
\end{lemma}

\begin{proof}
From \eqref{e:SHG} we obtain a short exact sequence
of {\em connected} reductive $k$-groups
\begin{equation}\label{e:SHG-0}
1\to S\to H^0\labelt\pi G^0\to 1,
\end{equation}
from which we see that $S$ is central in $H^0$.
Hence the subgroup $H^\sss\subseteq H^0$
acts by conjugation on $S$ trivially.
Since $Z(H)$ also acts on $S$ trivially, and $H=Z(H)\cdot H^\sss$ by Definition \ref{d:q-c-r}(3),
$H$ acts by conjugation on $S$ trivially,
and we conclude that $S$ is central, as desired.
\end{proof}

\begin{construction}\label{con:short-exact}
Consider a  short exact sequence of quasi-connected reductive $k$-groups
\eqref{e:SHG} as in Lemma \ref{l:l0-short-exact}.
Then $\pi$ is faithfully flat (by the definition of exactness).
Since $\pi(H^0)=G^0$, we have
\begin{equation}\label{e:Hss-Gss}
\pi(H^\sss)=G^\sss.
\end{equation}
It follows from the exactness of \eqref{e:SHG-0} that $H^\sss\to G^\sss$ is a central isogeny.
We have
\begin{equation}\label{e:ZHss-ZGss}
\pi\big(Z(H^\sss)\big)=Z(G^\sss)\quad\ \text{and}
    \quad\ \pi^{-1}\big(Z(G^\sss)\big)\cap H^\sss = Z(H^\sss).
\end{equation}
Let $T_H^\sssp\subset H^\sss$ be a maximal torus, and consider the corresponding maximal torus $T_G^\sssp\subset G^\sss$.
Then
\begin{equation}\label{e:THss-TGss}
\pi(T_H^\sssp)=T_G^\sssp\quad\ \text{and}\quad\ \pi^{-1}(T_G^\sssp)\cap H^\sss=T_H^\sssp.
\end{equation}
Consider the principal quasi-tori
\[T_H=\cZ_H(T_H^\sssp)=Z(H)\cdot T_H^\sssp\subseteq H,
    \quad\ \  T_G=\cZ_G(T_G^\sssp)=Z(G)\cdot T_G^\sssp\subseteq G\]
(we use Lemma \ref{l:t-product}(ii)\hs).
\end{construction}

\begin{lemma}\label{l:short-exact}
For a short exact sequence of quasi-connected reductive $k$-groups
\eqref{e:SHG}  as in Lemma \ref{l:l0-short-exact}, and
for $T_H$, and $T_G$ as in Construction \ref{con:short-exact}, we have:
\begin{enumerate}
\item[\rm (i)] $\pi\big(Z(H)\big)=Z(G)$
and $\pi^{-1}\big(Z(G)\big)=Z(H)$;
\item[\rm (ii)] $\pi(T_H)=T_G$ and $\pi^{-1}(T_G)=T_H$.
\end{enumerate}
\end{lemma}

\begin{proof}
Assertion (i) follows from the inclusions
\begin{align}
\pi\big(Z(H)\big)&\subseteq Z(G),\label{e:ZHG}\\
\pi^{-1}\big(Z(G)\big)&\subseteq Z(H).\label{e:ZGH}
\end{align}

Inclusion \eqref{e:ZHG} follows from the assumption that $\pi$ is faithfully flat;
see Lemma \ref{l:preimage-center}.

We prove \eqref{e:ZGH}.
Let $z_\mG\in Z(G)(R)$ for some $k$-algebra $R$, and  assume that   $z_\mG=\pi(h)$ for some $h\in H(R)$.
Since $H=Z(H)\cdot H^\sss$ by Definition \ref{d:q-c-r}(3), we may write $h=z_\mH\cdot h^\sss$,
where $R\into R'$ is some faithfully flat homomorphism,
$z_\mH\in Z(H)(R')$ and $h^\sss\in H^\sss(R')$.
Since $\pi(h)\in Z(G)(R)$ by assumption, and $\pi(z_\mH)\in Z(G)(R')$ by \eqref{e:ZHG}, we have
$\pi(h^\sss)\in Z(G)(R')$, whence
\[\pi(h^\sss)\in Z(G)(R')\cap \pi(H^\sss(R'))\subseteq
   Z(G)(R')\cap G^\sss(R')=Z(G^\sss)(R')\]
(we use \eqref{e:Hss-Gss} and Corollary \ref{c:cZ_G}).
It follows that
\[h^\sss\in\pi^{-1}\big(Z(G^\sss)\big)(R')\cap H^\sss(R')=Z(H^\sss)(R')\subseteq Z(H)(R')\]
(we use \eqref{e:ZHss-ZGss}\hs),
and $h\in Z(H)(R')$, which proves \eqref{e:ZGH} and (i).

Assertion (ii) follows from the inclusions
\begin{align}
\pi(T_H)&\subseteq T_G\hs,\label{e:THG}\\
\pi^{-1}(T_G)&\subseteq T_H\hs.\label{e:TGH}
\end{align}

We prove \eqref{e:THG}. Let $t_\mH\in T_H(R)$ for some $k$-algebra $R$; then by Lemma \ref{l:t-product}(ii)
we have $t_\mH\in Z(H)(R')\cdot T_H^\sssp(R')$
for some faithfully flat homomorphism $R\into R'$.
We have
\[\pi(t_\mH)\in \pi\big(Z(H)(R')\big)\cdot \pi\big(T_H^\sssp(R')\big)
\subseteq Z(G)(R')\cdot T_G^\sssp(R')\subseteq T_G(R'),\]
(we use \eqref{e:ZHG} and \eqref{e:THss-TGss}\hs), whence  $\pi(t_\mH)\in G(R) \cap T_G(R')=T_G(R)$, as desired
(we use  Lemma \ref{l:elementary} below).

We prove \eqref{e:TGH}.
Let $t_G\in T_G(R)$, $t_G=\pi(h)$ with $h\in H(R)$ for some $k$-algebra $R$.
Since $H=Z(G)\cdot H^\sss$ by Definition \ref{d:q-c-r}(3), we may write
\[h=z_\mH\cdot h^\sss\quad\ \text{with}\ \, z_\mH\in Z(H)(R'),\ h^\sss\in H^\sss(R')\]
for some faithfully flat homomorphism $R\into R'$.
By construction we have
\[\pi(h^\sss)=\pi(z_\mH^{-1})\cdot \pi(h)=\pi(z_\mH^{-1})\cdot t_G\in \big(Z(G)(R')\cdot T_G(R)\big)\cap G^\sss(R')
\subseteq T_G(R')\cap G^\sss(R')\]
(we use \eqref{e:ZHG} and  \eqref{e:Hss-Gss}\hs).
By Lemma \ref{l:t-product}(i), we obtain that  $\pi(h^\sss)\in T_G^\sssp(R')$,
whence by \eqref{e:THss-TGss} we have
$h^\sss\in \big(\pi^{-1}(T_G^\sssp)\cap H^\sss\big)(R')= T_H^\sssp(R')$.
Thus $h\in Z(H)(R')\cdot  T_H^\sssp(R')\subseteq T_H(R')$,
whence $h\in T_H(R')\cap H(R)=T_H(R)$
by Lemma \ref{l:elementary} below.
This proves \eqref{e:TGH} and (ii).
\end{proof}

\begin{lemma}[elementary]
\label{l:elementary}
    Let $X$ be an affine $k$-scheme,  $Y\subseteq X$ be a closed $k$-subscheme,
    and  $\phi\colon R \into R'$ be an injective  homomorphism of $k$-algebras. Then
    $Y(R')\cap X(R)=Y(R)$.
\end{lemma}

\begin{proof} It suffices to show that $Y(R') \cap X(R)\subseteq Y(R)$.
Write $X = \operatorname{Spec} A$ and $Y = \operatorname{Spec} A/I$,
where $A$ is a $k$-algebra and $I\subseteq A$ an ideal.
Then an $R'$-point  $x \in Y(R') \cap X(R)$ is a homomorphism $x \colon A \to R$
such that  $\phi(x(I))= \{0\}\subset R'$.
Since $\phi\colon R \to R'$ is injective, we have $x(I) = \{0\}\subset R$,
that is, $x\in Y(R)$.
\end{proof}

\begin{lemma}\label{l:some-any}
Let $\vk\colon G_1\to G_2$ be a homomorphism of quasi-connected reductive $k$-groups.
Assume that for {\emm some} principal quasi-torus  $T_1\subseteq G_1$,
its image $\vk(T_1)$ is contained in some principal quasi-torus $T_2\subseteq G_2$.
Then this holds for {\emm every} principal quasi-torus in $G_1$.
\end{lemma}

\begin{proof}
Write $H=\cZ_{G_2^\sss}(\vk(T_1))$.
Since $\vk(T_1)\subseteq T_2$ and $T_2$ is commutative,
we have $H\supseteq T_2^\sssp$, where $T_2^\sssp\coloneqq T_2\cap G_2^\sss$
is a maximal torus in $G_2^\sss$.
Hence, any maximal torus of $H_\kbar$
is a maximal torus of $G^\sss_{2,\kbar}$\hs.

Let $T_1'\subseteq G_1$ be any principal quasi-torus.
Then $T_1'=\cZ_{G_1}(T_1^{\prime\,\sssp})$,
where $T_1^{\prime\,\sssp}$ is a maximal torus in $G_1^\sss$.
All maximal tori in $G_1^\sss$ are conjugate over $\kbar$;
see \cite[Theorem 6.4.1]{Springer-LAG}.
Hence, there exists an element $g\in G_1^\sss(\kbar)$ such that
\begin{equation}\label{e:max-tori}
T_{1,\kbar}^{\prime\, \sssp}=g\cdot T_{1,\kbar}^\sssp\cdot g^{-1}.
\end{equation}
Set $H'=\cZ_{G_2^\sss}\big(\vk(T_1')\big)$.
Then $H'$ is a $k$-subgroup of $G_2^\sss$.
It follows from \eqref{e:max-tori}
that we have $H'_\kbar=\vk(g)\cdot H_\kbar\cdot\vk(g)^{-1}$,
and therefore every maximal torus of $H'_\kbar$
is a maximal torus of $G^\sss_{2,\kbar}$\hs.

On the other hand, the connected $\kbar$-group $H^{\prime\,0}_\kbar$  contains a maximal torus
defined over $k$, that is, coming from a $k$-torus $T_{H^{\prime\,0}}$ of $H^{\prime\,0}$;
see \cite[Theorem 13.3.6 and Remark 13.3.7]{Springer-LAG}.
Write $T_2^{\prime\,\sssp}=T_{H^{\prime\,0}}$; then $T_2^{\prime\,\sssp}$ is a maximal torus of $G_2^\sss$.
Take $T'_2=\cZ_{G_2}(T_2^{\prime\,\sssp})$; then $T'_2$ is a principal quasi-torus in $G_2$.
Since $\vk(T_1')$ commutes with $H'$, it commutes with $T_2^{\prime\,\sssp}=T_{H^{\prime\,0}}\subseteq H'$.
Thus $\vk(T_1')\subseteq \cZ_{G_2}(T_2^{\prime\,\sssp})= T_2'$, as desired.
\end{proof}

\begin{definition}
A homomorphism $\vk \colon G_1\to G_2$ of quasi-connected reductive $k$-groups
is called {\em principal} if for some (and hence every)
principal quasi-torus $T_1\subseteq G_1$
there exists a principal quasi-torus $T_2\subseteq G_2$
such that   $\vk(T_1)\subseteq T_2$.
\end{definition}

By Lemma \ref{l:short-exact}(ii), a faithfully flat homomorphism $\pi$ as in Lemma  \ref{l:l0-short-exact}
is principal.

\begin{remark}
Every homomorphism $\vk \colon G_1\to G_2$ of quasi-connected reductive $k$-groups
with  $G_1$ {\em connected} is principal.
Indeed, let $T_1^\sssp\subset G_1^\sss$ be a maximal torus, and write
\[T_1^0=\cZ_{G_1^0}(T_1^\sssp)\subseteq G_1^0,\quad\  T_1=\cZ_{G_1}(T_1^\sssp)\subseteq G_1.\]
Since $G_1$ is connected, we have $T_1=T_1^0$.
The torus $\vk(T_1^0)\subseteq G_2^0$ is contained in some maximal torus $T_2^0$ of $G_2^0$.
Set
\[T_2^\sssp=T_2^0\cap G_2^\sss, \quad\ T_2=\cZ_{G_2}(T_2^\sssp).\]
Then $\vk(T_1)=\vk(T_1^0)\subseteq T_2^0\subseteq T_2$, and so $\vk$ is principal, as desired.
\end{remark}

\begin{example}\label{ex:PU2}
Let $k=\R$, $H={\rm U}_2$, $G={\rm PU}_2\coloneqq H/Z(H)$.
Let $\pi\colon H\to G$ denote the canonical homomorphism.
Consider the elements $\tilde b_1,\tilde b_2\in G^\ssc(\R)=H^\sss(\R)$ given by the following matrices:
\[ \tilde b_1=\SmallMatrix{0 &1\\ -1&0},\quad \tilde b_2=\SmallMatrix{i &0\\0 &-i}.\]
Then
\[\tilde b_1^2=-1, \quad \tilde b_2^2=-1,\quad\tilde b_1 \tilde b_2=-\tilde b_2\tilde b_1.\]
Let $b_1$ and $b_2$ denote the images in $G(\R)$ of $\tilde b_1$ and $\tilde b_2$, respectively,
and let $B\subset G$ denote the subgroup generated by $b_1$ and $b_2$.
Then
\[ b_1^2=1,\quad b_2^2=1,\quad b_1 b_2=b_2 b_1.\]
We see that $B$ is a (non-cyclic)  group of order 4, hence a quasi-torus.
The quasi-torus $B\subset G$ is not contained in any principal quasi-torus $T_G\subseteq G$,
because by Lemma \ref{l:short-exact}(ii), $\pi^{-1}(T_G)$ is a principal quasi-torus $T_H$ of $H$,
hence a commutative $\R$-group, while $\pi^{-1}(B)$ is not commutative.
\end{example}

\begin{remark}\label{ex:PU2-not-principal}
For $B$ and $G$ as in Examples \ref{ex:intro} and \ref{ex:PU2},
the inclusion homomorphism of quasi-connected reductive $\R$-groups $\iota\colon B\into G$
is {\em not} principal.
Indeed,  $B$ is a principal quasi-torus in $B$, but $\iota(B)=B$
is not contained in a principal quasi-torus of $G$.
\end{remark}

\section{A Picard crossed module from a quasi-connected reductive group}
\label{s:Picard}

Let $G$ be a quasi-connected reductive group over a field $k$.
We write $Z^\sssp=Z(G^\sss)$ and $Z^\sscp=Z(G^\ssc)$.
By Corollary \ref{c:cZ_G} we have
\begin{equation}\label{e:Z^sssp}
Z^\sssp=Z(G)\cap G^\sss.
\end{equation}
It is well known that $Z^\sscp=\rho^{-1}(Z^\sssp)$.
The homomorphisms
$G^\ssc\onto G^\sss\into G$ induce canonical homomorphisms
\begin{equation} \label{e:ad}
G^\ssc/Z^\sscp\isoto G^\sss/Z^\sssp\isoto G/Z(G),
\end{equation}
which are isomorphisms.
Indeed, for the first homomorphism, this is well known.
For the second one, the injectivity follows from \eqref{e:Z^sssp}.
Since by Definition \ref{d:q-c-r}(3) we have $G=Z(G)\cdot G^\sss$,
by Lemma \ref{l:AN} the composite homomorphism $G^\sss\into G\onto G/Z(G)$ is faithfully flat,
and by Fact \ref{f:W-ff} the induced homomorphism $G^\sss/Z^\sssp\to G/Z(G)$
is faithfully flat as well.

\begin{construction}\label{cons-braiding}
Following an idea of Deligne  \cite[\S2.0.2]{Deligne},
we observe that the morphism of $k$-varieties given by
the commutator map
\[[\tdash,\tdash]\colon G^\ssc\times_k G^\ssc\to G^\ssc,\  (x,y)\mapsto [x,y]\coloneqq xyx^{-1}y^{-1}.\]
factors via a morphism of $k$-varieties
\[(G^\ssc/Z^\sscp)\times (G^\ssc/Z^\sscp)\to G^\ssc.\]
Composing the latter morphism with the  homomorphism $G\to G/Z(G)\cong G^\ssc/Z^\sscp$ obtained using  \eqref{e:ad},
we obtain a morphism of $k$-varieties
\begin{equation}\label{e:cons-br}
\br\colon G\times_k G\to G^\ssc
\end{equation}
lifting
\[ [\tdash,\tdash]\colon\, G\times_k G\to G,\quad\ (g_1,g_2)\mapsto [g_1,g_2]\coloneqq g_1\hs g_2\hs g_1^{-1} g_2^{-1}\hs.\]

Let $g_1,g_2\in G(R)$ for some $k$-algebra $R$.
Since $G=Z(G)\cdot G^\sss$ by Definition \ref{d:q-c-r}(3),
and the homomorphism $G^\ssc\to G^\sss$ is faithfully flat, we can write
\[g_i=z_i\cdot \rho(s_i)\quad\  \text{with  $z_i\in Z(G)(R')$, \, $s_i\in G^\ssc(R')$\ \ for $i=1,2$}\]
for some faithfully flat  homomorphism $R\into R'$ (we use Fact \ref{f:W-ff}).
For each $i$, the images of $g_i$, $\rho(s_i)$, and $s_i$  in
\[ \big(G/Z(G)\big)(R')\cong(G^\sss/Z^\sssp)(R')\cong (G^\ssc/Z^\sscp)(R') \]
coincide.
It follows from the definition of $\br$ that
\begin{equation}\label{e:br-br}
\{g_1,g_2\}=[s_1,s_2].
\end{equation}
If $G^\sss$ is {\em simply connected}, then $\rho(s_i)=s_i$, and we obtain that
\begin{equation}\label{e:br-sc}
\{g_1,g_2\}=[s_1,s_2]=\big[\rho(s_1),\rho(s_2)\big]=[g_1,g_2].
\end{equation}
\end{construction}

\begin{proposition}\label{p:br}
The morphism $\br$ of \eqref{e:cons-br} is a {\emm braiding} of the crossed module $(G^\ssc,G,\rho,\theta)$.
Namely, the following equalities hold (cf. \cite[Definition 1.7]{CFFL})
\begin{align}
\rho\big(\{g_1,g_2\}\big)&=[g_1,g_2],\tag{\sf Br1}\label{e:Br1}\\
\{\rho(s'),g\}&=s'\cdot{\ha}^\upg (s')^{-1},\tag{\sf Br2}\label{e:Br2}\\
\{g,\rho(s')\}&={\ha}^\upg s'\cdot (s')^{-1},\tag{\sf Br3}\label{e:Br3}\\
\{g_1,\hs g_2 g_3\}&=\{g_1,g_2\}\cdot {\ha}^{g_2}\{g_1, g_3\},\tag{\sf Br4}\label{e:Br4}\\
\{g_1 g_2,\hs g_3\}&={\ha}^{g_1} \{g_2,g_3\}\cdot\{g_1,g_3\}\tag{\sf Br5}\label{e:Br5}
\end{align}
for every $k$-algebra $R$ and
for all $s'\in G^\ssc(R),\ g, g_1,g_2,g_3\in G(R)$.
Moreover, the braiding $\br$ is {\emm symmetric}, that is,
\begin{align}
\{g_1,g_2\}\cdot \{g_2,g_1\}&=1 \tag{\sf Sym}\label{e:Sym}
\end{align}
for all $g_1,g_2\in G(R)$.
Furthermore, the symmetric braiding $\br$ is {\emm Picard}, that is,
\begin{align}
\{g,g\}&=1\tag{\sf Pic}\label{e:Pic}
\end{align}
for all $g\in G(R)$.
\end{proposition}

\begin{proof}
We write
\[ g=z\cdot \rho(s),\quad g_i=z_i\cdot \rho(s_i)\quad \text{for}\ g,g_i\in G(R)\subseteq G(R'),
   \, z,z_i\in Z_G(R'),\, s,s_i\in G^\ssc(R'), \, i=1,2,3\]
for some faithfully flat homomorphism $R\into R'$.
Then by \eqref{e:br-br} the desired equalities become
\begin{align}
\rho[s_1,s_2]&=[g_1,g_2],\tag{\sf Br1$'$}\label{e:Br1'}\\
[s',s]&=s'\cdot s\, s^{\prime\hs -1} s^{-1},\tag{\sf Br2$'$}\label{e:Br2'}\\
[s,s']&=s\hs s' s^{-1}\cdot s^{\prime\hs -1},\tag{\sf Br3$'$}\label{e:Br3'}\\
[s_1,\hs s_2 s_3]&=[s_1,s_2]\cdot s_2[s_1, s_3]s_2^{-1},\tag{\sf Br4$'$}\label{e:Br4'}\\
[s_1 s_2,\hs s_3]&=s_1 [s_2,s_3] s_1^{-1}\cdot[s_1,s_3],\tag{\sf Br5$'$}\label{e:Br5'}\\
[s_1,s_2]&\cdot [s_2,s_1]=1, \tag{\sf Sym$'$}\label{e:Sym'}\\
[s,s]&=1,\tag{\sf Pic$'$}\label{e:Pic'}
\end{align}
which are obvious.
\end{proof}

\begin{definition}
For a quasi-connected reductive $k$-group $G$,
we consider the following  5-tuple: $(G^\ssc, G, \rho, \theta,\br)$,
where $(G^\ssc, G, \rho, \theta)=\CrM(G)$ is our crossed module and $\br$ is the Picard braiding
of Construction \ref{cons-braiding}. We say that $(G^\ssc, G, \rho, \theta,\br)$
is the {\em Picard crossed module obtained from} $G$, and write
\[ \PCrM(G)=(G^\ssc, G, \rho, \theta,\br).\]
\end{definition}

\begin{remark}\label{r:not-braiding}
By Proposition \ref{p:functor-CrM}, a homomorphism of quasi-connected reductive $k$-groups
$\vk\colon G_1\to G_2$ induces a morphism of crossed modules $\CrM(G_1)\to \CrM(G_2)$.
However, in general it does not induce a morphism of {\em Picard} crossed modules $\PCrM(G_1)\to \PCrM(G_2)$.
Indeed, in Example \ref{ex:PU2},
for the inclusion homomorphism of quasi-connected reductive $\R$-groups $\iota\colon B\into G$,
for $b_1,b_2\in B(\R)\subset G(\R)$ we have $\{b_1,b_2\}_B=1$ (because $B^\ssc=1$),
but $\big\{\iota(b_1),\iota(b_2)\big\}_G=[\tilde b_1, \tilde b_2]=-1\in G^\ssc(\R)={\rm SU}_2$\hs.
We see that $\big\{\iota(b_1),\iota(b_2)\big\}_G\neq \iota^\ssc\big(\{b_1,b_2\}_B\big)$.
\end{remark}

\begin{definition}
\label{d:R}
Let $G$ be a quasi-connected reductive group over a field $k$.
A {\em $t$-extension of $G$} is an extension
\begin{equation}\label{e:t-ext}
1\to S\to H\labelt\pi G\to 1
\end{equation}
where $S$ is a $k$-torus and $H$ is a quasi-connected reductive $k$-group
such that $H^\sss$ is simply connected.
\end{definition}

Note that by Lemma \ref{l:l0-short-exact} every $t$-extension is central.

Definition \ref{d:R} generalizes \cite[Definition 2.1]{BGA}
of $t$-extension of a connected reductive $k$-group.
Note that $z$-extensions of Kottwitz \cite{Kottwitz-82}
and flasque resolutions of Colliot-Th\'el\`ene \cite{CT}
are special cases of $t$-extensions in the case when $G$ is connected.

\begin{lemma}\label{l:cR}
Every quasi-connected reductive $k$-group $G$ admits a $t$-extension.
\end{lemma}

\begin{proof}
According to Definition \ref{d:Labesse},
we may write $G=\ker[G_1\twoheadrightarrow T_0]$
where $G_1$ is a {\em connected} reductive $k$-group, and $T_0$ is a $k$-torus.
There exists a $t$-extension of $G_1$
\begin{equation*}
1\to S\to H_1\labelt{\pi} G_1\to 1;
\end{equation*}
see \cite[Proposition-Definition 3.1]{CT} or \cite[Proposition 2.2]{BGA}.
Set $H=\pi^{-1}(G)\subseteq H_1$.
Then
\begin{equation*}
1\to S\to H\to G\to 1
\end{equation*}
is a desired $t$-extension of $G$.
\end{proof}

Let $\pi\colon H\to G$ be a $t$-extension and let $g_1,g_2\in G(R)$ for some $k$-algebra $R$.
For $i=1,2$, using Fact \ref{f:W-ff} we write $g_i=\pi(h_i)$
for some faithfully flat homomorphism $R\into R'$ and  $h_i\in H(R')$.
Using the equality $G=Z(G)\cdot G^\sss$ of Definition \ref{d:q-c-r}(3)
and using Fact \ref{f:W-ff} applied to the faithfully flat homomorphism $H^\sss=G^\ssc\to G^\sss$,
we may write
\[ h_i=z_{i,\mH}\cdot s_{i}\quad\ \text{with}\ \, z_{i,\mH}\in Z(H)(R''),\ s_{i}\in H^\sss(R'')=G^\ssc(R'')\]
for some faithfully flat homomorphism $R'\into R''$.
Then
\[g_i=\pi(h_i)=\pi(z_{i,\mH})\cdot \pi(s_{i})=\pi(z_{i,\mH})\cdot \rho(s_{i})\]
where  $s_{i}\in G^\ssc(R'')= H^\sss(R'')$
and by Lemma  \ref{l:short-exact}(i)
we have $\pi(z_{i,\mH})\in Z(G)(R'')$.
Therefore, by \eqref{e:br-br} and \eqref{e:br-sc}
we have
\begin{equation}\label{e:br-H}
\{g_1,g_2\}=[s_1,s_2]=[h_1,h_2].
\end{equation}

\begin{definition}\label{def:r-of-morph}
Let $\vk\colon G\to G'$ be a homomorphism of quasi-connected reductive $k$-groups.
A {\em $t$-extension of $\vk$}
is a commutative diagram of quasi-connected reductive $k$-groups
\begin{equation}\label{e:t-vk}
\begin{aligned}
\xymatrix{
H\ar[d]_-\pi\ar[r]^-\lambda  & H'\ar[d]^-{\pi'}\\
G\ar[r]^-\vk   & G'
}
\end{aligned}
\end{equation}
for which the vertical arrows $H\to G$ and $H'\to G'$ are $t$-extensions.
\end{definition}

\begin{proposition}\label{p:principal-induces}
A  homomorphism $\vk\colon G\to G'$ of quasi-connected reductive $k$-groups that
{\emm admits a $t$-extension} preserves the Picard braiding and thus
induces a morphism of braided crossed modules $\vk_*\colon \PCrM(G)\to\PCrM(G')$
(with the evident definition of morphisms of braided crossed modules).
\end{proposition}

\begin{proof}
By  Proposition \ref{p:functor-CrM} the pair $(\vk^\ssc,\vk)$ is a morphism of crossed modules of $k$-groups
$\CrM(G)\to\CrM(G')$.
It remains to show that $(\vk^\ssc,\vk)$ preserves the canonical braiding.
We must show that for every $k$-algebra $R$ and all $g_1,g_2\in G(R)$,
with  evident notations we have
\[\vk^\ssc\big(\hs\{g_1,g_2\}_G\hs\big)=\big\{\hs\vk(g_1), \vk(g_2)\hs\big\}_{\hm G'}\hs.\]
By assumption, we have a commutative diagram \eqref{e:t-vk} in which the vertical arrows are $t$-extensions.
Write
$g_1=\pi(h_1)$ and  $g_2=\pi(h_2)$
with $h_1,h_2\in H(R')$ for some faithfully flat homomorphism $R\into R'$.
We may identify $G^\ssc=H^\sss$ and similarly for $G'$; then
$\{g_1,g_2\}\in H^\sss(R')$.
By  \eqref{e:br-H} we have
\[ \{g_1,g_2\}= [h_1,h_2],\]
and similarly for the $t$-extension $H'$ of $G'$.
Thus it remains to prove the formula
\begin{align*}
&\lambda\big([h_1,h_2]\big)=\big[\lambda(h_1),\lambda(h_2)\big],
\end{align*}
which follows immediately from the fact that $\lambda$ is a homomorphism.
\end{proof}

\begin{theorem}\label{t:t-ext-principal}
Every {\emm principal} homomorphism of quasi-connected reductive $k$-groups admits a $t$-extension.
\end{theorem}

Theorem \ref{t:t-ext-principal} will be proved in Section \ref{s:t-ext-principal}.
Note that for a homomorphism of {\em connected} reductive $k$-groups,
this is \cite[Lemma 3.3]{BGA}.

\begin{corollary}\label{c:principal-induces}
A {\emm principal} homomorphism $\vk\colon G\to G'$ of quasi-connected reductive $k$-groups
induces a morphism of braided crossed modules $\vk_*\colon \PCrM(G)\to\PCrM(G')$.
\end{corollary}

\begin{proof}
The corollary follows from Theorem \ref{t:t-ext-principal}
and Proposition \ref{p:principal-induces}.
\end{proof}

\begin{remark}\label{r:not-admit-t}
We show that the non-principal homomorphism $\iota\colon B\into G$
of Examples \ref{ex:intro} and \ref{ex:PU2}
does not admit a $t$-extension.
Indeed, let $\pi_B\colon A\to B$
be any $t$-extension of $B$.
By Lemma \ref{l:short-exact}(i) we have
\[Z(A)(\C)=\pi_B^{-1}\big(Z(B)(\C)\big)=\pi_B^{-1}\big(B(\C)\big)=A(\C),\]
hence the group $A(\C)$ is commutative.
Choose preimages $a_1,a_2\in A(\C)$ of $b_1,b_2\in B(\R)$, respectively; then  $a_1$ and $a_2$ commute.

On the other hand, let $\pi\colon H\to G$ be any $t$-extension of $G$,
and let $\lambda\colon A\to H$ be any morphism of $\R$-varieties
that fits into the commutative diagram
\[
\xymatrix{
A\ar[r]^-\lambda\ar[d]_{\pi_B}  &H\ar[d]^-{\pi} \\
B\ar[r]^-\iota                    &G.
}
\]
Let $h_j=\lambda(a_j)\in H(\C)$ for $j=1,2$. Then
\[\pi(h_j)=\pi(\lambda(a_j))=\iota(\pi_B(a_j))=\iota(b_j)=b_j\in G(\R).\]
Consider the preimages $\tilde b_j\in G^\ssc(\R)=H^\sss(\R)\subseteq H(\R)$ of $b_j$
of Example \ref{ex:PU2}.
Then
\begin{equation}\label{e:hi-xi-bi}
h_j=z_j\cdot \tilde b_j\ \  \text{for some}\ \, z_j\in S(\C)\subseteq Z(H)(\C)
\end{equation}
where $S=\ker[H\to G]$ (we use Lemma \ref{l:l0-short-exact}).
Since $\tilde b_1\tilde b_2=-\tilde b_2\tilde b_1\ne \tilde b_1\tilde b_2$, whereas $z_1,z_2$ are central,
we see from \eqref{e:hi-xi-bi} that $h_1$ and $h_2$ do not commute.
Since $a_1$ and $a_2$ commute, whereas their images $h_1=\lambda(a_1)$ and $h_2=\lambda(a_2)$ do not commute,
we see that  $\lambda$ cannot be a homomorphism of $\R$-groups,
which shows that our homomorphism $\iota$ does not admit a $t$-extension.
\end{remark}

\section{A \texorpdfstring{$t$}{t}-extension of a principal homomorphism}
\label{s:t-ext-principal}

In this section we prove Theorem \ref{t:t-ext-principal}.
The following proposition is inspired by \cite[Lemma 2.4.4]{Kottwitz-Duke}.

\begin{proposition}\label{p:H3}
Let $\vk\colon G_1\to G_2$ be a {\emm principal} homomorphism
of quasi-connected reductive $k$-groups, and let
\begin{align*}
&1\to S_1\labelt{\ha} H_1\labelt{\pi_1} G_1\to 1,\\
&1\to S_2\labelt{\ha} H_2\labelt{\pi_2} G_2\to 1
\end{align*}
be any $t$-extensions.
Then there exists a commutative diagram
\begin{equation}\label{e:H3}
\begin{aligned}
\xymatrix{
1 \ar[r] &S_1\ar[r]               &H_1\ar[r]^-{\pi_1}            &G_1\ar[r]              &1\\
1 \ar[r] &S_3\ar[r] \ar[u]\ar[d]    &H_3\ar[r]^-{\pi_3}\ar[u]\ar[d]^-\lambda
                                           &G_1\ar[r]\ar[u]_{\rm id}\ar[d]^-\vk  &1\\
1 \ar[r] &S_2\ar[r]               &H_2\ar[r]^-{\pi_2}            &G_2\ar[r]              &1
}
\end{aligned}
\end{equation}
in which the middle row is a $t$-extension of $G_1$.
\end{proposition}

\begin{proof}
We follow \cite[Proof of Proposition 2.8]{BGA},
where the case of {\em connected} $G_1$ and $G_2$ was considered.
We take $H_3$ to be the fiber product of $H_1$ and $H_2$ over $G_2$\hs,
and denote by $\lambda\colon H_3\to H_2$ the projection homomorphism.
The faithful flatness  of $\pi_2\colon H_2\to G_2$ implies that of $H_3\to H_1$,
which in turn implies the faithful flatness of $\pi_3\colon H_3\to G_1$
and the surjectivity on the identity components $H_3^0\to G_1^0$.
The kernel $S_3$ of $\pi_3\colon H_3\to G_1$ is the product of $S_1$ and $S_2$,
hence a torus in $H_3$.
Since  by Lemma \ref{l:l0-short-exact} the $k$-torus
$S_1$ is central in $H_1$ and the $k$-torus  $S_2$ is central in $H_2$,
we see that the $k$-torus $S_3$ is central in $H_3$.
We obtain short exact sequences
\begin{align}
&1\to S_3\labelt{\ha} H_3\labelt{\pi_3} G_1\to 1,\label{e:S3H3G1} \\
&1\to S_3\labelt{\ha} H_3^0\labelt{\pi_3} G_1^0\to 1.\label{e:S3H30G10}
\end{align}
By \cite[Proposition 1.62]{Milne-AG}(a)
it follows from \eqref{e:S3H3G1} that $H_3$ is smooth.
By \cite[Lemma 2.6]{DH22}, it follows from \eqref{e:S3H30G10} that $H_3^0$ is reductive.
Moreover, the faithfully flat homomorphism $H_3^0\to G_1^0$
induces a faithfully flat homomorphism $H_3^\sss\to G_1^\sss$
with (central) kernel $S_3\cap H_3^\sss$.
Since $H_3^\sss$ is semisimple,
the last homomorphism is in fact  a central isogeny.
It follows that the homomorphism $H_3^\sss\to H_1^\sss\cong G_1^\ssc$ is a central isogeny as well,
whence $H_3^\sss\cong G_1^\ssc$ is simply connected.
By Lemma \ref{p:H3-qcr} below the $k$-group $H_3$ is quasi-connected reductive,
and therefore the middle row of diagram \eqref{e:H3} is a $t$-extension of $G_1$\hs,
which completes the proof of the proposition (modulo Lemma \ref{p:H3-qcr}).
\end{proof}

\begin{lemma}\label{p:H3-qcr}
The $k$-group $H_3=H_1\times_{G_2} H_2$ in the proof of Proposition \ref{p:H3}
is quasi-connected reductive.
\end{lemma}

In order to prove Lemma \ref{p:H3-qcr}, we need a few lemmas.

\begin{lemma}\label{l:*}
Let $\vk\colon G_1\to G_2$ be a principal homomorphism
of quasi-connected reductive $k$-groups,
and let $\pi_2\colon H_2\to G_2$ be a $t$-extension.
Then the $k$-subgroup $\pi_2^{-1}\big(\vk(Z(G_1))\big)\subseteq H_2$ is commutative.
\end{lemma}

\begin{proof}
Let $T_1\subseteq G_1$ be a principal quasi-torus; then $Z(G_1)\subseteq T_1$.
Since $\vk$ is principal, there exists a principal quasi-torus $T_2\subseteq G_2$
such that $\vk(T_1)\subseteq T_2$. Then $\vk(Z(G_1))\subseteq T_2$
and $\pi_2^{-1}\big(\vk(Z(G_1))\big)\subseteq \pi_2^{-1}(T_2)$.
By Lemma \ref{l:short-exact}(ii) the  preimage  $\pi_2^{-1}(T_2)$
is a principal quasi-torus in $H_2$, hence a commutative $k$-group,
and the lemma follows.
\end{proof}

\begin{lemma}\label{l:Z3-mult-type}
The $k$-group $Z_3\coloneqq\pi_3^{-1}(Z(G_1))\subseteq H_3$ is a $k$-group of multiplicative type.
\end{lemma}

\begin{proof}
Let $R$ be a  commutative unital $k$-algebra, and let $h_3,h_3'\in Z_3(R)\subseteq H_3(R)$. We may write
\[h_3=(h_1,h_2),\, h'_3=(h'_1,h'_2)\in \big(H_1\times_{G_2} H_2\big)(R)\quad\text{where}\ \, h_1,h_1'\in H_1(R),\   h_2,h_2'\in H_2(R).\]
Then
\[\pi_2(h_2)=\vk(\pi_1(h_1))\in G_2(R),\quad \pi_2(h'_2)=\vk(\pi_1(h'_1))\in G_2(R).\]
Moreover, since  $h_3,h_3'\in Z_3(R)=\pi_3^{-1}\big(Z(G_1)\big)(R)$,
we have $h_1,h_1'\in\pi_1^{-1}\big(Z(G_1)\big)(R)=Z(H_1)(R)$ (\hs we use Lemma \ref{l:short-exact}(i)\hs),
and therefore the elements  $h_1$ and $h_1'$ commute.
Moreover, $h_2,h_2'\in \pi_2^{-1}\big(\vk(Z(G_1))\big)(R)$, and  since $\vk$ is {\em principal},
by Lemma \ref{l:*} the elements $h_2$ and $h_2'$ commute
(this is the only place in the proof of Theorem \ref{t:t-ext-principal}
where we use the assumption that $\vk$ is principal).
We conclude that $h_3$ and $h_3'$ commute as well, and
thus the $k$-group $Z_3$ is commutative.
Since the commutative $k$-group $Z_3$
is an extension of the group of multiplicative type $Z(G_1)$
by the torus $S_3\cong S_1\times_k S_2$ (which is also of multiplicative type),
by \cite[Corollary 12.22]{Milne-AG} the $k$-group $Z_3$ is of multiplicative type, as desired.
\end{proof}

\begin{corollary}\label{c:mult-type}
The center $Z(H_3)$ is a $k$-group of multiplicative type.
\end{corollary}

\begin{proof}
Indeed, by Lemma \ref{l:preimage-center},
$Z(H_3)\subseteq\pi_3^{-1}\big(Z(G_1)\big)\eqqcolon Z_3$\hs,
and by Lemma \ref{l:Z3-mult-type}
the $k$-group $Z_3$ is of multiplicative type.
Now the corollary follows from \cite[Theorem 12.23]{Milne-AG}.
\end{proof}

\begin{lemma}\label{l:commutes}
The subgroup $Z_3$ of $H_3$ commutes with $H_3^\sss$.
\end{lemma}
\begin{proof}
The identity component $H_3^0$ of $H_3$ is reductive.
For any $k$-algebra $R$,
the subgroup  $Z_3(R)$ of $H_3(R)$ acts on the left on the $R$-group scheme
$(H_{3}^\sss)_R$  by conjugation,
and this action is compatible with the action by conjugation of $Z(G_1)(R)$ on $(G_1^\sss)_R$\hs.
Since the latter action is trivial, we see that the action of $Z_3(R)$
on $(H_{3}^\sss)_R$ gives a homomorphism
\begin{equation}\label{e:AutH3-AutG1}
Z_3(R)\,\to\, \ker\Big[ \Aut\big(\hs (H^\sss_3)_R, (G^\sss_1)_R\big)
\,\to\,\Aut\,(G^\sss_1)_R\hs\Big].
\end{equation}
By Lemma \ref{l:uni_cov_bc}, the base change  $(H^\sss_3)_R\to (G^\sss_1)_R$
of the universal covering $H^\sss_3\to G^\sss_1$ is again a universal covering.
Now it follows from the uniqueness in Proposition \ref{p:GA}(i)
that the kernel in \eqref{e:AutH3-AutG1} is trivial.
Therefore, the action by conjugation of $Z_3(R)$ on $(H_3^\sss)_R$ is trivial,
and hence $Z_3(R)$ commutes with $H_3^\sss(R)$, as desired.
\end{proof}

\begin{lemma}\label{l:Z3H3ss}
$H_3=Z_3\cdot H_3^\sss$.
\end{lemma}

\begin{proof}
Clearly, $Z_3\cdot H_3^\sss\subseteq H_3$.
We show that $H_3= Z_3\cdot H_3^\sss$.
Since the $k$-group $G_1$ is quasi-connected reductive, by Definition \ref{d:q-c-r}(3)
we have $$G_1=Z(G_1)\cdot G_1^\sss.$$
We have  $\pi_3\big(H_3^\sss\big)=G_1^\sss$ and
by construction we have $\pi_3^{-1}\big(Z(G_1)\big)=Z_3$.
Let $h_3\in H_3(R)$ for some $k$-algebra $R$.
Then $\pi_3(h_3)=z_1\cdot g^\sss_1$ where $z_1\in Z(G_1)(R')$
and $g^\sss_1\in G_1^\sss(R')$ for some faithfully flat homomorphism $R\into R'$.
Let $h^\sss_3\in H_3^\sss(R'')$ be a preimage of $g^\sss_1$
for some faithfully flat homomorphism $R'\into R''$.
Then $\pi_3(h_3\cdot (h^\sss_3)^{-1})=z_1\in Z(G_1)(R'')$,
whence
\[z_3\coloneqq h_3 \cdot (h^\sss_3)^{-1}\in Z_3(R'').\]
Thus
\[ h_3=z_3\cdot h^\sss_3\in Z_3(R'')\cdot H_3^\sss(R''),\]
which means that $H_3= Z_3\cdot H_3^\sss$, as desired.
\end{proof}

\begin{lemma}\label{l:Z3=Z(H_3)}
$Z_3=Z(H_3)$.
\end{lemma}

\begin{proof}
Let $x\in Z_3(R)$ and $h\in H_3(R')$ where $R$ is a $k$-algebra and $R'$ is an $R$-algebra.
It follows from  Lemma \ref{l:Z3H3ss}  that
there exist a faithfully flat homomorphism $R'\into R''$ and $z\in Z_3(R'')$,  $h^\sss\in H_3^\sss(R'')$
such that $h = z\cdot h^\sss$.
Since by Lemma \ref {l:Z3-mult-type} the $k$-group $Z_3$ is commutative,
$x$ commutes with $z$.
Since by Lemma \ref{l:commutes} $Z_3$ commutes with $H_3^\sss$,
the element $x$ commutes with  $h^\sss$.
Thus $x$ commutes with $h$, which
means that $Z_3\subseteq Z(H_3)$.
Conversely, by Lemma \ref{l:preimage-center} we have
$Z(H_3)\subseteq \pi_3^{-1}\big(Z(G_1)\big)\eqqcolon Z_3$\hs,
and the lemma follows.
\end{proof}

\begin{proof}[Proof of Lemma \ref{p:H3-qcr}]
We have: the $k$-group $H_3$ is smooth, its identity component $H_3^0$ is reductive,
by Corollary \ref{c:mult-type} the center $Z(H_3)$ is  of multiplicative type,
and by Lemmas \ref{l:Z3H3ss} and \ref{l:Z3=Z(H_3)} we have
\[H_3=Z_3\cdot H_3^\sss=Z(H_3)\cdot H_3^\sss.\]
Thus, by Definition \ref{d:q-c-r} the $k$-group $H_3$ is quasi-connected reductive,
which completes the proofs of Lemma \ref{p:H3-qcr} and of Proposition \ref{p:H3}.
\end{proof}

\begin{proof}[Proof of Theorem \ref{t:t-ext-principal}]
Let $\vk\colon G_1\to G_2$ be a principal homomorphism.
Choose any $t$-extensions $\pi_1\colon H_1\to G_1$ and $\pi_2\colon H_2\to G_2\hs$,
and let the $t$-extension $\pi_3\colon H_3\to G_1$
and the homomorphism $\lambda\colon H_3\to H_2$ be as in Proposition \ref{p:H3}.
Then the commutative diagram
\[
\xymatrix{
H_3\ar[d]_-{\pi_3}\ar[r]^-{\lambda}  & H_2\ar[d]^-{\pi_2}\\
G_1\ar[r]^-\vk   & G_2\,.
}
\]
is a desired $t$-extension of $\vk$.
\end{proof}

\section{Abelian Galois cohomology via the Picard braiding}
\label{s:Picard-to-abelian}

Let $G$ be a quasi-connected reductive $k$-group.
In Section \ref{s:Picard}  we constructed a Picard crossed module of algebraic groups
\begin{equation}\label{e:PCM}
\PCrM(G)=(G^\ssc,G,\rho,\theta,\br).
\end{equation}

Passing from \eqref{e:PCM} to $k_s$-points, we obtain a
Picard crossed module of groups
\[ \PCrM(G)(k_s)=(G^\ssc(k_s), G(k_s), \rho, \theta,\br).\]

We can describe the Picard crossed module of $k$-groups $\PCrM(G)(k_s)$ in terms of a $t$-extension $H\to G$.
Let $g_1,g_2\in G(k_s)$.
The short exact sequence of smooth $k$-groups \eqref{e:t-ext}
gives rise to an exact sequence of Galois cohomology over $k_s$
\[ S(k_s)\to H(k_s)\to G(k_s)\to H^1(k_s,S)=1,\]
hence the homomorphism $H(k_s)\to G(k_s)$ is surjective, and therefore
we can lift $g_1,g_2$ to some $h_1,h_2\in H(k_s)$, respectively.
Then by \eqref{e:br-H}
we have
\begin{align*}
&\{g_1,g_2\} =[h_1,h_2]\in H^\sss(k_s)=G^\ssc(k_s).
\end{align*}

The Picard braiding
\[\br\colon G(k_s)\times G(k_s)\to G^\ssc(k_s)\]
defines an abelian group structure on the pointed set
\[H^1_\crr(k,G)=H^1\big(\Gal(k_s/k), (G^\ssc(k_s), G(k_s),\rho,\theta)\big);\]
see Appendix \ref{app:gr-coh-Picard}.
We denote the obtained abelian group by
\[ H^1_\ab(k,G)\coloneqq H^1\big(k,(G^\ssc, G,\rho,\theta,\br)\big)\]
and call it the {\em abelian} (or: {\em abelianized}) {\em Galois cohomology group of} $G$.
The morphism of pointed sets \eqref{e:cr1}
\begin{equation*}
\ab^1=\crr^1\colon\, H^1(k,G)\,\lra\, H^1_\crr(k, G)=H^1_\ab(k,G)
\end{equation*}
is called the {\em abelianization map}.

\begin{remark}\label{r:crr-not-hom}
For an arbitrary homomorphism of quasi-connected reductive $k$-groups $\vk\colon G_1\to G_2$\hs,
the induced map
\[\vk_\crr\colon \, H^1_\ab(k,G_1)=H^1_\crr(k,G_1)\,\lra\, H^1_\crr(k,G_2)= H^1_\ab(k,G_2)\]
of \eqref{e:commut-cr} does not have to be a homomorphism.
Indeed, consider the inclusion homomorphism $\iota\colon B\into G$  of Example \ref{ex:PU2}
and Remark \ref{r:not-braiding}.
Let $\Gamma=\Gal(\C/\R)$ and let $\gamma\in\Gamma$ denote the complex conjugation.
Since $B^\ssc=1$, we have $H^1_\ab(\R,B)=H^1(\R,B)\cong B(\R)$,
which is a non-cyclic abelian group of order 4.
From the short exact sequence of $\R$-groups
\[ 1\to\mu_2\to G^\ssc\to G\to 1\]
we see that the abelianization map is the coboundary map
\[\ab^1=\delta^1\colon\, H^1(\R,G)\,\lra\, H^1_\ab(\R,G)=H^2(\R,\mu_2)=\{\pm1\}.\]
It sends the class of the 1-cocycle $\iota(b_1)=b_1\in G(\R)$
to $\tilde b_1\cdot{}^\gamma \tilde b_1=\tilde b_1^2=-1$.
Since the same holds for the classes of the 1-cocycles  $b_2$ and $b_1 b_2$ as well,
we see that the induced map $\iota_\crr\colon H^1_\ab(\R,B)\to H^1_\ab(\R,G)$
sends the three non-unit elements of the abelian group of order 4 $H^1_\ab(\R,B)\cong B(\R)$
to $-1\in H^1_\ab(\R,G)\cong\{\pm1\}$.
Therefore, this map is not a homomorphism.
\end{remark}

By Corollary \ref{c:principal-induces}, a {\em principal} homomorphism
of quasi-connected reductive $k$-groups
\[\vk\colon G_1\to G_2\]
induces a morphism of Picard crossed modules of algebraic groups
\[\PCrM(G_1)\coloneqq \big(G_1^\ssc, G_1,\rho_1,\theta_1,\br_{1}\big)
\to \big(G_2^\ssc, G_2,\rho_2,\theta_2,\br_{2}\big)\eqqcolon \PCrM(G_2)\]
which in turn induces a morphism of Picard crossed modules of $\Gal(k_s/k)$-groups
\[ \PCrM(G_1)(k_s)\to \PCrM(G_2)(k_s)\]
whence we obtain an induced homomorphism of abelian groups
\[\vk_\ab\colon H^1_\ab(k,G_1)\to H^1_\ab(k,G_2).\]
Thus we obtain the following theorem:

\begin{theorem}\label{t:functor}
The assignments
\[ G\rightsquigarrow H^1_\ab(k,G),\ \, \vk\mapsto \vk_\ab\]
define a functor from the category of quasi-connected reductive $k$-groups
with {\emm principal} homomorphisms to the category of abelian groups.
\end{theorem}

For a principal homomorphism of quasi-connected reductive $k$-groups $\vk\colon G_1\to G_2$
we can write the commutative diagram \eqref{e:commut-cr} as
\begin{equation*}
\begin{aligned}
\xymatrix{
H^1(k, G_1)\ar[r]^-{\vk_*}\ar[d]_-\ab &H^1(k, G_2)\ar[d]^-\ab\\
H^1_\ab(k,G_1)\ar[r]^-{\vk_\ab} &H^1_\ab(k,G_2).
}
\end{aligned}
\end{equation*}
In this diagram, the bottom horizontal arrow is a homomorphism of abelian groups,
whereas the other arrows are morphisms of pointed sets.

One can compute the abelian group cohomology $H^1_\ab(k,G)$
of a quasi-connected reductive $k$-group $G$
using principal quasi-tori in $G$.
Let $T^\sscp\subset G^\ssc$ be a maximal torus, and set
\[
Z=Z(G),\ Z^\sscp=Z(G^\sscp),\quad
T=\cZ_G(T^\sssp)=Z\cdot T^\sssp, \ \  \text{where}
   \ \, T^\sssp=\rho(T^\sscp)\subset G^\sss\subseteq G.\]

\begin{lemma}\label{l:qi-ZTG}
Let $G$, $G^\ssc$, $Z$, $Z^\sscp$, $T^\sscp$, and $T$ be as above.
Then the inclusions of complexes of $k$-groups
\begin{equation*}
(Z^\sscp\to Z)\,\labelt i\, (T^\sscp\to T)\,\labelt j\, (G^\ssc\to G)
\end{equation*}
are quasi-isomorphisms.
\end{lemma}

\begin{proof}
We must show that the induced homomorphisms on the kernels and cokernels
are isomorphisms.
We have $\ker\rho\subseteq Z^\sscp\subseteq T^\sscp$, whence
\[\ker[Z^\sscp\to Z]= \ker[T^\sscp\to T]=\ker[G^\ssc\to G].\]
Consider the induced homomorphisms
\begin{align}
i_\cok \colon&\coker[Z^\sscp\to Z]\to \coker[T^\sscp\to T]=T/T^\sssp\label{e:ZT}\\
j_\cok \colon&\coker[T^\sscp\to T]\to \coker[G^\ssc\to G]=G/G^\sss.\label{e:TG}
\end{align}
By Corollary \eqref{c:cZ_G} and Lemma  \ref{l:t-product}(i) we have
\[Z\cap G^\sss=Z^\sssp=\rho(Z^\sscp)\quad\ \text{and}
\quad\   T\cap G^\sss=T^\sssp=\rho(T^\sscp).\]
It follows that the homomorphisms  \eqref{e:ZT} and \eqref{e:TG}
are injective.
We show that  they are faithfully flat.
Write
\begin{align*}
G_\cok&=\coker[G^\ssc\to G]=G/G^\sss,\\
T_\cok&=\coker[T^\sscp\to T]=T/T^\sssp,\\
Z_\cok&=\coker[Z^\sscp\to Z]=Z/Z^\sssp.
\end{align*}
Since by Lemma \ref{l:t-product}(ii) we have $T=Z\cdot T^\sssp$ and since by
Definition \ref{d:q-c-r}(3) we have $G=Z\cdot G^\sss$,
by Lemma \ref{l:AN} the composite homomorphisms
\[ Z\into T\onto T_\cok\quad\ \text{and}\quad\  Z\into G\onto G_\cok\]
are faithfully flat,
whence the composite homomorphism $T\into G\onto G_\cok$ is faithfully flat as well
(we use Fact \ref{f:W-ff}).
It follows that the induced homomorphisms $Z_\cok\to T_\cok$ of \eqref{e:ZT}
and $T_\cok\to G_\cok$ of \eqref{e:TG} are faithfully flat
(again, we use Fact \ref{f:W-ff}).
Hence they are isomorphisms, as desired.
\end{proof}

\begin{proposition}\label{t:Zev} \
\begin{enumerate}
\item[\rm(i)]
With the assumptions and notation of Lemma \ref{l:qi-ZTG}, the morphism of crossed modules
{\em of the groups of $k_s$-points}
\[\big(T^\sscp(k_s)\to T(k_s)\big)\,\lra\,  \big(G^\ssc(k_s)\to G(k_s)\big)\]
is a quasi-isomorphism.
\item[\rm(ii)]
Moreover, if the finite $k$-group $Z^\sscp$ is \'etale, then the morphism of complexes
{\em of the groups of $k_s$-points}
\[\big(Z^\sscp(k_s)\to Z(k_s)\big)\,\lra\,  \big(T^\sscp(k_s)\to T(k_s)\big)\]
is a quasi-isomorphism.
\end{enumerate}
\end{proposition}

\begin{proof}
We closely follow  Appendix B to \cite{Borovoi-CR}, due to Zev Rosengarten,
where the case of a {\em connected} reductive group $G$ was considered.
We must show that the homomorphisms
\begin{align}\label{e:ker}
 j_{\ker}\colon& \ker\big[T^\sscp(k_s) \to T(k_s)\big] \,\lra\, \ker\big[G^\ssc(k_s) \to G(k_s)\big],\\
\label{e:coker}
j_{\cok}\colon& \coker\big[T^\sscp(k_s) \to T(k_s)\big] \,\lra\, \coker\big[G^\ssc(k_s) \to G(k_s)\big]
\end{align}
are isomorphisms.
Moreover, when $Z^\sscp$ is \'etale, we must show that the homomorphisms
\begin{align}\label{e:Tker}
 i_{\ker}\colon& \ker\big[Z^\sscp(k_s) \to Z(k_s)\big] \,\lra\, \ker\big[T^\sscp(k_s) \to T(k_s)\big],\\
\label{e:Tcoker}
i_{\cok}\colon& \coker\big[Z^\sscp(k_s) \to Z(k_s)\big] \,\lra\, \coker\big[T^\sscp(k_s) \to T(k_s)\big]
\end{align}
are isomorphisms.

The homomorphisms \eqref{e:ker} and \eqref{e:Tker} are  isomorphisms
because by Lemma \ref{l:qi-ZTG} we have
\[\ker[Z^\sscp\to Z]=\ker[T^\sscp\to T]=\ker[G^\ssc\to G].\]

We prove the injectivity of \eqref{e:coker}.
Let   $[t]\in \coker\big[T^\sscp(k_s) \to T(k_s)\big]$, $t\in T(k_s)$,
and $[t]\in\ker j_\cok$\hs; then $t=\rho(s)$ for some $s\in G^\ssc(k_s)$.
Since $T^\sscp=\rho^{-1}(T)$, we see that $s\in T^\sscp(k_s)$, whence $[t]=1$, as desired.

We prove the injectivity of \eqref{e:Tcoker}.
Let   $[z]\in \coker\big[Z^\sscp(k_s) \to Z(k_s)\big]$, $z\in Z(k_s)$,
and $[z]\in\ker i_\cok$\hs; then $z=\rho(s)$ for some $s\in T^\sscp(k_s)$.
Since $Z^\sscp=\rho^{-1}(Z)$, we see that $s\in Z^\sscp(k_s)$, whence $[z]=1$, as desired.

We prove the surjectivity of \eqref{e:coker} and, when $Z^\sscp$ is \'etale,
the surjectivity of \eqref{e:Tcoker}.
By  Definition \ref{d:q-c-r}(3), the  homomorphism
\[\psi_G\colon Z \times G^\ssc \to G,\qquad (z,s)\mapsto z\cdot \rho(s)\]
is faithfully flat.
Its kernel $K$ is central and is canonically isomorphic to $\rho^{-1}(Z^\sssp)=Z^\sscp$, which may be non-smooth.
Similarly, by Lemma \ref{l:t-product}(ii) the homomorphism
\[\psi_T\colon Z \times T^\sscp \to T,\qquad (z,s)\mapsto z\cdot \rho(s)\]
is faithfully flat. It has the same kernel $K$.
We have a commutative diagram of $k$-group schemes with exact rows
\[
\xymatrix@R=6mm{
1 \ar[r] &K\ar[r]\ar@{=}[d]  &Z \times_k Z^\sscp\ar[r]^-{\psi_Z}\ar[d]  &Z\ar[r]\ar[d]   &1 \\
1 \ar[r] &K\ar[r]\ar@{=}[d]  &Z \times_k T^\sscp\ar[r]^-{\psi_T}\ar[d]  &T\ar[r]\ar[d]   &1 \\
1 \ar[r] &K\ar[r]            &Z \times_k G^\ssc\ar[r]^-{\psi_G}               &G \ar[r]        &1
}
\]
in which the maps on $k_s$-points
\[\psi_T\colon Z(k_s)\times T^\sscp(k_s)\to T(k_s)\quad\ \text{and}
\quad\ \psi_G\colon Z(k_s)\times G^\ssc(k_s)\to G(k_s)\]
may not be surjective.
This diagram gives rise to a commutative diagram of fppf (flat) cohomology groups with exact rows
\begin{equation}\label{e:ZTG}
\begin{aligned}
\xymatrix@R=6mm{
Z(k_s) \times Z^\sscp(k_s)\ar[r]^-{\psi_Z}\ar[d]  &Z(k_s)\ar[r]\ar[d]
    &H_\fppf^1(k_s, K)\ar[r]\ar@{=}[d]   &H_\fppf^1(k_s,Z \times_k Z^\sscp)\ar[d]^-{i_*}\\
Z(k_s) \times T^\sscp(k_s)\ar[r]^-{\psi_T}\ar[d]  &T(k_s)\ar[r]\ar[d]
    &H_\fppf^1(k_s, K)\ar[r]\ar@{=}[d]   &H_\fppf^1(k_s,Z \times_k T^\sscp)\ar[d]_-\simeq^-{j_*}\\
Z(k_s) \times G^\ssc(k_s)\ar[r]^-{\psi_G}                &G(k_s)\ar[r]
    &H_\fppf^1(k_s, K)\ar[r]             &H_\fppf^1(k_s,Z \times_k G^\ssc)
}
\end{aligned}
\end{equation}
in which the rightmost vertical arrow marked $j_*$ is the isomorphism
\[H_\fppf^1(k_s,Z \times_k T^\sscp)\cong H_\fppf^1(k_s,Z) \cong H_\fppf^1(k_s,Z \times_k G^\ssc)\]
because $k_s$ is separably closed
and the $k$-groups $T^\sscp$ and $G^\ssc$ are smooth.
Diagram \eqref{e:ZTG} shows that
\[G(k_s) = T(k_s)\cdot \psi_G\big(\hs Z(k_s) \times G^\ssc(k_s)\hs\big)=
    T(k_s)\cdot Z(k_s)\cdot\rho\big(G^\ssc(k_s)\big) = T(k_s)\cdot \rho\big(G^\ssc(k_s)\big),\]
whence the desired surjectivity of \eqref{e:coker}.

If $Z^\sscp$ is \'etale, then the rightmost vertical arrow marked $i_*$ in \eqref{e:ZTG}
is the isomorphism
\[
H_\fppf^1(k_s,Z \times_k Z^\sscp)\cong H_\fppf^1(k_s,Z) \cong H_\fppf^1(k_s,Z \times_k T^\sscp)
\]
because $k_s$ is separably closed
and the $k$-groups $Z^\sscp$ and $T^\sscp$ are smooth.
Diagram \eqref{e:ZTG} shows that
\[T(k_s) = Z(k_s)\cdot \psi\big(\hs Z(k_s) \times T^\sscp(k_s)\hs\big)=
Z(k_s)\cdot \rho\big(T^\sscp(k_s)\big),\]
whence the desired surjectivity of \eqref{e:Tcoker}.
\end{proof}

\begin{theorem}\label{t:f-iso}
For a quasi-connected reductive $k$-group $G$, a maximal torus $T^\sscp\subset G^\ssc$,
and the principal quasi-torus $T=\cZ_G(\rho(T^\sscp))$,
we have an isomorphism of abelian groups
\[H^1(k,T^\sscp\labelt\rho T)\isoto H^1_\ab(k,G),\]
which is functorial with respect to {\emm principal} homomorphisms $G_1\to G_2$.
Here functoriality means that for every principal homomorphism
of quasi-connected reductive groups $\vk\colon G_1\to G_2$
and for all principal quasi-tori $T_1\subseteq G_1$ and $T_2\subseteq G_2$
such that $\vk(T_1)\subseteq T_2$,
we have a commutative diagram of abelian groups
\begin{equation}\label{e:f-iso}
\begin{aligned}
\xymatrix@C=15mm{
H^1(k, T_1^\sscp\to T_1)\ar[r]^-{\vk_*}\ar[d]_-\sim &H^1(k, T_2^\sscp\to T_2)\ar[d]^-\sim\\
H^1_\ab (k,G_1)\ar[r]^-{\vk_\ab} &H^1_\ab(k,G_2).
}
\end{aligned}
\end{equation}

\end{theorem}

\begin{proof}
For $G,G^\ssc, T, T^\sscp$ as above, consider the Picard crossed module
\[ \big(T^\sscp, T,\rho_T,\theta_\triv, \br_\triv\big)\]
where $\rho_T\colon T^\sscp\to T$ is the restriction of $\rho$,
$\theta_\triv$ is the trivial action of $T$ on $T^\sscp$,
and the braiding $\br_\triv\colon T\times_k T\to T^\sscp$ is the trivial  map (identically $1$).
Then
\[ H^1(k,T^\sscp, T,\rho_T,\theta_\triv, \br_\triv)=H^1(k,T^\sscp\to T),\]
and the inclusion
\[ \big(T^\sscp, T,\rho_T,\theta_\triv, \br_\triv\big)\into \big(G^\ssc, G,\rho,\theta, \br\big)\]
is a morphism of Picard crossed modules.
Thus we obtain a homomorphism of abelian groups
\[ H^1(k, T^\sscp\to T)\to H^1_\ab(k,G).\]
By \cite[Theorem 3.3]{Borovoi-Preprint} or \cite[Proposition 5.6]{Noohi},
it follows from Proposition \ref{t:Zev}(i) that this homomorphism is bijective, hence an isomorphism.

Let
\[ \vk\colon G_1\to G_2\]
be a {\em principal} homomorphism of quasi-connected reductive $k$-groups.
Let
\[ T_1=\cZ_{G_1}(\rho_1(T_1^\sscp))=Z(G_1)\cdot\rho_1(T_1^\sscp)\subseteq G_1\]
be a principal quasi-torus in $G_1$, and let
\[ T_2=\cZ_{G_2}(\rho_2(T_2^\sscp))=Z(G_2)\cdot\rho_2(T_2^\sscp)\subseteq G_2\]
be a principal quasi-torus in $G_2$ such that
\[ \vk(T_1)\subseteq T_2 \ \ \text{and hence} \ \ \vk^\ssc(T_1^\sscp)\subseteq T_2^\sscp .\]
Then we have a natural commutative diagram of Picard crossed modules
\begin{equation*}
\begin{aligned}
\xymatrix@C=17mm{
\big(T_1^\sscp, T_1,\hs\rho_{1,T},\hs\theta_{1,\triv},\hs\br_{1,\triv}\big)\ar[r]\ar[d]
           &\big(T_2^\sscp, T_2,\hs\rho_{2,T},\hs\theta_{2,\triv},\hs\br_{2,\triv}\big)\ar[d]\\
\big(G_1^\sscp, G_1,\rho_1,\theta_{1},\br_{1}\big)\ar[r]
           &\big(G_2^\sscp, G_2,\rho_2,\theta_{2},\br_{2}\big)
}
\end{aligned}
\end{equation*}
(with evident notations), which induces the commutative  diagram of abelian groups \eqref{e:f-iso}.
\end{proof}

\begin{remark}
Instead of using Corollary \ref{c:principal-induces} in the proof of Theorem \ref{t:functor},
one can use the commutative diagram \eqref{e:f-iso}.
\end{remark}

\begin{corollary}\label{c:f-iso}
For $G,G^\ssc, T, T^\sscp, Z, Z^\sscp$ as above, if the finite $k$-group $Z^\sscp$ is \'etale,
then we have isomorphisms of abelian groups
\begin{equation}\label{e:H-ZTG}
H^1(k,Z^\sscp\labelt\rho Z)\isoto H^1(k,T^\sscp\labelt\rho T)\isoto H^1_\ab(k,G).
\end{equation}
Moreover, the composite isomorphism
\begin{equation}\label{e:H-ZG}
H^1(k,Z^\sscp\labelt\rho Z)\isoto H^1_\ab(k,G)
\end{equation}
does not depend on the choice of the pair  $(T^\sscp,T)$.
\end{corollary}

\begin{proof}
By Proposition \ref{t:Zev}(ii) the morphism of complexes of commutative groups
\[\big(Z^\sscp(k_s)\to Z(k_s)\big)\hs\into\hs\big(T^\sscp(k_s)\to T(k_s)\big)\]
is a quasi-isomorphism, from which
we obtain the first isomorphism of \eqref{e:H-ZTG}.
We obtain the second isomorphism of \eqref{e:H-ZTG} from Theorem \ref{t:f-iso}.
The composite isomorphism \eqref{e:H-ZG} comes from the morphism of Picard crossed modules
\[ \big(Z^\sscp, Z,\rho_Z,\theta_\triv, \br_\triv\big)\into \big(G^\ssc, G,\rho,\theta, \br\big)\]
(with evident notations), and therefore it does not depend on the choice of $(T^\sscp,T)$.
\end{proof}

\appendix
\numberwithin{equation}{section}

\section{Universal covering of a semisimple group scheme as a functor}
\label{app:universal}

In this appendix  we assume that  the reader is familiar with the background and terminology presented in \cite[Section 6]{Con14}.
Nevertheless, we shall provide the essential definitions to keep this appendix self-contained.

Let $S$ be a non-empty scheme and let $\mathcal{O}_S$ be its structure sheaf.
By an $S$-group we mean an $S$-group scheme.
For $S$-groups $G$ and $G'$, we denote by $\ul{\Hom}_{S\textnormal{-gp}}(G,G')$ the group functor on the category of $S$-schemes that is defined by
\[
\ul{\Hom}_{S\textnormal{-gp}}(G,G') (S')
=
\Hom_{S'\textnormal{-gp}}(G_{S'},G'_{S'})
\]
for every $S$-scheme $S'$.

We begin by recalling some basic facts concerning $S$-groups.
For a finitely generated abelian group $M$,
we denote by $M_S \coloneqq \bigsqcup_{m \in M} S_m$ the constant commutative $S$-group,
where $S_m$ denotes a copy of $S$ indexed by $m \in M$.
We define $D_S(M)$ to be the $S$-group $\operatorname{Spec}(\mathcal{O}_S[M])$ representing the group functor $\ul{\Hom}_{S\textnormal{-gp}}(M_S,\mathbb{G}_m)$; see \cite[Exposé VIII. Definition 1.1]{SGA3}, which is finitely presented and faithfully flat over $S$; see \cite[Exposé VIII. Proposition 2.1]{SGA3}.

\begin{definition}(\cite[Definition B.1.1]{Con14}, \cite[Exposé VIII, Definition 1.1]{SGA3})\label{d:mult_type}
    An $S$-group $H$ is \textit{diagonalizable} if $H \simeq D_S(M)$ over $S$ for a finitely generated abelian group $M$.
    An $S$-group $H$ is \textit{of multiplicative type} if there is an fppf covering $\{S_i \to S\}$ such that $H_{S_i}$ is diagonalizable over $S_i$ for each $i$.
\end{definition}

By fppf descent (for example, see \cite[Theorem 4.3.7]{Poonen}), an $S$-group of multiplicative type is finitely presented and faithfully flat over $S$.
A particular class of groups of multiplicative type consists of $S$-\textit{tori}.

\begin{definition}(\cite[Exposé IX, Definition 1.3]{SGA3})\label{d:torus}
    An $S$-group $T$ is an \textit{split $S$-torus} if $T \simeq D_S(M)$ over $S$
    for some finitely generated free abelian group $M$.
    An $S$-group $T$ is an $S$\textit{-torus} if there is an fppf covering $\{S_i \to S\}$ such that $T_{S_i}$ is a split torus for each $i$.
\end{definition}

If an $S$-group $H$ is diagonalizable, then $\ul{\Hom}_{S\textnormal{-gp}}(H,\mathbb{G}_m)$ is represented by a constant $S$-group associated with a finitely generated  abelian group $\X(H)$; see \cite[Exposé VIII, Corollary 1.5]{SGA3}.
The exact functors
\begin{equation}\label{t:equiv_multtype_locc}
    H \rightsquigarrow \X(H),
    \quad
    M \rightsquigarrow D_S(M) = \ul{\Hom}_{S\textnormal{-gp}}(M_S,\mathbb{G}_m)
\end{equation}
form an anti-equivalence between the category of diagonalizable $S$-groups and the category of finitely generated abelian groups; see \cite[Exposé VIII, Theorem 1.2]{SGA3}.
In fact, this anti-equivalence generalizes to the anti-equivalence  between the category of $S$-groups of multiplicative type and the category of {\em twisted constant $S$-groups} whose geometric fibers are finitely generated abelian groups; see {\cite[Exposé X, Corollary 5.9, Exposé VIII, Theorem 3.1]{SGA3},  or \cite[Proposition B.3.4]{Con14}}.

\begin{definition}(\cite[Expos\'e XIX, Definition 2.7]{SGA3})\label{d:red_ss_S-grp}
    An $S$-group $G$ is called \textit{reductive} (resp.  \textit{semisimple}) if it is affine and smooth over $S$ with connected and reductive (resp. semisimple) geometric fibers.
\end{definition}

\begin{definition}(\cite[Expos\'e XXII, Definition 4.2.9]{SGA3}, \cite[Definition 3.3.9]{Con14})
    Let $G$ and $G'$ be $S$-groups  that are affine and smooth over $S$.
    An $S$-homomorphism $f : G' \to G$ is an \textit{isogeny} if it is finite and faithfully flat.
    If, moreover, the kernel of $f$ is central in $G'$ (see \cite[Definition 2.2.2]{Con14}),
    then $f$ is called a \textit{central isogeny}.
\end{definition}

\begin{definition}(\cite[Expos\'e XXII, Definition 4.3.3]{SGA3})\label{d:sc}
A semisimple $S$-group $G$ is called {\em simply connected}
if its geometric fibers are simply connected
\end{definition}

\begin{definition}\label{d:universal}
Let $G$ be a semisimple $S$-group.
A {\em universal covering of} $G$ is a central isogeny $\pi\colon G^\ssc\to G$
where $G^\ssc$ is a simply connected semisimple $S$-group.
\end{definition}

\begin{lemma}\label{l:uni_cov_bc}
    Let $G$ be a semisimple $S$-group and let $S'$ be an $S$-scheme.
    Then $G_{S'}$ is a semisimple $S'$-group, and for a universal covering $\pi \colon G^\ssc \to G$ of $G$,
    the induced $S'$-homomorphism $\pi_{S'} \colon G^\ssc_{S'} \to G_{S'}$ is a universal covering of $G_{S'}$.
\end{lemma}
\begin{proof}
    Since the property of being simply connected is unchanged by any base field extension (see \cite[Corollary 5.6.23]{Poonen}), the $S'$-group $G_{S'}$ (resp. $G^\ssc_{S'}$) is a semisimple (resp. simply connected semisimple) $S'$-group by Definition \ref{d:red_ss_S-grp} (resp. Definition \ref{d:sc}).

    It remains to prove that $\pi_{S'}$ is a central isogeny.
    By \cite[Proposition 3.3.10]{Con14}, a surjective homomorphism of reductive $S$-groups is a central isogeny if and only if each of its geometric fibers has a central kernel.
    Since the geometric fibers of $\pi_{S'}$ are obtained from those of $\pi$ by base field extension, it suffices to prove the claim after replacing $S$ by the spectrum of a field.
    Thus let $S=\operatorname{Spec} k$ for a field $k$, and let $k'/k$ be a field extension. Then $\pi_{k'}$ has central kernel, since both kernels and centers are preserved under base field extension (see \cite[Remark 5.2.15]{Poonen} for the center).
    Hence $\pi_{k'}$ is a central isogeny, which completes the proof.
\end{proof}

\begin{definition}(\cite[Definition 3.2.1]{Con14})\label{d:max_tor}
    A \textit{maximal torus} in a smooth affine $S$-scheme $G$ is a torus $T \subset G$ such that the fiber $T_{\overline{s}}$ of any geometric point $\overline{s}$ of $S$ is a maximal torus in $G_{\overline{s}}$.
\end{definition}

Suppose that a reductive $S$-group $G$ admits a split maximal torus $T \subset G$.
Fix an isomorphism $T \simeq D_S(M)$ for a finitely generated free abelian group $M$
and let $M^\vee$ be the dual lattice of $M$.
By \cite[Exposé X, Theorem 5.6, Exposé XXII, 1.11]{SGA3}, both functors  $\ul{\Hom}_{S\textnormal{-gp}}(T,\mathbb{G}_m)$ and $\ul{\Hom}_{S\textnormal{-gp}}(\mathbb{G}_m,T)$ are representable by the following constant $S$-groups:
\[
M_S \cong \ul{\Hom}_{S\textnormal{-gp}}(T,\mathbb{G}_m)
\quad\text{and}\quad
M^\vee_S \cong \ul{\Hom}_{S\textnormal{-gp}}(\mathbb{G}_m,T).
\]
This induces canonical isomorphisms on $S$-points
\[
M_S(S) \cong \Hom_{S\textnormal{-gp}}(T,\mathbb{G}_m)
\quad\text{and}\quad
M^\vee_S(S) \cong \Hom_{S\textnormal{-gp}}(\mathbb{G}_m,T),
\]
where $M_S(S)$ (resp. $M^\vee_S(S)$) consists of locally constant $M$-valued (resp. $M^\vee$-valued) functions on $S$.
In particular, by identifying $M$ (resp. $M^\vee$) with the group of constant $M$-valued (resp. $M^\vee$-valued) functions on $S$, we obtain injective homomorphisms
\[
M \hookrightarrow \Hom_{S\textnormal{-gp}}(T,\mathbb{G}_m)
\quad\text{and}\quad
M^\vee \hookrightarrow \Hom_{S\textnormal{-gp}}(\mathbb{G}_m,T).
\]
If the scheme $S$ is connected, then these injections are also surjective.

We consider the vector bundle $\mathfrak{g} = \Lie(G/S)$ over $S$ equipped with the natural $T$-action
via the adjoint representation of $G$ and with the natural $\mathbb{G}_m$-action defined by scalar multiplication.
By identifying $m \in M$ with its image under $M \hookrightarrow \Hom_{S\textnormal{-gp}}(T,\mathbb{G}_m)$,
we may regard $m$ as a character of $G$, hence
$m(t) \in \mathbb{G}_m(S')$ for every $S$-scheme $S'$ and every $t \in T(S')$.
We define $\mathfrak{g}_m$ to be the subbundle of $\mathfrak{g}$ whose $S'$-points for every $S$-scheme $S'$ are given by
\[
\mathfrak{g}_m(S') = \{v \in \mathfrak{g}(S') \mid t \cdot v_{S''} = m(t) \cdot v_{S''}
\textit{ for every  $S'$-scheme } S''\textit{ and } t \in T(S'') \}
\]
where $t \cdot v_{S''}$ and $m(t) \cdot v_{S''}$ denote the actions
of $T$ and $\mathbb{G}_m$ on $\mathfrak{g}$,  respectively; see \cite[Exposé XIX, \S 3.1]{SGA3}.
There is an $\mathcal{O}_S$-linear $M$-graded decomposition $\mathfrak{g} = \bigoplus_{m \in M} \mathfrak{g}_m$;
see \cite[Section 4]{Con14} or \cite[Lemma A.8.8]{CGP}.
The elements $m \in M$ for which $\mathfrak{g}_m \neq 0$ are called the \textit{weights} of the $T$-action on $\mathfrak{g}$.

\begin{definition}(\cite[Definition 4.1.1]{Con14})
    Suppose that a reductive $S$-group $G$ admits a split maximal $S$-torus $T \subset G$ with an isomorphism $T \simeq D_S(M)$ for a finitely generated free abelian group $M$.
    A \textit{root} for $(G,T,M)$ is a nonzero $a \in M$ such that $\mathfrak{g}_a$ is a line bundle.
    We call such $\mathfrak{g}_a$ a \textit{root space}.
\end{definition}

Under the natural inclusion $M \xhookrightarrow{} \Hom_{S\textnormal{-gp}}(T,\mathbb{G}_m)$, we may view a root $a \in M$ as a nontrivial character $a \colon T \to \mathbb{G}_m$ .
A root space $\mathfrak{g}_a$ need not be a trivial line bundle; see \cite[Example 4.1.2]{Con14}.

\begin{definition}(\cite[Definition 5.1.1]{Con14})\label{def:split_red_S-grp}
    A reductive $S$-group $G$ is \textit{split} if there is a (split) maximal torus $T$ with an isomorphism $T \simeq D_S(M)$ for a finitely generated free abelian group $M$ such that
    \begin{enumerate}
        \item the nontrivial weights $a\colon T \to \mathbb{G}_m$ that occur in $\mathfrak{g}$ arise from elements of $M$
        under $M \xhookrightarrow{} \Hom_{S\textnormal{-gp}}(T,\mathbb{G}_m)$,
        \item  each weight space $\mathfrak{g}_{a}$ is a trivial line bundle,
         \item each coroot $a^\vee \colon \mathbb{G}_m \to T$ arises from an element of $M^\vee$ under $M^\vee \xhookrightarrow{} \Hom_{S\textnormal{-gp}}(\mathbb{G}_m,T)$; see \cite[Theorem 4.2.6]{Con14} for the definition of coroots.
    \end{enumerate}
    We call such $(G,T,M)$ a \textit{split triple}.
\end{definition}

In particular, every nontrivial weight $a \colon T \to \mathbb{G}_m$ occurring in Definition \ref{def:split_red_S-grp} is a root for $(G,T,M)$.
For a split triple $(G,T,M)$, we denote by
 $\Phi$ the set of roots $a \in M$ for $(G,T,M)$, and by $\Phi^\vee$ the set of corresponding coroots $a^\vee \in M^\vee$.
The associated $R(G,T,M) \coloneqq (M,\Phi,M^\vee,\Phi^\vee)$ is a reduced root datum; see \cite[Proposition 5.1.6]{Con14}.

\begin{definition}(\cite[Definition 6.1.6]{Con14})\label{d:compatible_splittings}
    Let $(G,T,M)$ and $(G',T',M')$ be split triples over a non-empty scheme $S$, and let $(M,\Phi,M^\vee,\Phi^\vee)$ and $(M',\Phi',M'\hs^\vee,\Phi'\hs^\vee)$ be the associated reduced root data respectively.
    A central isogeny $f \colon G' \to G$ such that $f(T') \subseteq T$ is called \textit{compatible with the splittings} if there is a homomorphism $h \colon M \to M'$ and a bijection $d \colon \Phi' \isoto \Phi$ such that
    \begin{enumerate}
        \item the induced homomorphism $f^* \colon M_S = \ul{\Hom}_{S\textnormal{-gp}}(T,\mathbb{G}_m) \to \ul{\Hom}_{S\textnormal{-gp}}(T',\mathbb{G}_m) = M'_{S}$ arises from $h$,
        \item $h(d(a')) = a'$ and $h^\vee(a^{\prime\hs\vee}) = d(a')^\vee$ for every $a' \in \Phi'$.
    \end{enumerate}
\end{definition}

We also write $f \colon (G',T',M') \to (G,T,M)$ to denote that $f$ is compatible with the splittings.

Note that the restriction $f|_{T'} \colon T' \to T$ is an isogeny and  $h \colon M \to M'$
is a finite-index inclusion; see \cite[Exposé XXII, Proposition 4.2.10]{SGA3}.
Furthermore, the isogeny $f$ uniquely determines $h$ and $d$.
The notion of an \textit{isomorphism} between split triples is the evident one.

\begin{definition}[{\cite[Exposé XXII, \S 4.2]{SGA3}}]\label{d:splittable}
    A central isogeny $f \colon G' \to G$ of reductive $S$-groups is \textit{splittable}
    if there exist maximal tori $T' \simeq D_S(M')$ of $G'$ and $T \simeq D_S(M)$ of $G$ such that
    \begin{enumerate}
        \item $(G,T,M)$ and $(G',T',M')$ are split triples,
        \item $f(T') \subseteq T$,
        \item $f\colon (G',T',M') \to (G,T,M)$ is compatible with the splittings.
    \end{enumerate}
\end{definition}

\begin{lemma}[{\cite[Exposé XXII, Corollary 4.2.13]{SGA3}}]\label{l:isog_split}
    Every central isogeny $f \colon G' \to G$ of semisimple $S$-groups is étale-locally splittable, that is, there exists an étale covering $\{S_i \to S\}$ such that $f_{S_i} \colon G'_{S_i} \to G_{S_i}$ is splittable.
\end{lemma}

Let us recall the definitions of semisimple and simply connected root data and the notion of a central isogeny between reduced root data.

\begin{definition}(\cite[Definition 1.3.11]{Con14})\label{d:root_data}
    A reduced root datum $R = (X, \Phi, X^\vee, \Phi^\vee)$ is \textit{semisimple} if $\Phi$ spans $X_{\mathbb{Q}}$ over $\mathbb{Q}$.
    A semisimple reduced root datum is \textit{simply connected} if $\mathbb{Z}\Phi^\vee = X^\vee$.
\end{definition}

\begin{lemma}\label{l:ssc_iff_ssc}
    Let $(G,T,M)$ be a split triple and let $R = (M,\Phi,M^\vee,\Phi^\vee)$ be the associated root datum.
    \begin{enumerate}
        \item[\rm (i)] An $S$-group $G$ is semisimple if and only if $R$ is semisimple.
        \item[\rm (ii)] A semisimple $S$-group $G$ is simply connected if and only if $R$ is simply connected.
    \end{enumerate}
\end{lemma}

\begin{proof}
    Assertion (i) follows from \cite[Example 5.1.7]{Con14}.
    For assertion (ii), suppose that $G$ is a semisimple $S$-group.
    Then $\Phi$ spans $M_{\mathbb{Q}}$ over $\mathbb{Q}$ by assertion (i), hence both $\mathbb{Z} \Phi \subseteq M$ and $\mathbb{Z}\Phi^\vee \subseteq M^\vee$ are finite-index inclusions, which yields
    \[
    \mathbb{Z} \Phi \subseteq M \subseteq (\mathbb{Z}\Phi^\vee)^*
    \]
    where the lattice $(\mathbb{Z}\Phi^\vee)^*$ in $M_{\mathbb{Q}}$ is $\mathbb{Z}$-dual to the coroot lattice $\mathbb{Z}\Phi^\vee \subseteq M^\vee$.
    For all geometric points $\bar s$ of $S$, the split triples $(G_{\bar s},T_{\bar s},M)$ have the same associated root datum $(M,\Phi,M^\vee,\Phi^\vee)$.
    This yields that any geometric fiber $G_{\bar s}$ is simply connected if and only if $M = (\mathbb{Z}\Phi^\vee)^*$ by \cite[Example 3.2.6]{CGP}, which proves assertion (ii) by Definition \ref{d:sc}.
\end{proof}

\begin{definition}(\cite[Definition 6.1.8, Example 6.1.9]{Con14})\label{d:isog_root}
    Let $R = (X,\Phi,X^\vee,\Phi^\vee)$ and
    $R' = (X',\Phi',X^{\prime\hs\vee},\Phi^{\prime\hs\vee})$ be reduced root data.
    A \textit{central isogeny} $R' \to R$ of root data is a pair $(h,d)$ where $h \colon X \to X'$ is a homomorphism and $d \colon \Phi' \isoto \Phi$ is a bijection such that
    \begin{enumerate}
        \item [(i)]  $h$ is a finite-index injection (i.e., $h_\mathbb{Q}$ is an isomorphism),
        \item[(ii)] $h(d(a')) = a'$ and $h^\vee(a^{\prime\hs\vee}) = d(a')^\vee$ for all $a' \in \Phi'$.
    \end{enumerate}
\end{definition}

Since the root data are reduced, the homomorphism $h$ uniquely determines $d$ via the condition $h(d(a')) = a'$.
The notion of \textit{isomorphism} between root data is the evident one.

For a central isogeny $f \colon (G',T',M') \to (G,T,M)$ compatible with splittings, there is the associated central isogeny of root data
\[
R(f) \colon R(G',T',M') = (M',\Phi',M^{\prime\hs\vee},\Phi^{\prime\hs\vee}) \to (M,\Phi,M^\vee,\Phi^\vee) = R(G,T,M).
\]

\begin{theorem}[{\cite[Theorem 6.1.16, Theorem 6.1.17]{Con14}, \cite[Exposé XXV, Theorem 1.1, Corollary 1.2]{SGA3}}]\label{t:3thms}
    Let $(G,T,M)$ and $(G',T',M')$ be split triples over $S$.
    \begin{enumerate}
        \item \textup{(Isogeny Theorem)}
        Every central isogeny of root data $R(G',T',M') \to R(G,T,M)$ is induced by a central isogeny $f \colon (G',T',M') \to (G,T,M)$ compatible with the splittings, and such $f$ is unique up to the natural faithful $(T/Z_G)(S)$-action on $G$.
        \item \textup{(Existence Theorem)}
        Every reduced root datum is isomorphic to the root datum of a split triple $(G,T,M)$.
        \item \textup{(Isomorphism Theorem)}
        Every isomorphism of root data $R(G',T',M') \isoto R(G,T,M)$ is induced by an isomorphism of split triples $(G',T',M') \isoto (G,T,M)$, and this isomorphism is unique up to the natural faithful $(T/Z_G)(S)$-action on $G$.
    \end{enumerate}
\end{theorem}

Theorem \ref{t:3thms} is stated in \cite{SGA3} for morphisms of split reductive $S$-groups compatible with the splittings and $p$-morphisms of root data  (defined in \cite[Exposé XXII, Definition 4.2.1 and  Exposé XXI, Definition 6.8.1]{SGA3}),
and the statement is specialized in \cite[Theorem 6.1.16 and Theorem 6.1.17]{Con14} to the case of isogenies
(not necessarily central).
Here, we restrict our statement to \textit{central} isogenies, which is sufficient for our purposes;
see \cite[Example 6.1.9]{Con14}.
We explain the natural faithful $(T/Z_G)(S)$-action on $G$ in the following remark:

\begin{remark}\label{r:conj_act}
    Let $G$ be a reductive $S$-group.
    Consider a homomorphism of group functors
    \[
    \inn \colon G \to \ul{\Aut}(G), \quad g \mapsto \inn(g),
    \]
    where for any $S$-scheme $S'$ and any $g \in G(S')$, the induced automorphism $\inn(g) \in \Aut_{S'}(G_{S'})$ is defined on $S''$-points by $\inn(g)(g') = g g' g^{-1}$ for every $g' \in G(S'')$ and every $S'$-scheme $S''$; see Construction \ref{cons:theta}.
    This homomorphism fits into an exact sequence of group functors
    \[
    1 \to Z_G \to G \xrightarrow{\, \inn\, } \ul{\Aut}(G),
    \]
    which induces an injective homomorphism
    \[
    c \colon G/Z_{G} \to \ul{\Aut}(G).
    \]
    See \cite[Exposé XXIV \S 1.1]{SGA3}.
    The faithful action of $(T/Z_{G})(S)$ on $G$ in Theorem \ref{t:3thms}(1),(3) is induced by the composition
    \[
    T/Z_G \to G/Z_{G} \xrightarrow{\,c\,} \ul{\Aut}(G).
    \]
    For any $\bar g \in (G/Z_G)(S)$, the definition of the fppf quotient yields that there exists an fppf covering $\{S_i \to S\}$ and $g_i \in G(S_i)$ such that ${\bar g}_{S_i} = \bar{g}_i \in G/Z_{G}(S_i)$, hence
    \[
    c(\bar g)_{S_i} = c(\bar g_i) = \inn(g_i) \in \ul{\Aut}(G)(S_i) = \Aut_{S_i}(G_{S_i})
    \]
    for each $i$, where $\bar g_i$ is the image of $g_i$ under $G(S_i) \to (G/Z_G)(S_i)$ and $\bar g_{S_i}$ (resp. $c(\bar g)_{S_i}$) is the image of $\bar g$ under $(G/Z_G)(S) \to (G/Z_G)(S_i)$ (resp. the image of $c(\bar g)$ under $\ul{\Aut}(G)(S) \to \ul{\Aut}(G)(S_i)$).
\end{remark}

\begin{proposition} [{\cite[Exercise 6.5.2(i)]{Con14}}]
\label{p:Conrad-existence}
Any semisimple $S$-group $G$ over a non-empty scheme $S$ admits a universal $S$-covering $\pi\colon G^\ssc\to G$.
Such $(G^\ssc,\pi)$ is unique up to a unique isomorphism, that is,
    if $(G^{\prime\hs\ssc},\pi')$ is another universal covering of $G$,
    then there exists a unique isomorphism $G^\ssc \isoto G^{\prime\hs\ssc}$
    compatible with $\pi$ and $\pi'$.
\end{proposition}

Proposition \ref{p:Conrad-existence} is well known for semisimple groups over a field;
see \cite[Proposition (2.24)(ii)]{BT} or \cite[Corollary A.4.11(1)]{CGP}.
For group schemes it is stated as an exercise  in \cite[Exercise 6.5.2(i)]{Con14}.
We give a detailed proof here for the sake of completeness,
since our Proposition \ref{p:GA} depends on \cite[Corollary 2.6]{GA}, a result whose proof implicitly assumes Proposition \ref{p:Conrad-existence}.

\begin{proof}
    In Step 1 and Step 2, we first assume that our semisimple $S$-group $G$ is split.

    \textit{Step 1.} We prove the existence of a splittable universal covering of a split semisimple $S$-group
    (see Definition \ref{d:splittable}).
    Let $(G,T,M)$ be a split triple and let $R = (M,\Phi,M^\vee,\Phi^\vee)$ be the associated reduced root datum.
    Since $G$ is semisimple, the root datum $R$ is also semisimple; see Lemma \ref{l:ssc_iff_ssc}.
    Hence, we have $\mathbb{Z} \Phi \subseteq M \subseteq (\mathbb{Z} \Phi^\vee)^*$,
    where the lattice $(\mathbb{Z}\Phi^\vee)^*$ in $M_{\mathbb{Q}}$ is $\mathbb{Z}$-dual
    to the coroot lattice $\mathbb{Z}\Phi^\vee \subseteq M^\vee$.
    We define a root datum $R^\ssc \coloneqq (M^\ssc,\Phi,{M^\ssc}^\vee,\Phi^\vee)$ with $M^\ssc \coloneqq (\mathbb{Z} \Phi^\vee)^*$ and ${M^\ssc}^\vee \coloneqq \mathbb{Z} \Phi^\vee$.
    This root datum $R^\ssc$ is simply connected semisimple (see Definition \ref{d:root_data}),
    with the evident central isogeny $(h,d) \colon R^\ssc \to R$ of root data
    where $h \colon M \hookrightarrow M^\ssc$ is the natural inclusion and $d \colon \Phi \to \Phi$ is the identity map.
    By the Existence Theorem and Isogeny Theorem,
    there exists a split triple $(G^\ssc,T^\sscp,M^\ssc)$ whose associated root datum is $R^\ssc$,
    together with a central isogeny $\pi \colon (G^\ssc,T^\sscp,M^\ssc) \to (G,T,M)$
    compatible with the splittings and inducing the isogeny $(h,d)$ of root data.
    Since the root datum $R^\ssc$ is simply connected semisimple, the $S$-group $G^\ssc$ is simply connected semisimple by Lemma \ref{l:ssc_iff_ssc},
    providing a splittable universal covering $\pi \colon G^\ssc \to G$.

    \textit{Step 2.}
    We prove the existence of an isomorphism between splittable universal coverings of a split semisimple $S$-group.
    Let $G'$ be a simply connected split semisimple $S$-group and let $\pi' \colon G' \to G$ be a splittable universal covering.
    Then there exist split maximal $S$-tori $T'\simeq D_S(M')$ in $G'$ and $T \simeq D_S(M)$ in $G$ such that $\pi' \colon (G',T',M') \to (G,T,M)$ is a central isogeny compatible with the splittings.
    Let $\pi \colon G^\ssc \to G$ be a splittable universal covering in Step 1.
    We wish to show that there exists an isomorphism $\varphi \colon G' \to G^\ssc$ such that $\pi \circ \varphi = \pi'$.

    Consider a central isogeny of root data associated with $\pi'$
    \[
    (h',d') \colon R' = (M',\Phi',M^{\prime\hs\vee},\Phi^{\prime\hs\vee}) \to (M,\Phi,M^\vee,\Phi^\vee) = R
    \]
    where $h' \colon M \hookrightarrow M'$ is a finite-index injective homomorphism and $d' \colon \Phi' \isoto \Phi$ is a bijection.
    Since $G'$ is simply connected semisimple, we have $M' = (\mathbb{Z}\Phi^{\prime\hs\vee})^*$
    and $M^{\prime\hs\vee} = \mathbb{Z}\Phi^{\prime\hs\vee}$ by Lemma \ref{l:ssc_iff_ssc}.
    The bijection $d'$ induces an isomorphism on the coroot lattices
    $d^{\prime\hs\vee} \colon \mathbb{Z} \Phi^{\prime\hs\vee} \isoto \mathbb{Z} \Phi^\vee, \, a^{\prime\hs\vee} \mapsto d'(a')^\vee$.
    By taking the dual, we have an isomorphism $h'' \coloneqq (d^{\prime\hs\vee})^* \colon M^\ssc \isoto M'$ where $M^\ssc = (\mathbb{Z} \Phi^\vee)^*$ and $M' = (\mathbb{Z} \Phi^{\prime\hs\vee})^*$.
    With a bijection $d'' \coloneqq d^{-1} \circ d' \colon \Phi' \isoto \Phi$, this yields an isomorphism between the root data
    \[
    (h'',d'') \colon R' = (M',\Phi',M^{\prime\hs\vee},\Phi^{\prime\hs\vee})
    \isoto
    (M^\ssc,\Phi,M^{\ssc\hs\vee},\Phi^\vee) = R^\ssc.
    \]
    We wish to show that this isomorphism satisfies the compatibility $(h',d') = (h,d) \circ (h'',d'')$.
    Since $d' = d \circ d''$, it suffices to show $h' = h'' \circ h$.
    By Definition \ref{d:isog_root}(ii), we have
    \[
    h^{\prime\hs\vee}(a^{\prime\hs\vee}) = d'(a')^\vee = d(d''(a'))^\vee = h^\vee(d''(a')^\vee) = h^\vee(h^{\prime\prime\hs\vee}(a^{\prime\hs\vee})) = (h^\vee \circ h^{\prime\prime\hs\vee})(a'\hs^\vee)
    \]
    for all $a' \in \Phi'$.
    Since $M^{\prime\hs\vee} = \mathbb{Z} \Phi^{\prime\hs\vee}$, this implies
    $h^{\prime\hs\vee} = h^\vee \circ h^{\prime\prime\hs\vee}$, hence $h' = h'' \circ h$.

    By Isomorphism Theorem, the isomorphism of root data $(h'',d'')$ is induced by an isomorphism between split triples
    $
    \pi'' \colon (G',T',M') \isoto (G^\ssc,T^\sscp,M^\ssc).
    $
    Since $\pi'$ and $\pi \circ \pi''$ induce the same central isogeny of root data $(h',d') = (h,d) \circ (h'',d'')$, by Isogeny Theorem there exists $\bar{t} \in (T/Z_G)(S)$ such that $\pi \circ \pi'' = c(\bar t) \circ \pi'$; see Remark \ref{r:conj_act} for $c(\bar t)$.
    We then modify $\pi''$ to define an isomorphism $\varphi \colon G' \isoto G^\ssc$ satisfying the compatibility $\pi \circ \varphi = \pi'$.
    We wish to lift $\bar t$ to an element of $(T^\sscp/Z_{G^\ssc})(S)$.
    The centers of $G^\ssc$ and $G$ can be identified with $Z_{G^\ssc} \simeq D_S(M^\ssc/\mathbb{Z}\Phi)$ and $Z_G \simeq D_S(M/\mathbb{Z}\Phi)$; see \cite[Example 5.1.7]{Con14}.
    Since $T^\sscp \simeq D_S(M^\ssc)$ and $T\simeq D_S(M)$, the anti-equivalence in \eqref{t:equiv_multtype_locc} shows that both $T^\ssc/Z_{G^\ssc}$ and $T/Z_{G}$ are isomorphic to $D_S(\mathbb{Z}\Phi)$.
    Thus $\pi$ induces an isomorphism $T^\ssc/Z_{G^\ssc} \isoto T/Z_{G}$, and  hence there is a unique lift $\bar{t}^\ssc \in (T^\ssc/Z_{G^\ssc})(S)$ of $\bar{t} \in (T/Z_{G})(S)$.
    Now the isomorphism $\varphi \coloneqq c(\bar{t}^\ssc)^{-1} \circ \pi'' \colon G' \isoto G^\ssc$ satisfies
     \begin{equation}\label{e-pi-varphi-pi'}
    \pi \circ \varphi
    = \pi \circ c(\bar{t}^\ssc)^{-1} \circ \pi''
    = c(\bar{t})^{-1} \circ \pi \circ \pi''
    = \pi',
    \end{equation}
   where the second equality follows from Lemma \ref{l:inn_comm} below,
    and \eqref{e-pi-varphi-pi'} gives the desired compatibility $\pi \circ \varphi = \pi'$.

    In the remaining steps, we consider an arbitrary semisimple $S$-group $G$.

    \textit{Step 3.} We prove the uniqueness of an isomorphism between universal coverings.
    Let $\pi \colon G^\ssc \to G$ and $\pi' \colon {G^\ssc}' \to G$ be universal coverings.
    Suppose $\phi_1, \phi_2 \colon G^\ssc \isoto {G^\ssc}'$ are two $S$-group isomorphisms such that $\pi' \circ \phi_1 = \pi$ and $\pi' \circ \phi_2 = \pi$.
    For any $S$-scheme $S'$ and any $g \in G^\ssc(S')$, we have $\pi' \circ \phi_1(g) = \pi' \circ \phi_2(g)$, which implies that $\phi_1(g) \cdot \phi_2(g)^{-1} \in \ker \pi'(S')$.
    Since $\pi'$ is a central isogeny, $\ker\pi'$ is a central subgroup scheme of ${G^\ssc}'$.
    Hence the pointwise product
    $ (\phi_1 \cdot \phi_2^{-1})(g)=\phi_1(g)\cdot \phi_2(g)^{-1} $
    defines an $S$-group homomorphism
    $\phi_1\cdot \phi_2^{-1}\colon G^\ssc\longrightarrow\ker\pi'$.
    By \cite[Exposé XXII, Theorem 6.2.1(iv,v)]{SGA3}, any homomorphism from a semisimple $S$-group
    to a commutative $S$-group is trivial.
    This implies that the homomorphism $\phi_1 \cdot \phi_2^{-1}$ is trivial, whence $\phi_1 = \phi_2$.

    In Step 4 and Step 5, we prove the uniqueness and the existence of a universal covering for an arbitrary semisimple $S$-group $G$ via étale descent.
    For an étale covering $\{S_i \to S\}$, we set $S_{ij}= S_i\times_S S_j$ and $S_{ijk} = S_i \times_S S_j \times_{S} S_k$.
    For any pair $(i,j)$, we denote the natural projections by $p_1 \colon S_{ij} \to S_i$ and $p_2 \colon S_{ij} \to S_j$.
    For any triple $(i,j,k)$, we denote the natural projections by $p_{12} \colon S_{ijk} \to S_{ij}$, $p_{13} \colon S_{ijk} \to S_{ik}$, $p_{23} \colon S_{ijk} \to S_{jk}$, $q_1\colon S_{ijk} \to S_i$, $q_2\colon S_{ijk} \to S_j$, and $q_3\colon S_{ijk} \to S_k$.

    \textit{Step 4.}
    We prove the uniqueness of a universal covering of a semisimple $S$-group up to a unique isomorphism.
    For universal coverings $\pi \colon G^\ssc \to G$ and $\pi'\colon {G^\ssc}' \to G$,
    we wish to show that there exists a unique isomorphism $\varphi \colon G^\ssc \to {G^\ssc}'$
    satisfying the compatibility $\pi' \circ \varphi = \pi$.
    The uniqueness of such $\varphi$ follows from Step 3.
    By Lemma \ref{l:uni_cov_bc} and Lemma \ref{l:isog_split}, there exists an étale covering $\{S_i \to S\}$ such that both $\pi_{S_i} \colon G^\ssc_{S_i} \to G_{S_i}$ and $\pi'_{S_i} \colon {G^\ssc_{S_i}}' \to G_{S_i}$ are splittable universal coverings for each $i$.
    Then there exists a unique isomorphism $\varphi_i \colon G^\ssc_{S_i} \isoto {G^\ssc_{S_i}}'$ compatible with $\pi_{S_i}$ and $\pi_{S_i}'$ by Step 2 and Step 3.
    By base change to $S_{ij}$, this yields two isomorphisms $p_1^* \varphi_{i}, p_2^* \varphi_j \colon G^\ssc_{S_{ij}} \to {G^{\prime\hs\ssc}_{S_{ij}}}$, both compatible with $\pi_{S_{ij}}$ and $\pi'_{S_{ij}}$.
    These two isomorphisms must coincide by Step 3, hence $p_1^* \varphi_i = p_2^* \varphi_j$ for all $i, j$.
    Since $\ul{\Hom}_{S\textnormal{-gp}}(G^\ssc, G^{\prime\hs\ssc})$ is an étale sheaf on $S$ (in fact, an fpqc sheaf),
    we have an equalizer diagram
    \[
    \begin{tikzcd}[column sep=large]
        \Hom_{S\textnormal{-gp}}(G^\ssc, G'\hs^\ssc) \arrow[r]
        & \prod_i \Hom_{S_i\textnormal{-gp}}(G^\ssc_{S_i}, G^{\prime\hs\ssc}_{S_i})
          \arrow[r, shift left, "p_1^*"]
          \arrow[r, shift right, "p_2^*"']
        & \prod_{i,j} \Hom_{S_{ij}\textnormal{-gp}}(G^\ssc_{S_{ij}}, G^{\prime\hs\ssc}_{S_{ij}}).
    \end{tikzcd}
    \]
    Consequently, the collection $(\varphi_i)_i \in \prod_i \Hom_{S_i\textnormal{-gp}}(G^\ssc_{S_i} , G^{\prime\hs\ssc}_{S_i})$ such that $p_1^* \varphi_i = p_2^* \varphi_j$ glues to a unique $S$-homomorphism $\varphi \colon G^\ssc \to G^{\prime\hs\ssc}$ such that $\varphi_{S_i} = \varphi_i$ for each $i$.
    Since $\pi'_{S_i} \circ \varphi_{S_i} = \pi'_{S_i} \circ \varphi_i = \pi_{S_i}$ for each $i$, the equalizer diagram implies the desired compatibility $\pi' \circ \varphi = \pi$ over $S$.
    Furthermore, the fact that each $\varphi_{S_i}$ is an isomorphism yields that $\varphi$ is also an isomorphism; see \cite[\href{https://stacks.math.columbia.edu/tag/02L4}{Tag 02L4}]{stacks-project}.

    \textit{Step 5.}
    We prove the existence of a universal covering of a semisimple $S$-group.
    Since $G$ splits étale-locally, there exists an étale covering $\{S_i \to S\}$
    such that $G_{S_i}$ is a split semisimple $S_i$-group; see \cite[Lemma 5.1.3]{Con14}.
    For each $i$, the split semisimple $S_i$-group $G_{S_i}$ admits a universal covering $\pi_i \colon G_i^\ssc \to G_{S_i}$ by Step 1.
    After base change to $S_{ij}$, both $p_1^* G_i^\ssc$ and $p_2^* G_j^\ssc$ are universal coverings of a semisimple $S_{ij}$-group $G_{S_{ij}}$.
    Then by Step 4, there exists a unique $S_{ij}$-group isomorphism
    $
    \varphi_{ij} \colon p_2^* G_j^\ssc \isoto p_1^* G_i^\ssc
    $
    compatible with $p_2^* \pi_j$ and $p_1^*\pi_i$ for all $i,j$.
    Base changing further to $S_{ijk}$, we obtain isomorphisms between universal coverings of $G_{S_{ijk}}$ as follows:
    \[
    p_{12}^* \varphi_{ij} \colon q_2^* G_j^\ssc \isoto q_1^* G_i^\ssc, \quad
    p_{23}^* \varphi_{jk} \colon q_3^* G_k^\ssc \isoto q_2^* G_j^\ssc, \quad \text{and} \quad
    p_{13}^* \varphi_{ik} \colon q_3^* G_k^\ssc \isoto q_1^* G_i^\ssc.
    \]
    By Step 3, we must have
    \[
    p^*_{12} \varphi_{ij} \circ p_{23}^* \varphi_{jk} = p_{13}^* \varphi_{ik}
    \colon q_3^* G_k^\ssc \isoto q_1^* G_i^\ssc
    \]
    for all triples $(i,j,k)$, which establishes the cocycle condition,
    hence the collection $\{G^\ssc_i,\varphi_{ij}\}$ is a descent datum;
    see \cite[Definition 4.2]{Vistoli} or \cite[\href{https://stacks.math.columbia.edu/tag/026B}{Tag 026B}]{stacks-project}.
    Since étale descent is effective for affine $S$-schemes (see \cite[Theorem 4.33]{Vistoli} or  \cite[\href{https://stacks.math.columbia.edu/tag/0245}{Tag 0245}]{stacks-project}), there is an affine $S$-group $G^\ssc$ with $\pi \colon G^\ssc \to G$ such that for each $i$, there exists an $S_i$-isomorphism $G^\ssc_{S_i} \isoto G^\ssc_i$ compatible with $\pi_{S_i}$ and $\pi_i$.

    It remains to show that $G^\ssc$ is a simply connected semisimple $S$-group and that $\pi$ is a central isogeny.
    Since $G^\ssc$ is étale-locally smooth, it is a smooth $S$-group;
    see \cite[\href{https://stacks.math.columbia.edu/tag/02VL}{Tag 02VL}]{stacks-project}.
    We see that every geometric fiber of $G^\ssc$ is simply connected semisimple by
    construction, hence $G^\ssc$ is a simply connected semisimple $S$-group; see Definition \ref{d:sc}.
    Similarly, $\pi \colon G^\ssc \to G$ is surjective since it is étale-locally surjective; see \cite[\href{https://stacks.math.columbia.edu/tag/02KV}{Tag 02KV}]{stacks-project}.
    Furthermore, since the induced map $\pi_{\bar s}$ on each geometric fiber is a central isogeny, it follows from \cite[Proposition 3.3.10]{Con14} that $\pi$ is a central isogeny over $S$.
    Therefore, $\pi \colon G^\ssc \to G$ is a universal covering.
    \end{proof}

    \begin{lemma}\label{l:inn_comm}
    Let $f \colon H \to G$ be a homomorphism of reductive $S$-groups and let $\bar f \colon H/Z_H \to G/Z_G$ be the induced homomorphism.
    For any $\bar h \in (H/Z_H)(S)$ and any $\bar g \coloneqq \bar{f}(\bar h) \in (G/Z_G)(S)$, the diagram
    \[
    \begin{tikzcd}
        H \ar[d,"f"'] \ar[r,"{c_H(\bar{h})}"] &H \ar[d,"f"] \\
        G \ar[r,"{c_G(\bar{g})}"] &G
    \end{tikzcd}
    \]
    commutes, where $c_G(\bar g) \in \Aut_S(G)$ and $c_H(\bar h) \in \Aut_S(H)$ are defined in Remark \ref{r:conj_act}.
\end{lemma}

\begin{proof}
    Let $\{S_i \to S\}$ be an fppf covering together with
    $h_i \in H(S_i)$ such that
    $\bar h_i=\bar h_{S_i}\in(H/Z_H)(S_i)$.
    Let $g_i=f(h_i)\in G(S_i)$.
    Then $\bar g_i=\bar g_{S_i}$, and
    Remark \ref{r:conj_act} gives
    \[
    c_H(\bar h)_{S_i}=\inn(h_i) \quad \text{and} \quad
    c_G(\bar g)_{S_i}=\inn(g_i).
    \]
    Since $g_i=f(h_i)$, we have a commutative diagram
    \[
    \begin{tikzcd}
        H_{S_i} \ar[d,"f_{S_i}"'] \ar[r,"\inn(h_i)"] &
        H_{S_i} \ar[d,"f_{S_i}"]\\
        G_{S_i} \ar[r,"\inn(g_i)"] &
        G_{S_i},
    \end{tikzcd}
    \]
    which implies that
    \[
    (f\circ c_H(\bar h))_{S_i}
    =
    (c_G(\bar g)\circ f)_{S_i}
    \]
    for every $i$.
    Since equality of morphisms of $S$-schemes may be checked fppf-locally,
    we conclude that
    \[
    f\circ c_H(\bar h)=c_G(\bar g)\circ f.
    \]
\end{proof}

\begin{proposition}[{\cite[Corollary 2.6 and Remark 2.7]{GA}}]
\label{p:GA}
Let $S$ be a non-empty scheme.
\begin{enumerate}
\item[\rm (i)]
    Let $\varphi\colon G_1\to G_2$ be a homomorphism of semisimple $S$-groups.
    Then there exists a unique homomorphism $\varphi^\ssc\colon G_1^\ssc\to G_2^\ssc$
    such that the following diagram commutes:
    \[
    \xymatrix@C=13mm{
    G_1^\ssc\ar[r]^-{\varphi^\ssc}\ar[d]_-{\pi_1}    &G_2^\ssc\ar[d]^-{\pi_2}  \\
    G_1\ar[r]^-\varphi                               &G_2
    }
    \]
\item[\rm (ii)] For any homomorphisms of semisimple $S$-groups
    \[ G_1\labelto{\varphi_{12} } G_2\labelto{\varphi_{23} } G_3\]
    we have
    \[ (\varphi_{23}\circ\varphi_{12})^\ssc= \varphi_{23}^\ssc\circ\varphi_{12}^\ssc\hs.\]
\end{enumerate}
\end{proposition}

\begin{remark}
    Let $G$ be a semisimple $S$-group and let $\pi \colon G^\ssc \to G$ be a universal covering.
    Using Proposition \ref{p:Conrad-existence} and Proposition \ref{p:GA}, one can easily show that $G^\ssc$ is initial among  central isogenies over $S$; namely, for every central isogeny $\pi' \colon G' \to G$, there exists a unique central isogeny $\pi'' \colon G^\ssc \to G'$ such that $\pi' \circ \pi'' = \pi$.
\end{remark}

\section{Group cohomology of a braided crossed module}
\label{app:gr-coh-Picard}

In this appendix we define a structure of abelian group
on the groups cohomology of a symmetrically braided crossed module (in particular,
in the case of a Picard crossed module). We follow Noohi \cite[Sections 3 and 4]{Noohi};
the difference is that we consider a left crossed module,
while Noohi works with right crossed modules.
Like Noohi, we omit the calculations in the version to be published in a journal.
However, unlike Noohi, we provide detailed calculations
in this arXiv version of our paper.

Let $\Gamma$ be a profinite group, and let $\big(A \labelt{\rho} G,\theta,\br \big)$
be a braided crossed module of $\Gamma$-groups; see Proposition \ref{p:br} for the definition of a braiding.
We assume that the braiding $\br$ is $\Gamma$-equivariant, that is,
${}^\upsig \{g,g'\} = \{{}^\upsig\hmm g,{}^\upsig\hmm g'\}$
for any $g,g' \in G$ and $\sigma \in \Gamma$.
We construct a natural group structure on  $H^1(\Gamma,A \to G)$,
and we show that under the additional assumption that the braiding $\br$ is symmetric,
the obtained group  is abelian.

We first prove the following lemma:

\begin{lemma}\label{l:noname0}
One has that
\begin{align}
    {\ha}^\upg s \cdot \{g,g'\} &= \{g,\rho(s)g'\} \cdot s \label{e:B.1} \\
    s \cdot \{g',g\} &= \{\rho(s)g',g\} \cdot {\ha}^\upg s \label{e:B.2} \\
    \{g,g'\}\cdot {\ha}^{g'g} s &= {\ha}^{g g'}s \cdot \{g,g'\} \label{e:B.3} \\
    \{1,g\} &= \{g,1\} = 1 \label{e:B.4} \\
    \{g,g'\} &= {\ha}^{g'g}\{g^{-1},g^{\prime\,-1}\} \label{e:B.5}
\end{align}
for any $s \in A$ and $g, g' \in G$.
\end{lemma}
\begin{proof}
    We prove \eqref{e:B.1} as follows:
    \begin{align*}
        {\ha}^\upg s \cdot \{g,g'\}
        &= \{g,\rho(s)\}\cdot s \cdot \{g,g'\}
        &&\textit{ by  } \eqref{e:Br3} \\
        &= \{g,\rho(s)\}\cdot {\ha}^{\rho(s)}\{g,g'\} \cdot s
        &&\textit{ by  } \eqref{e:CM1} \\
        &= \{g,\rho(s)g'\}\cdot s
        &&\textit{ by  } \eqref{e:Br4}.
    \end{align*}
    The proof of \eqref{e:B.2} is similar.
    We prove  \eqref{e:B.3} as follows:
    \begin{align*}
        \{g,g'\} \cdot {\ha}^{g'g} s
        &= {\ha}^{\rho(\{g,g'\})} {\ha}^{g'g}\hs s \cdot \{g,g'\}
        &&\textit{ by  } \eqref{e:CM1} \\
        &= {\ha}^{[g,g'] g'g}\hs s \cdot \{g,g'\}
        &&\textit{ by  } \eqref{e:Br1} \\
        &= {\ha}^{gg'} s \cdot \{g,g'\}.
    \end{align*}
    Equality \eqref{e:B.4} is clear since $\{1,g\} = \{\rho(1),g\} = 1 = \{g,\rho(1)\} = \{g,1\}$ by \eqref{e:Br2} and \eqref{e:Br3}.
    We prove \eqref{e:B.5} as follows:
    \begin{align*}
        \{g,g'\}
        &= \{g,1\}\cdot {\ha}^{g'\!}\{g,g^{\prime\,-1}\}^{-1}
        &&\textit{ by } \eqref{e:Br4} \\
        &= {\ha}^{g'\!}\{g,g^{\prime\,-1}\}^{-1}
        &&\textit{ by } \eqref{e:B.4} \\
        &={\ha}^{g'\!}({\ha}^\upg\{g^{-1},g^{\prime\,-1}\}^{-1} \cdot \{1,g'\})^{-1}
        &&\textit{ by } \eqref{e:Br5} \\
        &= {\ha}^{g'g}\{g^{-1},g^{\prime\,-1}\}
        &&\textit{ by } \eqref{e:B.4}
    \end{align*}
    This completes the proof of Lemma \ref{l:noname0}.
\end{proof}

We wish to construct group structures on the set of
0-(hyper-)cochains $C^0(\Gamma,A \to G) \coloneqq  \operatorname{Maps}(\Gamma,A) \times G$
and the set of 1-(hyper-)cochains
$C^1(\Gamma,A \to G) \coloneqq  \operatorname{Maps}(\Gamma \times \Gamma,A) \times \operatorname{Maps}(\Gamma,G)$.
Using the braiding $\br$, we define a group structure on $C^0(\Gamma,A \to G)$ under the product and the inverse given by
\begin{align*}
    (\varphi^1,g_1) \cdot (\varphi^2,g_2)
    &\coloneqq
    (\varphi^{1,2},g_1 g_2) \\
    (\varphi,g)^{-1} &\coloneqq  (\varphi',g^{-1})
\end{align*}
where for any $\sigma \in \Gamma$
\begin{align}
    \varphi^{1,2}_\sigma
    &=
    {\ha}^{g_1} \{g_1^{-1} \rho(\varphi^1_\sigma) {\ha\ha}^\upsig\hmm g_1,g_2\}^{-1}
    \cdot
    \varphi^1_\sigma
    \cdot
    {\ha}^{{\ha}^\upsig\hmm g_1} \varphi^2_\sigma\hs,  \label{eq:product_C0} \\
    \varphi'_\sigma
    &=
    {\ha}^{{\ha}^\upsig\hmm g^{-1}} \varphi^{-1}_\sigma
    \cdot
    {\ha}^{{\ha}^\upsig\hmm g^{-1} g} \{g^{-1}\rho(\varphi_\sigma) {\ha\ha}^\upsig\hmm g,g^{-1}\}.\notag
\end{align}
One can check that we indeed obtain a group structure on $C^0(\Gamma,A \to G)$.

\begin{remark}\label{rmk:grp_C0_ordinary}
    This group structure on $C^0(\Gamma,A \to G)$ is different from the one in \cite[Section 3.3.1]{Borovoi-Memoir},
    for which the product is defined by \eqref{e:grp_C0_Bor}.
    However, they coincide when we restrict to the set of 0-cocycles  $Z^0(\Gamma,A \to G)$, which consists of $(\varphi,g) \in C^0(\Gamma,A \to G)$ where $\varphi_{\sigma\tau} = \varphi_\sigma {\ha}^\upsig \varphi_{\tau}$ and
    $\rho(\varphi_\sigma) {\ha}^\upsig\hmm g = g$
    for any $\sigma, \tau \in \Gamma$.
\end{remark}

\begin{lemma}\label{l:grp_str_c1}
    The set $C^1(\Gamma, A \to G)$ becomes a group with the operations defined as follows
    (for any $(u^i, \psi^i) \in C^1(\Gamma,A \to G)$ and $\sigma, \tau \in \Gamma$) :
    \begin{enumerate}
        \item The product $(u^1,\psi^1) \cdot (u^2,\psi^2)$ is the pair $(u^{1,2},\psi^{1,2})$ where
    \begin{align}
        u^{1,2}_{\sigma, \tau}
        &=
        u^1_{\sigma, \tau} \cdot {\ha}^{\psi^1_\sigma {\ha}^\upsig \psi^1_\tau} u^2_{\sigma, \tau} \cdot {\ha}^{\psi^1_\sigma} \left\{ {\ha}^\upsig \psi^1_\tau, \psi^2_\sigma \right\},
        \label{eq:def_prod_1} \\
        \psi^{1,2}_\sigma &= \psi^1_\sigma \psi^2_\sigma; \label{eq:def_prod_2}
    \end{align}
     \item The inverse of $(u,\psi)$ is the pair $(u',\psi')$ where
    \begin{align}
        u'_{\sigma,\tau} &\coloneqq  {\ha}^{\psi_{\sigma\tau}^{-1}} u_{\sigma,\tau}^{-1} \cdot  \{{\ha}^\upsig\psi_\tau^{-1},\psi_\sigma^{-1}\},  \label{eq:def_inv_1} \\
        \psi'_\sigma &\coloneqq  \psi_\sigma^{-1}; \label{eq:def_inv_2}
    \end{align}
    \item The identity element is $(1,1)$.
    \end{enumerate}
\end{lemma}

\begin{proof}
    For any $(u^1,\psi^1), (u^2,\psi^2), (u^3,\psi^3) \in C^1(\Gamma,A \to G)$,
    we denote
    $(u^{(1,2),3},\psi^{(1,2),3}) \coloneqq  ((u^1,\psi^1) \cdot (u^2,\psi^2)) \cdot (u^3,\psi^3)$
    and
    $(u^{1,(2,3)},\psi^{1,(2,3)}) \coloneqq  (u^1,\psi^1) \cdot ((u^2,\psi^2) \cdot (u^3,\psi^3))$.
    To prove the associativity, we must show that $u^{(1,2),3}_{\sigma,\tau} = u^{1,(2,3)}_{\sigma,\tau}$ and $\psi^{(1,2),3}_{\sigma} = \psi^{1,(2,3)}_{\sigma}$ for
    any $\sigma, \tau \in \Gamma$, which are proved as follows, respectively:
    \begin{align*}
        u^{(1,2),3}_{\sigma,\tau}
        &= u^{1,2}_{\sigma,\tau}
        \cdot
        {\ha}^{\psi^1_\sigma \psi^2_\sigma {\ha}^\upsig \psi^1_\tau {\ha}^\upsig \psi^2_\tau} u^3_{\sigma,\tau}
        \cdot
        {\ha}^{\psi^1_\sigma \psi^2_\sigma}
        \{{\ha}^\upsig \psi^1_\tau {\ha}^\upsig \psi^2_\tau,\psi^3_\sigma\}
        && \textit{ by } \eqref{eq:def_prod_1}, \eqref{eq:def_prod_2} \\
        &= (u^1_{\sigma, \tau} \cdot {\ha}^{\psi^1_\sigma {\ha}^\upsig \psi^1_\tau} u^2_{\sigma, \tau} \cdot {\ha}^{\psi^1_\sigma} \left\{ {\ha}^\upsig \psi^1_\tau, \psi^2_\sigma \right\})
        \cdot
        {\ha}^{\psi^1_\sigma \psi^2_\sigma {\ha}^\upsig \psi^1_\tau {\ha}^\upsig \psi^2_\tau} u^3_{\sigma,\tau}
        \cdot
        {\ha}^{\psi^1_\sigma \psi^2_\sigma}
        \{{\ha}^\upsig \psi^1_\tau {\ha}^\upsig \psi^2_\tau,\psi^3_\sigma\}
        && \textit{ by } \eqref{eq:def_prod_1} \\
        &= u^1_{\sigma, \tau} \cdot {\ha}^{\psi^1_\sigma {\ha}^\upsig \psi^1_\tau} u^2_{\sigma, \tau}
        \cdot
        {\ha}^{\psi^1_\sigma {\ha}^\upsig \psi^1_\tau \psi^2_\sigma  {\ha}^\upsig \psi^2_\tau} u^3_{\sigma,\tau}
        \cdot
        {\ha}^{\psi^1_\sigma} \left\{ {\ha}^\upsig \psi^1_\tau, \psi^2_\sigma \right\}
        \cdot
        {\ha}^{\psi^1_\sigma \psi^2_\sigma}
        \{{\ha}^\upsig \psi^1_\tau {\ha}^\upsig \psi^2_\tau,\psi^3_\sigma\}
        && \textit{ by } \eqref{e:B.3} \\
        &=u^1_{\sigma,\tau}
        \cdot {\ha}^{\psi^1_\sigma {\ha}^\upsig \psi^1_\tau}
        (u^2_{\sigma,\tau} \cdot {\ha}^{\psi^2_\sigma {\ha}^\upsig \psi^3_\tau} u^3_{\sigma,\tau})
        \cdot
        {\ha}^{\psi^1_\sigma} \left\{ {\ha}^\upsig \psi^1_\tau, \psi^2_\sigma \right\}
        \cdot
        {\ha}^{\psi^1_\sigma \psi^2_\sigma {\ha}^\upsig \psi^1_\tau}
        \{{\ha}^\upsig \psi^2_\tau,\psi^3_\sigma\}
        \cdot
        {\ha}^{\psi^1_\sigma \psi^2_\sigma}
        \{{\ha}^\upsig \psi^1_\tau,\psi^3_\sigma\}
        && \textit{ by } \eqref{e:Br5} \\
        &=u^1_{\sigma,\tau}
        \cdot {\ha}^{\psi^1_\sigma {\ha}^\upsig \psi^1_\tau}
        (u^2_{\sigma,\tau} \cdot {\ha}^{\psi^2_\sigma {\ha}^\upsig \psi^3_\tau} u^3_{\sigma,\tau})
        \cdot
        {\ha}^{\psi^1_\sigma {\ha}^\upsig \psi^1_\tau \psi^2_\sigma}\{{\ha}^\upsig \psi^2_\tau,\psi^3_\sigma\}
        \cdot
        {\ha}^{\psi^1_\sigma}
        \{{\ha}^\upsig \psi^1_\tau, \psi^2_\sigma\}
        \cdot
        {\ha}^{\psi^1_\sigma \psi^2_\sigma}
        \{{\ha}^\upsig \psi^1_\tau,\psi^3_\sigma\}
        && \textit{ by } \eqref{e:B.3} \\
        &
        =u^1_{\sigma,\tau}
        \cdot {\ha}^{\psi^1_\sigma {\ha}^\upsig \psi^1_\tau}
        (u^2_{\sigma,\tau} \cdot {\ha}^{\psi^2_\sigma {\ha}^\upsig \psi^3_\tau} u^3_{\sigma,\tau}
        \cdot
        {\ha}^{\psi^2_\sigma}\{{\ha}^\upsig \psi^2_\tau,\psi^3_\sigma\})
        \cdot
        {\ha}^{\psi^1_\sigma}
        \{{\ha}^\upsig \psi^1_\tau, \psi^2_\sigma \psi^3_\sigma\}
        && \textit{ by } \eqref{e:Br5} \\
        &
        =u^1_{\sigma,\tau}
        \cdot {\ha}^{\psi^1_\sigma {\ha}^\upsig \psi^1_\tau} u^{2,3}_{\sigma,\tau}
        \cdot
        {\ha}^{\psi^1_\sigma}
        \{{\ha}^\upsig \psi^1_\tau, \psi^2_\sigma \psi^3_\sigma\}
        && \textit{ by } \eqref{eq:def_prod_1}  \\
        &=u^{1,(2,3)}_{\sigma,\tau}
        && \textit{ by } \eqref{eq:def_prod_1}, \eqref{eq:def_prod_2} \\
        \psi^{(1,2),3}_\sigma
        &= \psi^1_\sigma \psi^2_\sigma \psi^3_\sigma = \psi^{1,(2,3)}_\sigma
        &&\textit{ by } \eqref{eq:def_prod_2}.
    \end{align*}
    To prove the other group axioms, it suffices to show that  for any $(u,\psi) \in C^1(\Gamma,A\to G)$, $(1,1)$ is the right identity and $(u',\psi')$ is the right inverse of $(u,\psi)$.
    It is clear that $(u,\psi) \cdot (1,1) = (u,\psi)$ by
    \eqref{e:B.4}, \eqref{eq:def_prod_1}, and \eqref{eq:def_prod_2}.
    If we denote $(u'',\psi'') \coloneqq  (u,\psi) * (u',\psi')$, then for any $\sigma,\tau \in \Gamma$ we have:
    \begin{align*}
        u''_{\sigma,\tau}
        &=
        u_{\sigma,\tau}
        \cdot {\ha}^{\psi_\sigma {\ha}^{\upsig} \psi_\tau} u'_{\sigma,\tau}
        \cdot {\ha}^{\psi_\sigma}\{{\ha}^\upsig \psi_\tau,\psi^{-1}_\sigma\}
        &&\textit{ by }\eqref{eq:def_prod_1}, \eqref{eq:def_inv_2} \\
        &=
        u_{\sigma,\tau}
        \cdot {\ha}^{\psi_\sigma {\ha}^{\upsig} \psi_\tau} ({\ha\ha}^{\psi_{\sigma\tau}^{-1}} u_{\sigma,\tau}^{-1} \cdot
        \{{\ha}^\upsig\psi_\tau^{-1},\psi_\sigma^{-1}\})
        \cdot {\ha}^{\psi_\sigma}\{{\ha}^\upsig \psi_\tau,\psi^{-1}_\sigma\}
        &&\textit{ by } \eqref{eq:def_inv_1} \\
        &=
        u_{\sigma,\tau}
        \cdot  {\ha}^{\rho(u^{-1}_{\sigma,\tau})} u_{\sigma,\tau}^{-1}
        \cdot
        {\ha}^{\psi_\sigma {\ha}^{\upsig} \psi_\tau}\{{\ha}^\upsig\psi_\tau^{-1},\psi_\sigma^{-1}\}
        \cdot {\ha}^{\psi_\sigma}\{{\ha}^\upsig \psi_\tau,\psi^{-1}_\sigma\}
        &&\textit{ by } \eqref{eq:def1_1-cocycle} \\
        &=
        {\ha}^{\psi_\sigma {\ha}^{\upsig} \psi_\tau}\{{\ha}^\upsig\psi_\tau^{-1},\psi_\sigma^{-1}\}
        \cdot {\ha}^{\psi_\sigma}\{{\ha}^\upsig \psi_\tau,\psi^{-1}_\sigma\}
        &&\textit{ by } \eqref{e:CM1} \\
        &=
        {\ha}^{\psi_\sigma}\{{\ha}^{\upsig} \psi_\tau {\ha}^\upsig\psi_\tau^{-1},\psi_\sigma^{-1}\}
        &&\textit{ by } \eqref{e:Br5} \\
        &= 1
        &&\textit{ by } \eqref{e:B.4} \\
        \psi''_\sigma &= \psi^{-1}_\sigma \psi_\sigma = 1
        &&\textit{ by } \eqref{eq:def_prod_2}, \eqref{eq:def_inv_2}.
    \end{align*}
    This shows that $(u'',\psi'') = (1,1)$, which completes the proof of Lemma \ref{l:grp_str_c1}.
\end{proof}

Recall that the set of 1-cocycles $Z^1(\Gamma,A \to G)$ is a subset of $C^1(\Gamma,A \to G)$ consisting of the elements $(u,\psi)$
satisfying  conditions \eqref{eq:def1_1-cocycle} and \eqref{eq:def2_1-cocycle}.

\begin{lemma}\label{l:grp_str_Z1}
    Under the group structure defined in Lemma \ref{l:grp_str_c1},
    the set of 1-cocycles $Z^1(\Gamma,A \to G)$ forms a subgroup of $C^1(\Gamma,A \to G)$.
\end{lemma}

\begin{proof}
    It suffices to show that $Z^1(\Gamma,A \to G)$ is closed under the multiplication and inversion defined in Lemma \ref{l:grp_str_c1}, by proving
    \eqref{eq:def1_1-cocycle} and \eqref{eq:def2_1-cocycle}.
    For any $(u^1,\psi^1), (u^2,\psi^2) \in Z^1(\Gamma,A \to G)$ and $\sigma, \tau, \nu \in \Gamma$, their product $(u^{1,2},\psi^{1,2})$ satisfies
    \eqref{eq:def1_1-cocycle} and \eqref{eq:def2_1-cocycle} as follows:
    \begin{align*}
    \rho(u^{1,2}_{\sigma,\tau}) \cdot \psi^{1,2}_\sigma \cdot {\ha}^\upsig \psi^{1,2}_{\tau}
    &=
    \rho(u^1_{\sigma,\tau} \cdot {\ha}^{\psi^1_\sigma {\ha}^\upsig \psi^1_\tau} u^2_{\sigma,\tau}
    \cdot {\ha}^{\psi^1_\sigma} \{{\ha}^\upsig \psi^1_\tau,\psi^2_\sigma\}) \cdot (\psi^1_\sigma \psi^2_\sigma) \cdot ({\ha}^\upsig \psi^1_\tau {\ha}^\upsig \psi^2_\tau)
     &&\textit{ by } \eqref{eq:def_prod_1}, \eqref{eq:def_prod_2} \\
      &=
    \rho({\ha}^{\rho(u^1_{\sigma,\tau})\psi^1_\sigma {\ha}^\upsig \psi^1_\tau} u^2_{\sigma,\tau} \cdot u^1_{\sigma,\tau}
    \cdot {\ha}^{\psi^1_\sigma} \{{\ha}^\upsig \psi^1_\tau,\psi^2_\sigma\}) \cdot (\psi^1_\sigma \psi^2_\sigma) \cdot ({\ha}^\upsig \psi^1_\tau {\ha}^\upsig \psi^2_\tau)
     &&\textit{ by } \eqref{e:CM1} \\
    &=
    \rho({\ha\ha}^{\psi^1_{\sigma\tau}} u^2_{\sigma,\tau} )\cdot \rho(u^1_{\sigma,\tau}) \cdot \rho( {\ha}^{\psi^1_\sigma} \{{\ha}^\upsig \psi^1_\tau,\psi^2_\sigma\}) \cdot (\psi^1_\sigma \psi^2_\sigma) \cdot ({\ha}^\upsig \psi^1_\tau {\ha}^\upsig \psi^2_\tau)
     &&\textit{ by } \eqref{eq:def1_1-cocycle} \\
    &=\psi^1_{\sigma \tau} \rho(u^2_{\sigma,\tau}) (\psi^1_{\sigma \tau})^{-1} \cdot \rho(u^1_{\sigma,\tau}) \cdot \psi^1_{\sigma} [{\ha}^\upsig \psi^1_\tau,\psi^2_\sigma] \cdot \psi^2_\sigma \cdot ({\ha}^\upsig \psi^1_\tau {\ha}^\upsig \psi^2_\tau)
    &&\textit{ by } \eqref{e:CM2}, \eqref{e:Br1} \\
    &= \psi^1_{\sigma \tau} \rho(u^2_{\sigma,\tau}) (\psi^1_{\sigma \tau})^{-1} \cdot (\rho(u^1_{\sigma,\tau}) \cdot \psi^1_{\sigma} {\ha}^\upsig \psi^1_\tau) \cdot \psi^2_\sigma {\ha}^\upsig \psi^2_\tau \\
    &= \psi^1_{\sigma \tau} \psi^2_{\sigma \tau}
    &&\textit{ by } \eqref{eq:def1_1-cocycle} \\
    &= \psi^{1,2}_{\sigma \tau}
    &&\textit{ by } \eqref{eq:def_prod_2}.
    \end{align*}
        \begin{align*}
        &u^{1,2}_{\sigma,\tau \nu} \cdot {\ha}^{\psi^{1,2}_\sigma \sigma}u^{1,2}_{\tau,\nu}= \\
        &=(u^1_{\sigma, \tau \nu} \cdot {\ha}^{\psi^1_\sigma {\ha}^\upsig \psi^1_{\tau \nu}} u^2_{\sigma,\tau \nu} \cdot {\ha}^{\psi^1_\sigma} \left\{ {\ha}^\upsig \psi^1_{\tau \nu}, \psi^2_\sigma \right\})
         \cdot
         {\ha}^{\psi^1_\sigma \psi^2_\sigma \sigma}
         (u^1_{\tau, \nu} \cdot {\ha}^{\psi^1_\tau {\ha}^\tau \psi^1_\nu} u^2_{\tau,\nu} \cdot {\ha}^{\psi^1_\tau} \left\{ {\ha}^\tau \psi^1_\nu, \psi^2_\tau \right\})
         &&\textit{ by } \eqref{eq:def_prod_1} \\
         & ={\ha}^{\rho(u^1_{\sigma,\tau \nu})\psi^1_{\sigma}{\ha}^\upsig \psi^1_{\tau \nu}} u^2_{\sigma,\tau \nu}
         \cdot
         u^1_{\sigma, \tau \nu}
         \cdot
         {\ha}^{\psi^1_\sigma} (\left\{ {\ha}^\upsig \psi^1_{\tau \nu}, \psi^2_\sigma \right\}
         \cdot
         {\ha}^{\psi^2_\sigma \sigma} u^1_{\tau, \nu} \cdot {\ha}^{{\psi^2_\sigma} {\ha}^\upsig \psi^1_\tau {\ha}^{\upst} \psi^1_\nu} {\ha}^\upsig u^2_{\tau,\nu} \cdot {\ha}^{\psi^2_\sigma {\ha}^\upsig \psi^1_\tau} \left\{ {\ha}^{\upst} \psi^1_\nu, {\ha}^\upsig \psi^2_\tau \right\})
         &&\textit{ by } \eqref{e:CM1}\\
         & ={\ha}^{\psi^1_{\sigma \tau \nu}} u^2_{\sigma,\tau \nu}
         \cdot
         u^1_{\sigma, \tau \nu}
         \cdot
         {\ha}^{\psi^1_\sigma} (\left\{ {\ha}^\upsig \psi^1_{\tau \nu}, \psi^2_\sigma \right\}
         \cdot
         {\ha}^{\psi^2_\sigma \sigma} u^1_{\tau, \nu} \cdot {\ha}^{{\psi^2_\sigma} {\ha}^\upsig \psi^1_\tau {\ha}^{\upst} \psi^1_\nu} {\ha}^\upsig u^2_{\tau,\nu} \cdot {\ha}^{\psi^2_\sigma {\ha}^\upsig \psi^1_\tau} \left\{ {\ha}^{\upst} \psi^1_\nu, {\ha}^\upsig \psi^2_\tau \right\})
         &&\textit{ by } \eqref{eq:def1_1-cocycle}\\
         & ={\ha}^{\psi^1_{\sigma \tau \nu}} u^2_{\sigma,\tau \nu}
         \cdot
         u^1_{\sigma, \tau \nu}
         \cdot
         {\ha}^{\psi^1_\sigma} (\left\{\rho({\ha}^\upsig  u^1_{\tau,\nu}) {\ha}^\upsig \psi^1_{\tau} {\ha}^{\upst} \psi^1_\nu, \psi^2_\sigma \right\}
         \cdot
         {\ha}^{\psi^2_\sigma \sigma} u^1_{\tau, \nu} \cdot {\ha}^{{\psi^2_\sigma} {\ha}^\upsig \psi^1_\tau {\ha}^{\upst} \psi^1_\nu} {\ha}^\upsig u^2_{\tau,\nu} \cdot {\ha}^{\psi^2_\sigma {\ha}^\upsig \psi^1_\tau} \left\{ {\ha}^{\upst} \psi^1_\nu, {\ha}^\upsig \psi^2_\tau \right\})
         &&\textit{ by } \eqref{eq:def1_1-cocycle} \\
         & ={\ha}^{\psi^1_{\sigma \tau \nu}} u^2_{\sigma,\tau \nu}
         \cdot
         u^1_{\sigma, \tau \nu}
         \cdot
         {\ha}^{\psi^1_\sigma} {\ha}^\upsig u^1_{\tau, \nu}
         \cdot
         {\ha}^{\psi^1_\sigma} (\left\{ {\ha}^\upsig \psi^1_{\tau}{\ha}^{\upst} \psi^1_\nu, \psi^2_\sigma \right\}
         \cdot
         {\ha}^{{\psi^2_\sigma} {\ha}^\upsig \psi^1_\tau {\ha}^{\upst} \psi^1_\nu} {\ha}^\upsig u^2_{\tau,\nu} \cdot {\ha}^{\psi^2_\sigma {\ha}^\upsig \psi^1_\tau} \left\{ {\ha}^{\upst} \psi^1_\nu, {\ha}^\upsig \psi^2_\tau \right\})
         &&\textit{ by } \eqref{e:B.2} \\
         & ={\ha}^{\psi^1_{\sigma \tau \nu}} u^2_{\sigma,\tau \nu}
         \cdot
         u^1_{\sigma \tau, \nu}
         \cdot
        u^1_{\sigma, \tau}
         \cdot
         {\ha}^{\psi^1_\sigma} (\left\{{\ha}^\upsig \psi^1_{\tau} {\ha}^{\upst} \psi^1_\nu, \psi^2_\sigma \right\}
         \cdot
         {\ha}^{{\psi^2_\sigma} {\ha}^\upsig \psi^1_\tau {\ha}^{\upst} \psi^1_\nu} {\ha}^\upsig u^2_{\tau,\nu} \cdot {\ha}^{\psi^2_\sigma {\ha}^\upsig \psi^1_\tau} \left\{ {\ha}^{\upst} \psi^1_\nu, {\ha}^\upsig \psi^2_\tau \right\})
         &&\textit{ by } \eqref{eq:def2_1-cocycle}\\
         & ={\ha}^{\psi^1_{\sigma \tau \nu}} u^2_{\sigma,\tau \nu}
         \cdot
         u^1_{\sigma \tau, \nu}
         \cdot
        u^1_{\sigma, \tau}
         \cdot
         {\ha}^{\psi^1_\sigma} (
         {\ha}^{{\ha}^\upsig \psi^1_\tau {\ha}^{\upst} \psi^1_\nu {\psi^2_\sigma} } {\ha}^\upsig u^2_{\tau,\nu}
         \cdot
         \{{\ha}^\upsig \psi^1_{\tau} {\ha}^{\upst} \psi^1_\nu, \psi^2_\sigma \}
         \cdot
        {\ha}^{\psi^2_\sigma {\ha}^\upsig \psi^1_\tau} \left\{ {\ha}^{\upst} \psi^1_\nu, {\ha}^\upsig \psi^2_\tau \right\})
        &&\textit{ by } \eqref{e:B.3} \\
        & ={\ha}^{\psi^1_{\sigma \tau \nu}} u^2_{\sigma,\tau \nu}
         \cdot
         u^1_{\sigma \tau, \nu}
         \cdot
        u^1_{\sigma, \tau}
         \cdot
         {\ha}^{\psi^1_\sigma} (
         {\ha}^{{\ha}^\upsig \psi^1_\tau {\ha}^{\upst} \psi^1_\nu {\psi^2_\sigma} } {\ha}^\upsig u^2_{\tau,\nu}
         \cdot
         {\ha}^{{\ha}^\upsig \psi^1_{\tau}}\{{\ha}^{\upst} \psi^1_\nu, \psi^2_\sigma\}
         \{{\ha}^{\upsig} \psi^1_\tau, \psi^2_\sigma \}
         \cdot
        {\ha}^{\psi^2_\sigma {\ha}^\upsig \psi^1_\tau} \left\{ {\ha}^{\upst} \psi^1_\nu, {\ha}^\upsig \psi^2_\tau \right\})
        &&\textit{ by } \eqref{e:Br5} \\
        & ={\ha}^{\psi^1_{\sigma \tau \nu}} u^2_{\sigma,\tau \nu}
         \cdot
         u^1_{\sigma \tau, \nu}
         \cdot
        u^1_{\sigma, \tau}
         \cdot
         {\ha}^{\psi^1_\sigma} (
         {\ha}^{{\ha}^\upsig \psi^1_\tau {\ha}^{\upst} \psi^1_\nu {\psi^2_\sigma} } {\ha}^\upsig u^2_{\tau,\nu}
         \cdot
         {\ha}^{{\ha}^\upsig \psi^1_{\tau}}\{{\ha}^{\upst} \psi^1_\nu, \psi^2_\sigma\}
         \cdot
         {\ha}^{{\ha}^\upsig \psi^1_\tau \psi^2_\sigma}\left\{ {\ha}^{\upst} \psi^1_\nu, {\ha}^\upsig \psi^2_\tau \right\}
         \cdot
         \{{\ha}^{\upsig} \psi^1_\tau, \psi^2_\sigma \})
         &&\textit{ by } \eqref{e:B.3}\\
         & ={\ha}^{\psi^1_{\sigma \tau \nu}} u^2_{\sigma,\tau \nu}
         \cdot
         u^1_{\sigma \tau, \nu}
         \cdot
        u^1_{\sigma, \tau}
         \cdot
         {\ha}^{\psi^1_\sigma} (
         {\ha}^{{\ha}^\upsig \psi^1_\tau {\ha}^{\upst} \psi^1_\nu {\psi^2_\sigma} } {\ha}^\upsig u^2_{\tau,\nu}
         \cdot
         {\ha}^{{\ha}^\upsig \psi^1_{\tau}}\{{\ha}^{\upst} \psi^1_\nu, \psi^2_\sigma {\ha}^\upsig \psi^2_\tau \}
         \cdot
         \{{\ha}^{\upsig} \psi^1_\tau, \psi^2_\sigma \})
         &&\textit{ by } \eqref{e:Br4} \\
         & ={\ha}^{\psi^1_{\sigma \tau \nu}} u^2_{\sigma,\tau \nu}
         \cdot
         {\ha}^{\rho(u^1_{\sigma \tau,\nu})(\rho(u^1_{\sigma,\tau}) \psi^1_\sigma {\ha}^\upsig \psi^1_\tau) {\ha}^{\upst} \psi^1_\nu {\psi^2_\sigma} } {\ha}^\upsig u^2_{\tau,\nu}
        \cdot
          u^1_{\sigma \tau, \nu}
         \cdot
        u^1_{\sigma, \tau}
         \cdot
         {\ha}^{\psi^1_\sigma} (
         {\ha}^{{\ha}^\upsig \psi^1_{\tau}}\{{\ha}^{\upst} \psi^1_\nu, \psi^2_\sigma {\ha}^\upsig \psi^2_\tau \}
         \cdot
         \{{\ha}^{\upsig} \psi^1_\tau, \psi^2_\sigma \})
         &&\textit{ by } \eqref{e:CM1} \\
         & ={\ha}^{\psi^1_{\sigma \tau \nu}} u^2_{\sigma,\tau \nu}
         \cdot
         {\ha}^{\psi^1_{\sigma \tau \nu}\psi^2_\sigma} {\ha}^\upsig u^2_{\tau,\nu}
        \cdot
          u^1_{\sigma \tau, \nu}
         \cdot
        u^1_{\sigma, \tau}
         \cdot
         {\ha}^{\psi^1_\sigma} (
         {\ha}^{{\ha}^\upsig \psi^1_{\tau}}\{{\ha}^{\upst} \psi^1_\nu, \psi^2_\sigma {\ha}^\upsig \psi^2_\tau \}
         \cdot
         \{{\ha}^{\upsig} \psi^1_\tau, \psi^2_\sigma \})
         &&\textit{ by } \eqref{eq:def1_1-cocycle} \\
         & ={\ha}^{\psi^1_{\sigma \tau \nu}} u^2_{\sigma\tau, \nu}
         \cdot
         {\ha}^{\psi^1_{\sigma \tau \nu}} u^2_{\sigma,\tau}
        \cdot
          u^1_{\sigma \tau, \nu}
         \cdot
        u^1_{\sigma, \tau}
         \cdot
         {\ha}^{\psi^1_\sigma} (
         {\ha}^{{\ha}^\upsig \psi^1_{\tau}}\{{\ha}^{\upst} \psi^1_\nu, \psi^2_\sigma {\ha}^\upsig \psi^2_\tau \}
         \cdot
         \{{\ha}^{\upsig} \psi^1_\tau, \psi^2_\sigma \})
         &&\textit{ by } \eqref{eq:def2_1-cocycle} \\
         & ={\ha}^{\psi^1_{\sigma \tau \nu}} u^2_{\sigma\tau, \nu}
         \cdot
         {\ha}^{\rho(u^1_{\sigma \tau,\nu})(\rho(u^1_{\sigma,\tau}) \psi^1_\sigma {\ha}^\upsig \psi^1_\tau) {\ha}^{\upst} \psi^1_\nu {\psi^2_\sigma} } {\ha}^\upsig u^2_{\sigma,\tau}
        \cdot
          u^1_{\sigma \tau, \nu}
         \cdot
        u^1_{\sigma, \tau}
         \cdot
         {\ha}^{\psi^1_\sigma} (
         {\ha}^{{\ha}^\upsig \psi^1_{\tau}}\{{\ha}^{\upst} \psi^1_\nu, \psi^2_\sigma {\ha}^\upsig \psi^2_\tau \}
         \cdot
         \{{\ha}^{\upsig} \psi^1_\tau, \psi^2_\sigma \})
         &&\textit{ by } \eqref{eq:def1_1-cocycle} \\
         &={\ha}^{\psi^1_{\sigma \tau \nu}} u^2_{\sigma \tau, \nu}
         \cdot
         u^1_{\sigma\tau,\nu}
         \cdot
         u^1_{\sigma,\tau}
         \cdot
         {\ha}^{\psi^1_\sigma {\ha}^\upsig \psi^1_\tau}({\ha}^{{\ha}^{\upst} \psi^1_\nu} u^2_{\sigma, \tau} \cdot \left\{ {\ha}^{\upst} \psi^1_\nu, \psi^2_{\sigma}{\ha}^\upsig \psi^2_\tau \right\} )
         \cdot
         {\ha}^{\psi^1_\sigma} \left\{ {\ha}^\upsig \psi^1_\tau, \psi^2_\sigma \right\}
         &&\textit{ by } \eqref{e:CM1} \\
         &={\ha}^{\psi^1_{\sigma \tau \nu}} u^2_{\sigma \tau, \nu}
         \cdot
         u^1_{\sigma\tau,\nu}
         \cdot
         u^1_{\sigma,\tau}
         \cdot
         {\ha}^{\psi^1_\sigma {\ha}^\upsig \psi^1_\tau}(\left\{ {\ha}^{\upst} \psi^1_\nu, \rho(u^2_{\sigma,\tau}) \psi^2_{\sigma}{\ha}^\upsig \psi^2_\tau \right\} \cdot u^2_{\sigma, \tau})
         \cdot
         {\ha}^{\psi^1_\sigma} \left\{ {\ha}^\upsig \psi^1_\tau, \psi^2_\sigma \right\}
         &&\textit{ by } \eqref{e:B.1} \\
         &={\ha}^{\rho(u^1_{\sigma\tau,\nu})\psi^1_{\sigma \tau}{\ha}^{\upst} \psi^1_\nu} u^2_{\sigma \tau, \nu}
         \cdot
         u^1_{\sigma\tau,\nu}
         \cdot
         u^1_{\sigma,\tau}
         \cdot
         {\ha}^{\psi^1_\sigma {\ha}^\upsig \psi^1_\tau}(\left\{ {\ha}^{\upst} \psi^1_\nu, \psi^2_{\sigma \tau}\right\} \cdot u^2_{\sigma, \tau})
         \cdot
         {\ha}^{\psi^1_\sigma} \left\{ {\ha}^\upsig \psi^1_\tau, \psi^2_\sigma \right\}
         &&\textit{ by } \eqref{eq:def1_1-cocycle} \\
         &=u^1_{\sigma\tau, \nu} \cdot {\ha}^{\psi^1_{\sigma\tau} {\ha}^{\upst} \psi^1_\nu}\cdot u^2_{\sigma\tau, \nu}
         \cdot {\ha}^{\rho(u^1_{\sigma,\tau})\psi^1_{\sigma}{\ha}^\upsig \psi^1_\tau} \left\{ {\ha}^{\upst} \psi^1_\nu, \psi^2_{\sigma \tau}\right\}
         \cdot
         u^1_{\sigma, \tau} \cdot {\ha}^{\psi^1_\sigma {\ha}^\upsig \psi^1_\tau} u^2_{\sigma, \tau} \cdot {\ha}^{\psi^1_\sigma} \left\{ {\ha}^\upsig \psi^1_\tau, \psi^2_\sigma \right\}
         &&\textit{ by } \eqref{e:CM1} \\
         &=(u^1_{\sigma\tau, \nu} \cdot {\ha}^{\psi^1_{\sigma\tau} {\ha}^{\upst} \psi^1_\nu}\cdot u^2_{\sigma\tau, \nu}
         \cdot {\ha}^{\psi^1_{\sigma\tau}} \left\{ {\ha}^{\upst} \psi^1_\nu, \psi^2_{\sigma \tau}\right\})
         \cdot
         (u^1_{\sigma, \tau} \cdot {\ha}^{\psi^1_\sigma {\ha}^\upsig \psi^1_\tau} u^2_{\sigma, \tau} \cdot {\ha}^{\psi^1_\sigma} \left\{ {\ha}^\upsig \psi^1_\tau, \psi^2_\sigma \right\})
         &&\textit{ by } \eqref{eq:def1_1-cocycle}\\
         &= u^{1,2}_{\sigma \tau,\nu} \cdot u^{1,2}_{\sigma,\tau}
         &&\textit{ by } \eqref{eq:def_prod_1}.
    \end{align*}

    For any $(u,\psi) \in Z^1(\Gamma,A \to G)$, its inverse $(u',\psi')$ satisfies \eqref{eq:def1_1-cocycle} as follows:
    \begin{align*}
        \rho(u'_{\sigma,\tau}) \cdot \psi'_\sigma \cdot {\ha}^\upsig \psi'_\tau
        &= \rho({\ha\ha}^{\psi^{-1}_{\sigma \tau}} u^{-1}_{\sigma,\tau}) \cdot \rho(\{{\ha}^\upsig \psi_\tau^{-1},\psi_\sigma^{-1}\}) \cdot \psi^{-1}_\sigma \cdot {\ha}^\upsig \psi^{-1}_\tau \\
        &= \psi^{-1}_{\sigma \tau} \rho(u^{-1}_{\sigma,\tau}) \psi_{\sigma \tau}\cdot \rho(\{{\ha}^\upsig \psi_\tau^{-1},\psi_\sigma^{-1}\}) \cdot \psi^{-1}_\sigma \cdot {\ha}^\upsig \psi^{-1}_\tau
        &&\textit{ by } \eqref{e:CM2} \\
        &= \psi^{-1}_{\sigma \tau} \rho(u^{-1}_{\sigma,\tau}) \psi_{\sigma \tau}
        \cdot
        [{\ha}^\upsig \psi_\tau^{-1},\psi_\sigma^{-1}]
        \cdot \psi^{-1}_\sigma \cdot {\ha}^\upsig \psi^{-1}_\tau
        &&\textit{ by } \eqref{e:Br1} \\
        &= \psi^{-1}_{\sigma \tau}
        \cdot\rho(u^{-1}_{\sigma,\tau})
        \cdot
        \rho(u_{\sigma,\tau})\cdot \psi_\sigma \cdot {\ha}^\upsig \psi_\tau
        \cdot
        [{\ha}^\upsig \psi_\tau^{-1},\psi_\sigma^{-1}]
        \cdot \psi^{-1}_\sigma \cdot {\ha}^\upsig \psi^{-1}_\tau
        &&\textit{ by } \eqref{eq:def1_1-cocycle} \\
        &= \psi'_{\sigma \tau}\hs.
    \end{align*}
    To check \eqref{eq:def2_1-cocycle} for $(u',\psi')$, we use the equality
    \begin{equation}\label{eq:u'_inverse}
    \begin{aligned}
    {\ha}^{\psi_{\sigma \tau}} {u'}_{\sigma,\tau}^{-1}
        =
        u_{\sigma,\tau}\{\psi_\sigma,{\ha}^\upsig \psi_\tau\}, \text{ proved as }
        {\ha}^{\psi_{\sigma \tau}} {u'}^{-1}_{\sigma, \tau}
        &= {\ha}^{\psi_{\sigma \tau}}(\{{\ha}^\upsig \psi_\tau^{-1},\psi^{-1}_\sigma\})^{-1} \cdot u_{\sigma,\tau}
        &&\textit{ by } \eqref{eq:def_inv_1} \\
        &= {\ha}^{\psi_{\sigma \tau}{\ha}^\upsig \psi_\tau^{-1} \psi_\sigma^{-1}}\{{\ha}^\upsig \psi_\tau,\psi_\sigma,\}^{-1} \cdot u_{\sigma,\tau}
        &&\textit{ by } \eqref{e:B.5} \\
        &= {\ha}^{\rho(u_{\sigma,\tau})}\{{\ha}^\upsig \psi_\tau,\psi_\sigma\}^{-1} \cdot u_{\sigma,\tau}
        &&\textit{ by } \eqref{eq:def1_1-cocycle} \\
        &= u_{\sigma,\tau} \cdot \{{\ha}^\upsig \psi_\tau,\psi_\sigma\}^{-1}
        &&\textit{ by } \eqref{e:CM1}.
    \end{aligned}
    \end{equation}
    By taking the inverse and ${\ha}^{\psi_{\sigma \tau \nu}} (\quad)$ on the both sides, \eqref{eq:def2_1-cocycle} for $(u',\psi')$ is equivalent to
    \[
    {\ha}^{\psi_{\sigma \tau \nu}}({\ha\ha}^{\psi'_\sigma \sigma} {u'}_{\tau,\nu}^{-1} \cdot {u'}_{\sigma, \tau \nu}^{-1})
    =
    {\ha}^{\psi_{\sigma \tau \nu}}({u'}_{\sigma,\tau}^{-1} \cdot {u'}_{\sigma \tau,\nu}^{-1}),
    \]
    which is proved as follows:
    \begin{align*}
        &{\ha}^{\psi_{\sigma \tau \nu}}({\ha\ha}^{\psi'_\sigma \sigma} {u'}_{\tau,\nu}^{-1} \cdot {u'}_{\sigma, \tau \nu}^{-1})= \\
        &={\ha}^{\psi_{\sigma \tau \nu}}{\ha}^{\psi^{-1}_\sigma \sigma} {u'}_{\tau,\nu}^{-1} \cdot {\ha}^{\psi_{\sigma \tau \nu}}{u'}_{\sigma, \tau \nu}^{-1}
        &&\textit{ by } \eqref{eq:def_inv_2} \\
        &={\ha}^{\psi_{\sigma \tau \nu}}{\ha}^{\psi^{-1}_\sigma} {\ha}^{{\ha}^\upsig\psi^{-1}_{\tau\nu}}
        ({\ha}^\upsig {u}_{\tau,\nu} \cdot \{{\ha}^{\upst} \psi_\nu,{\ha}^\upsig \psi_\tau\}^{-1})
        \cdot
        u_{\sigma,\tau \nu} \cdot \{{\ha}^{\upsig} \psi_{\tau\nu},\psi_{\sigma}\}^{-1}
        &&\textit{ by } \eqref{eq:u'_inverse} \\
        &={\ha}^{\rho(u_{\sigma,\tau \nu})[\psi_{\sigma},{\ha}^{\upsig} \psi_{\tau\nu}]}
        ({\ha}^\upsig {u}_{\tau,\nu} \cdot \{{\ha}^{\upst} \psi_\nu,{\ha}^\upsig \psi_\tau\}^{-1})
        \cdot
        u_{\sigma,\tau \nu} \cdot \{{\ha}^{\upsig} \psi_{\tau\nu},\psi_{\sigma}\}^{-1}
        &&\textit{ by } \eqref{eq:def1_1-cocycle} \\
        &= u_{\sigma,\tau \nu}
        \cdot\{{\ha}^{\upsig} \psi_{\tau\nu},\psi_{\sigma}\}^{-1}
        \cdot {\ha}^{\rho(\{{\ha}^\upsig \psi_{\tau \nu},\psi_\sigma\} u_{\sigma,\tau \nu}^{-1})\rho(u_{\sigma,\tau \nu}\{\psi_\sigma,{\ha}^\upsig \psi_{\tau \nu}\})} ({\ha}^\upsig {u}_{\tau,\nu} \cdot \{{\ha}^{\upst} \psi_\nu,{\ha}^\upsig \psi_\tau\}^{-1})
        &&\textit{ by } \eqref{e:CM1} \\
        &= u_{\sigma,\tau \nu}
        \cdot
        \{{\ha}^\upsig \psi_{\tau \nu},\psi_\sigma\}^{-1}
        \cdot
        {\ha}^\upsig {u}_{\tau,\nu}
        \cdot \{{\ha}^{\upst} \psi_\nu,{\ha}^\upsig \psi_\tau\}^{-1}
        &&\textit{ by } \eqref{e:Br1} \\
        &= u_{\sigma,\tau \nu}
        \cdot
        \{\rho({\ha}^\upsig u_{\tau,\nu}) {\ha}^\upsig\psi_{\tau} {\ha}^{\upst} \psi_{\nu},\psi_\sigma\}^{-1}
        \cdot
        {\ha}^\upsig {u}_{\tau,\nu}
        \cdot \{{\ha}^{\upst} \psi_\nu,{\ha}^\upsig \psi_\tau\}^{-1}
        &&\textit{ by } \eqref{eq:def1_1-cocycle} \\
        &= u_{\sigma,\tau \nu}
        \cdot
        {\ha}^{\psi_\sigma \sigma} {u}_{\tau,\nu}
        \cdot
        \{{\ha}^\upsig\psi_{\tau} {\ha}^{\upst} \psi_{\nu},\psi_\sigma\}^{-1}
        \cdot
         \{{\ha}^{\upst} \psi_\nu,{\ha}^\upsig \psi_\tau\}^{-1}
        &&\textit{ by } \eqref{e:B.2} \\
        &= u_{\sigma,\tau \nu}
        \cdot
        {\ha}^{\psi_\sigma \sigma} {u}_{\tau,\nu}
        \cdot
        \{{\ha}^\upsig\psi_{\tau},\psi_\sigma\}^{-1}
        \cdot
        {\ha}^{{\ha}^\upsig\psi_{\tau}}\{{\ha}^{\upst} \psi_{\nu},\psi_\sigma\}^{-1}
        \cdot
         \{{\ha}^{\upst} \psi_\nu,{\ha}^\upsig \psi_\tau\}^{-1}
        &&\textit{ by } \eqref{e:Br5} \\
        &= u_{\sigma,\tau \nu}
        \cdot
        {\ha}^{\psi_\sigma \sigma} {u}_{\tau,\nu}
        \cdot
        \{{\ha}^\upsig\psi_{\tau},\psi_\sigma\}^{-1}
        \cdot
        \{{\ha}^{\upst} \psi_{\nu},{\ha}^\upsig\psi_{\tau} \psi_\sigma\}^{-1}
        &&\textit{ by } \eqref{e:Br4}
        \\
        &= u_{\sigma\tau, \nu}
        \cdot
        {u}_{\sigma,\tau}
        \cdot
        \{{\ha}^\upsig\psi_{\tau},\psi_\sigma\}^{-1}
        \cdot
        \{{\ha}^{\upst} \psi_{\nu},{\ha}^\upsig\psi_{\tau} \psi_\sigma\}^{-1}
        &&\textit{ by } \eqref{eq:def2_1-cocycle}
        \\
        &= u_{\sigma\tau, \nu}
        \cdot
        \{{\ha}^{\upst} \psi_{\nu},\rho({u}_{\sigma,\tau}
        \cdot
        \{{\ha}^\upsig\psi_{\tau},\psi_\sigma\}^{-1}){\ha}^\upsig\psi_{\tau} \psi_\sigma\}^{-1}
        \cdot
        {\ha}^{{\ha}^{\upst}\psi_\nu}({u}_{\sigma,\tau}
        \cdot
        \{{\ha}^\upsig\psi_{\tau},\psi_\sigma\}^{-1})
        &&\textit{ by } \eqref{e:B.1} \\
        &= u_{\sigma\tau, \nu}
        \cdot
        \{{\ha}^{\upst} \psi_{\nu},\rho({u}_{\sigma,\tau}) \psi_\sigma {\ha}^\upsig\psi_{\tau}\}^{-1}
        \cdot
        {\ha}^{{\ha}^{\upst}\psi_\nu}({u}_{\sigma,\tau}
        \cdot
        \{{\ha}^\upsig\psi_{\tau},\psi_\sigma\}^{-1})
        &&\textit{ by } \eqref{e:Br1} \\
        &= (u_{\sigma\tau, \nu}
        \cdot
        \{{\ha}^{\upst} \psi_{\nu},\psi_{\sigma \tau}\}^{-1})
        \cdot
        {\ha}^{{\ha}^{\upst}\psi_\nu}({u}_{\sigma,\tau}
        \cdot
        \{{\ha}^\upsig\psi_{\tau},\psi_\sigma\}^{-1})
        &&\textit{ by } \eqref{eq:def1_1-cocycle}
        \\
        &=
        {\ha}^{\rho(u_{\sigma\tau,\nu}\{{\ha}^{\upst}\psi_\nu,\psi_{\sigma \tau}\}^{-1}) {\ha}^{\upst} \psi_\nu}
        (u_{\sigma,\tau} \cdot \{{\ha}^\upsig \psi_\tau,\psi_\sigma\}^{-1})
        \cdot
        (u_{\sigma \tau,\nu} \cdot \{{\ha}^{\upst}\psi_\nu,\psi_{\sigma \tau}\}^{-1})
        &&\textit{ by } \eqref{e:CM1} \\
        &={\ha}^{\rho(u_{\sigma\tau,\nu})\psi_{\sigma \tau}{\ha}^{\upst}\psi_\nu \psi_{\sigma\tau}^{-1}}
        (u_{\sigma,\tau} \cdot \{{\ha}^\upsig \psi_\tau,\psi_\sigma\}^{-1})
        \cdot
        (u_{\sigma \tau,\nu} \cdot \{{\ha}^{\upst}\psi_\nu,\psi_{\sigma \tau}\}^{-1})
        &&\textit{ by } \eqref{e:Br1} \\
        &=
        {\ha}^{\psi_{\sigma \tau \nu} \psi_{\sigma\tau}^{-1}}
        (u_{\sigma,\tau} \cdot \{{\ha}^\upsig \psi_\tau,\psi_\sigma\}^{-1})
        \cdot
        (u_{\sigma \tau,\nu} \cdot \{{\ha}^{\upst}\psi_\nu,\psi_{\sigma \tau}\}^{-1})
        &&\textit{ by } \eqref{eq:def1_1-cocycle} \\
        &=
        {\ha}^{\psi_{\sigma \tau \nu}}({u'}_{\sigma,\tau}^{-1} \cdot {u'}_{\sigma \tau,\nu}^{-1})
        &&\textit{ by } \eqref{eq:u'_inverse}.
    \end{align*}
    This completes the proof of Lemma \ref{l:grp_str_Z1}.
\end{proof}

We wish to construct a differential:  a group homomorphism $d \colon C^0(\Gamma,A \to G) \to Z^1(\Gamma,A \to G)$.
For $(\varphi,g) \in C^0(\Gamma,A \to G)$, we define $d(\varphi,g)$
to be the 1-cocycle $(u,\psi) \in C^1(\Gamma,A \to G)$ with
\begin{align}
    u_{\sigma,\tau}
    &\coloneqq  {\ha}^{g^{-1}}(\varphi_{\sigma \tau}
    \cdot  {\ha}^\upsig \varphi_\tau^{-1}
    \cdot \varphi_\sigma^{-1}),\notag\\
    \tag{\sf Dif1}\label{e:HM1} \\
    \psi_\sigma
    &\coloneqq
    g^{-1} \cdot \rho(\varphi_\sigma) \cdot {\ha}^\upsig\hmm g
    \tag{\sf Dif2}\label{e:HM2}
\end{align}
for any $\sigma, \tau \in \Gamma$.

We prove certain equalities, which will be used to prove other lemmas.
\begin{lemma}\label{l:noname1}
    \begin{enumerate}
        \item
        For any $(\varphi^1,g_1)\in C^0(\Gamma,A \to G)$ and $g_2 \in G$, we have
        \begin{align}
        {\ha}^{g_2} \{g_1,\psi^2_\sigma\}
        =
        \{g_1,g_2\}^{-1} \cdot {\ha}^{g_1}  \varphi^2_\sigma \cdot \{g_1,{\ha}^\upsig\hmm g_2\} \cdot (\varphi^2_\sigma)^{-1}
        \label{eq:lem_tech-1},
        \end{align}
        \begin{align}
        {\ha}^{g_1} \{\psi^1_\sigma,g_2\}
        =
        \varphi^1_\sigma  \cdot \{{\ha}^\upsig\hmm g_1,g_2\} \cdot {\ha}^{g_2}  (\varphi^1_\sigma)^{-1} \cdot \{g_1,g_2\}^{-1}
        \label{eq:lem_tech-2}
        \end{align}
        where $\psi^i_\sigma = g_i^{-1}\rho(\varphi^i_\sigma) {\ha}^\upsig\hmm g_i$ for $i = 1,2$.

        \item
    For any $(u,\psi) \in Z^1(\Gamma,A \to G)$ and $(\varphi,g) \in C^0(\Gamma,A \to G)$, we have
    \[
    d((\varphi,g)) \cdot (u,\psi)
    =
    (u',\psi')
    \]
    where
    \begin{align}
        u'_{\sigma,\tau} &=
        {\ha}^{g^{-1}}(\varphi_{\sigma \tau}
        \cdot {\ha}^{{\ha}^{\upst} g}u_{\sigma,\tau}
        \cdot
        \{{\ha}^{\upst} g,\psi_\sigma\} \cdot {\ha}^{\psi_\sigma \sigma}\varphi_\tau^{-1} \cdot \{{\ha}^\upsig\hmm g,\psi_\sigma\}^{-1} \cdot \varphi_\sigma^{-1}) ,
        \label{eq:lem_tech-3} \\
        \psi'_\sigma &= g^{-1} \rho(\varphi_\sigma) {\ha}^\upsig\! g \cdot \psi_\sigma\hs. \label{eq:lem_tech-4}
    \end{align}
    \end{enumerate}
\end{lemma}
\begin{proof}
    We prove \eqref{eq:lem_tech-1} as follows:
    \begin{align*}
        {\ha}^{g_2} \{g_1,\psi^2_\sigma\}
        &= {\ha}^{g_2} \{g_1,g_2^{-1}\rho(\varphi^2_\sigma){\ha}^\upsig\hmm g_2\} \cdot \varphi^2_\sigma \\
        &= {\ha}^{g_2}\{g_1,g_2^{-1}\}
        \cdot
        \{g_1,\rho(\varphi^2_\sigma){\ha}^\upsig\hmm g_2\}
        &&\textit{ by } \eqref{e:Br4}\\
        &=
        \{g_1,g_2\}^{-1} \cdot {\ha}^{g_1}  \varphi^2_\sigma \cdot \{g_1,{\ha}^\upsig\hmm g_2\}
        \cdot (\varphi^2_\sigma)^{-1} &&\textit{ by } \eqref{e:Br4}, \eqref{e:B.4}.
    \end{align*}
    The proof of \eqref{eq:lem_tech-2} is similar.
    We denote $(u'',\psi'') = d((\varphi,g)) \cdot (u,\psi)$.
        Then \eqref{eq:lem_tech-4} holds since $\psi'_\sigma = (g^{-1} \varphi_\sigma {\ha}^\upsig\hmm g)
        \cdot
        \psi_\sigma = \psi''_\sigma$ for any $\sigma$ by \eqref{e:HM2}.
        We prove \eqref{eq:lem_tech-3} as follows:
    \begin{align*}
        &u''_{\sigma,\tau}= \\
        &= {\ha}^{g^{-1}}(\varphi_{\sigma \tau}
        \cdot {\ha}^\upsig \varphi_\tau^{-1}
        \cdot
        \varphi_\sigma^{-1})
        \cdot {\ha}^{g^{-1} \rho(\varphi_\sigma {\ha}^\upsig \varphi_\tau) {\ha}^{\upst} g}u_{\sigma,\tau}
        \cdot
        {\ha}^{g^{-1} \rho(\varphi_\sigma) {\ha}^\upsig\hmm g}
        \{{\ha}^\upsig\hmm g^{-1} \rho({\ha}^\upsig \varphi_\tau) {\ha}^{\upst} g,\psi_\sigma\}
        &&\textit{ by } \eqref{eq:def_prod_1}, \eqref{e:HM1}, \eqref{e:HM2} \\
        &= {\ha}^{g^{-1}}(\varphi_{\sigma \tau}
        \cdot {\ha}^{{\ha}^{\upst} g}u_{\sigma,\tau}
        \cdot {\ha}^\upsig \varphi_\tau^{-1}
        \cdot
        {\ha}^{{\ha}^\upsig\hmm g}
        \{{\ha}^\upsig\hmm g^{-1} \rho({\ha}^\upsig \varphi_\tau) {\ha}^{\upst} g,\psi_\sigma\}
        \cdot \varphi_\sigma^{-1})
        &&\textit{ by } \eqref{e:CM1} \\
        &= {\ha}^{g^{-1}}(\varphi_{\sigma \tau}
        \cdot {\ha}^{{\ha}^{\upst} g}u_{\sigma,\tau}
        \cdot {\ha}^\upsig \varphi_\tau^{-1}
        \cdot
        {\ha}^{{\ha}^\upsig}({\ha}^\upg
        \{g^{-1} \rho(\varphi_\tau) {\ha}^{\tau} g,{\ha}^{\sigma^{-1}}\psi_\sigma\})
        \cdot \varphi_\sigma^{-1})
        \\
        &= {\ha}^{g^{-1}}(\varphi_{\sigma \tau}
        \cdot {\ha}^{{\ha}^{\upst} g}u_{\sigma,\tau}
        \cdot
        \{{\ha}^{\upst} g,\psi_\sigma\} \cdot {\ha}^{\psi_\sigma \sigma}\varphi_\tau^{-1} \cdot \{{\ha}^\upsig\hmm g,\psi_\sigma\}^{-1} \cdot \varphi_\sigma^{-1})
        &&\textit{ by } \eqref{eq:lem_tech-2}.
    \end{align*}
    This completes the proof of Lemma \ref{l:noname1}.
\end{proof}

\begin{lemma}\label{l:noname2}
    \begin{enumerate}
        \item We have $d(\varphi,g) \in Z^1(\Gamma,A \to G)$ for any $(\varphi,g) \in C^0(\Gamma,A \to G)$.
        \item The map $d \colon C^0(\Gamma,A \to G) \to Z^1(\Gamma,A \to G)$ is a group homomorphism.
    \end{enumerate}
\end{lemma}

\begin{proof}
    To prove (1), it suffices to show that $(u,\psi) \coloneqq  d(\varphi,g)$ satisfies \eqref{eq:def1_1-cocycle} and \eqref{eq:def2_1-cocycle}.
    We prove \eqref{eq:def1_1-cocycle} as follows:
    \begin{align*}
        \rho(u_{\sigma,\tau}) \cdot \psi_\sigma \cdot  {\ha}^\upsig \psi_\tau
        &= \rho({\ha}^{g^{-1}}(\varphi_{\sigma \tau} \cdot {\ha}^\upsig \varphi^{-1}_\tau \cdot \varphi_\sigma^{-1}))
        \cdot
        (g^{-1} \cdot \rho(\varphi_\sigma) \cdot {\ha}^\upsig\hmm g)
        \cdot
        ({\ha}^\upsig\hmm g^{-1} \cdot \rho({\ha}^\upsig\varphi_\tau) \cdot {\ha}^{\upst} g)
        &&\textit{ by } \eqref{e:HM1}, \eqref{e:HM2} \\
        &= g^{-1} \cdot
        \rho(\varphi_{\sigma \tau} \cdot {\ha}^\upsig \varphi^{-1}_\tau \cdot \varphi_\sigma^{-1})
        \cdot
        \rho(\varphi_\sigma)
        \cdot
        \rho({\ha}^\upsig\varphi_\tau) \cdot {\ha}^{\upst} g
        &&\textit{ by } \eqref{e:CM2} \\
        &=
        g^{-1}
        \cdot
        \rho(\varphi_{\sigma \tau})
        \cdot g^{\upst}
        \\
        &= \psi_{\sigma \tau}
        &&\textit{ by } \eqref{e:HM1}.
        \end{align*}
        We prove \eqref{eq:def2_1-cocycle} as follows:
        \begin{align*}
        &u_{\sigma,\tau \nu} \cdot {\ha}^{\psi_\sigma \sigma} u_{\tau,\nu} =\\
        &= {\ha}^{g^{-1}}(\varphi_{\sigma \tau \nu} \cdot  {\ha}^{\upsig} \varphi_{\tau \nu}^{-1} \cdot \varphi_\sigma^{-1})
        \cdot
        {\ha}^{g^{-1}\rho(\varphi_\sigma)}({\ha}^\upsig\varphi_{\tau \nu} \cdot  {\ha}^{\upst} \varphi_\nu^{-1} \cdot {\ha}^\upsig \varphi_\tau^{-1})
        &&\textit{ by } \eqref{e:HM1}, \eqref{e:HM2} \\
        &={\ha}^{g^{-1}}(\varphi_{\sigma \tau \nu} \cdot  {\ha}^{\upsig} \varphi_{\tau\nu}^{-1}
        \cdot
        \varphi_\sigma^{-1}
        \cdot
        \varphi_\sigma
        \cdot
        {\ha}^\upsig \varphi_{\tau \nu} \cdot  {\ha}^{\upst} \varphi_\nu^{-1} \cdot {\ha}^\upsig\varphi_\tau^{-1}
        \cdot \varphi_\sigma^{-1})
        &&\textit{ by } \eqref{e:CM1} \\
        &={\ha}^{g^{-1}}(\varphi_{\sigma \tau \nu}
        \cdot  {\ha}^{\upst} \varphi_\nu^{-1} \cdot {\ha}^\upsig\varphi_\tau^{-1}
        \cdot \varphi_\sigma^{-1})
        \\
        &= {\ha}^{g^{-1}}(\varphi_{\sigma \tau \nu} \cdot
        {\ha}^{\upst}
        \varphi_\nu^{-1}
        \cdot
        \varphi_{\sigma \tau}^{-1})
        \cdot
        {\ha}^{g^{-1}}
        (\varphi_{\sigma \tau}
        \cdot
        {\ha}^{\upsig}\varphi_\tau^{-1}
        \cdot \varphi_\sigma^{-1})
        \\
        &= u_{\sigma \tau,\nu} \cdot u_{\sigma,\tau}
        &&\textit{ by } \eqref{e:HM1}.
    \end{align*}
    For $(\varphi^1,g_1), (\varphi^2,g_2) \in C^0(\Gamma,A \to G)$, we denote
    $(\varphi^3,g_3) = (\varphi^1,g_1) \cdot (\varphi^2,g_2)$, $(u^3,\psi^3) \coloneqq  d((\varphi^3,g_3))$, $(u^1,\psi^1) \coloneqq  d((\varphi^1,g_1))$, $(u^2,\psi^2) \coloneqq  d((\varphi^1,g_1))$, and $(u^{1,2},\psi^{1,2}) \coloneqq  (u^1,\psi^1) \cdot (u^2,\psi^2)$.
    To prove (2), it suffices to show that $(u^3,\psi^3) = (u^{1,2},\psi^{1,2})$, which is proved as follows:
        \begin{align*}
        \psi^{1,2}_\sigma
        &= \psi^1_\sigma \cdot \psi^2_\sigma
        \\
        &= (g_1^{-1} \rho(\varphi^1_\sigma){\ha}^\upsig\hmm g_1) \cdot (g_2^{-1}\rho(\varphi^2_\sigma) {\ha}^\upsig\hmm g_2)
        &&\textit{ by } \eqref{e:HM2} \\
        &= g_2^{-1}[g_1^{-1}\rho(\varphi^1_\sigma){\ha}^{\upsig} g_1,g_2]^{-1} g_1^{-1}\rho(\varphi^1_\sigma){\ha}^\upsig\hmm g_1
        \rho( \varphi^2_\sigma)
        {\ha}^\upsig\hmm g_2
        \\
        &= (g_1 g_2)^{-1}\rho({\ha}^{g_1}\{g_1^{-1}\rho(\varphi^1_\sigma){\ha}^{\upsig} g_1,g_2 \}^{-1} \varphi^1_\sigma {\ha}^{{\ha}^\upsig\hmm g_1}\varphi^2_\sigma) {\ha}^\upsig (g_1g_2)
        &&\textit{ by } \eqref{e:CM2}, \eqref{e:Br1} \\
        &= \psi^3_\sigma
        &&\textit{ by }\eqref{eq:product_C0}, \eqref{e:HM2},
        \end{align*}
        \begin{align*}
        &{\ha}^{g_3} u^{1,2}_{\sigma,\tau}= \\
        &= {\ha}^{g_1 g_2 g_1^{-1}}(\varphi^1_{\sigma \tau}
        \cdot {\ha}^{{\ha}^{\upst} g_1}u^2_{\sigma,\tau}
        \cdot
        \{{\ha}^{\upst} g_1,\psi^2_\sigma\}
        \cdot
        {\ha}^{\psi^2_\sigma \sigma}(\varphi_\tau^1)^{-1} \cdot \{{\ha}^\upsig\hmm g_1,\psi^2_\sigma\}^{-1} \cdot (\varphi^1_\sigma)^{-1})
        &&\textit{ by } \eqref{eq:lem_tech-3} \\
        &= {\ha}^{[g_1,g_2]g_2}(\varphi^1_{\sigma \tau}
        \cdot {\ha}^{{\ha}^{\upst} g_1}u^2_{\sigma,\tau}
        \cdot
        \{{\ha}^{\upst} g_1,\psi^2_\sigma\} \cdot {\ha}^{\psi^2_\sigma \sigma}(\varphi_\tau^1)^{-1} \cdot \{{\ha}^\upsig\hmm g_1,\psi^2_\sigma\}^{-1} \cdot (\varphi^1_\sigma)^{-1}) \\
        &= {\ha}^{[g_1,g_2]g_2}(\varphi^1_{\sigma \tau}
        \cdot {\ha}^{{\ha}^{\upst} g_1}u^2_{\sigma,\tau}
        \cdot
        \{{\ha}^{\upst} g_1,\psi^2_\sigma\} \cdot {\ha}^{g_2^{-1}\rho(\varphi^2_\sigma){\ha}^\upsig\hmm g_2 \sigma}(\varphi_\tau^1)^{-1} \cdot \{{\ha}^\upsig\hmm g_1,\psi^2_\sigma\}^{-1} \cdot (\varphi^1_\sigma)^{-1})
        &&\textit{ by } \eqref{eq:action_C0-2}\\
        &= {\ha}^{[g_1,g_2]}({\ha}^{g_2}(\varphi^1_{\sigma \tau}
        \cdot {\ha}^{{\ha}^{\upst} g_1}u^2_{\sigma,\tau}
        \cdot
        \{{\ha}^{\upst} g_1,\psi^2_\sigma\})
        \cdot {\ha}^{\rho(\varphi^2_\sigma){\ha}^\upsig\hmm g_2 \sigma}(\varphi_\tau^1)^{-1}
        \cdot {\ha}^{g_2}\{{\ha}^\upsig\hmm g_1,\psi^2_\sigma\}^{-1} \cdot ({\ha}^{g_2}\varphi^1_\sigma)^{-1}) \\
        &= {\ha}^{[g_1,g_2]}({\ha}^{g_2}(\varphi^1_{\sigma \tau}
        \cdot {\ha}^{{\ha}^{\upst} g_1}u^2_{\sigma,\tau}
        \cdot
        \{{\ha}^{\upst} g_1,\psi^2_\sigma\})
        \cdot \varphi^2_\sigma
        \cdot
        {\ha}^{{\ha}^\upsig\hmm g_2 \sigma}(\varphi_\tau^1)^{-1}
        (\varphi^2_\sigma)^{-1}
        \cdot {\ha}^{g_2}\{{\ha}^\upsig\hmm g_1,\psi^2_\sigma\}^{-1} \cdot ({\ha}^{g_2}\varphi^1_\sigma)^{-1})
        &&\textit{ by } \eqref{e:CM1}\\
        &= {\ha}^{[g_1,g_2]}({\ha}^{g_2}(\varphi^1_{\sigma \tau}
        \cdot {\ha}^{{\ha}^{\upst} g_1}u^2_{\sigma,\tau}
        \cdot
        \{{\ha}^{\upst} g_1,\psi^2_\sigma\})
        \cdot \varphi^2_\sigma
        \cdot
        {\ha}^{{\ha}^\upsig\hmm g_2 \sigma}(\varphi_\tau^1)^{-1}
        &&\textit{ by } \eqref{eq:lem_tech-1} \\
        &\quad \cdot (\{{\ha}^\upsig\hmm g_1,{\ha}^\upsig\hmm g_2\}^{-1}
        \cdot
        {\ha}^{{\ha}^\upsig\hmm g_1} (\varphi^2_\sigma)^{-1}
        \cdot
        \{{\ha}^\upsig\hmm g_1,g_2\})
        \cdot ({\ha}^{g_2}\varphi^1_\sigma)^{-1})) \\
        &= \{g_1,g_2\} \cdot ({\ha}^{g_2}(\varphi^1_{\sigma \tau}
        \cdot {\ha}^{{\ha}^{\upst} g_1}u^2_{\sigma,\tau}
        \cdot
        \{{\ha}^{\upst} g_1,\psi^2_\sigma\})
        \cdot \varphi^2_\sigma
        \cdot
        {\ha}^{{\ha}^\upsig\hmm g_2 \sigma}(\varphi_\tau^1)^{-1}
        \cdot
        \{{\ha}^\upsig\hmm g_1,{\ha}^\upsig\hmm g_2\}^{-1}
        &&\textit{ by } \eqref{e:Br1}, \eqref{e:CM1} \\
        &\quad \cdot(
        {\ha}^{{\ha}^\upsig\hmm g_1} (\varphi^2_\sigma)^{-1}
        \cdot
        \{{\ha}^\upsig\hmm g_1,g_2\})
        \cdot ({\ha}^{g_2}\varphi^1_\sigma)^{-1})) \cdot \{g_1,g_2\}^{-1}
        \\
        &= \{g_1,g_2\} \cdot ({\ha}^{g_2}(\varphi^1_{\sigma \tau}
        \cdot {\ha}^{{\ha}^{\upst} g_1}u^2_{\sigma,\tau}
        \cdot
        \{{\ha}^{\upst} g_1,\psi^2_\sigma\})
        \cdot \varphi^2_\sigma
        \cdot
        {\ha}^{{\ha}^\upsig\hmm g_2 \sigma}(\varphi_\tau^1)^{-1}
        \cdot
        \{{\ha}^\upsig\hmm g_1,{\ha}^\upsig\hmm g_2\}^{-1}
        &&\textit{ by } \eqref{eq:lem_tech-2} \\
        &\quad
        \cdot
        {\ha}^{{\ha}^\upsig\hmm g_1} (\varphi^2_\sigma)^{-1}
        \cdot
        (\varphi^1_\sigma)^{-1}
        \cdot
        {\ha}^{g_1}\{\psi^1_\sigma,g_2\}
        \\
        &= \{g_1,g_2\}
        \cdot {\ha}^{g_2}(\varphi^1_{\sigma \tau}
        \cdot {\ha}^{{\ha}^{\upst} g_1}u^2_{\sigma,\tau}
        \cdot
        \{{\ha}^{\upst} g_1,\psi^2_\sigma\})
        \cdot \varphi^2_\sigma
        \cdot
        {\ha}^{{\ha}^\upsig\hmm g_2 \sigma}(\varphi_\tau^1)^{-1}
        \cdot
        \{{\ha}^\upsig\hmm g_1,{\ha}^\upsig\hmm g_2\}^{-1}
        \cdot
        (\varphi^3_\sigma)^{-1}
        &&\textit{ by } \eqref{eq:product_C0} \\
        &= \{g_1,g_2\} \cdot {\ha}^{g_2}(\varphi^1_{\sigma \tau}
        \cdot {\ha}^{{\ha}^{\upst} g_1}u^2_{\sigma,\tau})
        \cdot
        \{{\ha}^{\upst} g_1,g_2\}^{-1}
        \cdot
        {\ha}^{{\ha}^{\upst} g_1} \varphi^2_\sigma
        \cdot \{{\ha}^{\upst} g_1,{\ha}^\upsig\hmm g_2\}
        &&\textit{ by } \eqref{eq:lem_tech-1}\\
        &\quad
        \cdot
        {\ha}^{{\ha}^\upsig\hmm g_2 \sigma}(\varphi_\tau^1)^{-1}
        \cdot
        \{{\ha}^\upsig\hmm g_1,{\ha}^\upsig\hmm g_2\}^{-1}
        \cdot
        (\varphi^3_\sigma)^{-1} \\
        &= \{g_1,g_2\} \cdot {\ha}^{g_2} \varphi^1_{\sigma \tau}
        \cdot {\ha}^{g_2{\ha}^{\upst} g_1 g_2^{-1}} (\varphi^2_{\sigma \tau}
        \cdot {\ha}^\upsig (\varphi^2_\tau)^{-1}
        \cdot
        (\varphi^2_\sigma)^{-1})
        \cdot
        \{{\ha}^{\upst} g_1,g_2\}^{-1}
        &&\textit{ by } \eqref{e:HM1}\\
        &\quad \cdot {\ha}^{{\ha}^{\upst} g_1} \varphi^2_\sigma
        \cdot \{{\ha}^{\upst} g_1,{\ha}^\upsig\hmm g_2\}
        \cdot
        {\ha}^{{\ha}^\upsig\hmm g_2 \sigma}(\varphi_\tau^1)^{-1}
        \cdot
        \{{\ha}^\upsig\hmm g_1,{\ha}^\upsig\hmm g_2\}^{-1}
        \cdot
        (\varphi^3_\sigma)^{-1} \\
        &= \{g_1,g_2\} \cdot {\ha}^{g_2} \varphi^1_{\sigma \tau}
        \cdot {\ha}^{\rho(\{{\ha}^{\upst} g_1,g_2 \}^{-1}){\ha}^{\upst}g_1} (\varphi^2_{\sigma \tau}
        \cdot {\ha}^\upsig (\varphi^2_\tau)^{-1}
        \cdot
        (\varphi^2_\sigma)^{-1})
        \cdot
        \{{\ha}^{\upst} g_1,g_2\}
        &&\textit{ by } \eqref{e:Br1}\\
        &\quad \cdot {\ha}^{{\ha}^{\upst} g_1} \varphi^2_\sigma
        \cdot \{{\ha}^{\upst} g_1,{\ha}^\upsig\hmm g_2\}
        \cdot
        {\ha}^{{\ha}^\upsig\hmm g_2 \sigma}(\varphi_\tau^1)^{-1}
        \cdot
        \{{\ha}^\upsig\hmm g_1,{\ha}^\upsig\hmm g_2\}^{-1}
        \cdot
        (\varphi^3_\sigma)^{-1}
        \\
        &= \{g_1,g_2\} \cdot {\ha}^{g_2} \varphi^1_{\sigma \tau}
        \cdot
        \{{\ha}^{\upst} g_1,g_2\}^{-1}
        \cdot
        {\ha}^{{\ha}^{\upst}g_1} (\varphi^2_{\sigma \tau}
        \cdot {\ha}^\upsig (\varphi^2_\tau)^{-1}
        \cdot
        (\varphi^2_\sigma)^{-1})
        &&\textit{ by } \eqref{e:CM1}\\
        &\quad \cdot {\ha}^{{\ha}^{\upst} g_1} \varphi^2_\sigma
        \cdot \{{\ha}^{\upst} g_1,{\ha}^\upsig\hmm g_2\}
        \cdot
        {\ha}^{{\ha}^\upsig\hmm g_2 \sigma}(\varphi_\tau^1)^{-1}
        \cdot
        \{{\ha}^\upsig\hmm g_1,{\ha}^\upsig\hmm g_2\}^{-1}
        \cdot
        (\varphi^3_\sigma)^{-1} \\
         &= {\ha}^{g_1}\{g_1^{-1}\rho(\varphi^1_{\sigma \tau}) {\ha}^{\upst} g_1,g_2\}^{-1}
        \cdot
        \varphi^1_{\sigma \tau}
        \cdot
        {\ha}^{{\ha}^{\upst}g_1} (\varphi^2_{\sigma \tau}
        \cdot ({\ha}^\upsig \varphi^2_\tau)^{-1}
        )
        &&\textit{ by } \eqref{eq:lem_tech-2}\\
        &\quad \cdot  \{{\ha}^{\upst} g_1,{\ha}^\upsig\hmm g_2\}
        \cdot
        {\ha}^{{\ha}^\upsig\hmm g_2 \sigma}(\varphi_\tau^1)^{-1}
        \cdot
        \{{\ha}^\upsig\hmm g_1,{\ha}^\upsig\hmm g_2\}^{-1}
        \cdot
        (\varphi^3_\sigma)^{-1} \\
        &= \varphi^3_{\sigma \tau} \cdot
        {\ha}^{{\ha}^{\upst}g_1} ({\ha}^\upsig \varphi^2_\tau)^{-1}
        \cdot \{{\ha}^{\upst} g_1,{\ha}^\upsig\hmm g_2\}
        \cdot
        {\ha}^{{\ha}^\upsig\hmm g_2 }({\ha}^\upsig\varphi_\tau^1)^{-1}
        \cdot
        \{{\ha}^\upsig\hmm g_1,{\ha}^\upsig\hmm g_2\}^{-1}
        \cdot
        (\varphi^3_\sigma)^{-1}
        &&\textit{ by } \eqref{eq:product_C0} \\
        &= \varphi^3_{\sigma \tau} \cdot
        {\ha}^\upsig(\varphi^3_\tau)^{-1}
        \cdot
        (\varphi^3_\sigma)^{-1}
        &&\textit{ by } \eqref{eq:product_C0}, \eqref{eq:lem_tech-2} \\
        &={\ha}^{g_3} u^3_{\sigma,\tau}
        &&\textit{ by } \eqref{e:HM1}.
    \end{align*}
    This completes the proof of Lemma \ref{l:noname2}.
\end{proof}

Consider the set of $0$-coboundaries $B^0(\Gamma,A \to G)$, which consists of $(\varphi^s,\rho(s)^{-1}) \in C^0(\Gamma,A \to G)$ where $\varphi^s_\sigma = s^{-1} {\ha}^\upsig s$ for $\sigma \in \Gamma$.
It is easy to check that $B^0(\Gamma,A \to G)$ is a normal subgroup of $C^0(\Gamma,A \to G)$.
Since $d$ vanishes on $B^0(\Gamma,A \to G)$, we obtain an induced homomorphism
\begin{equation}\label{e:bar-d}
\bar{d} \colon C^0(\Gamma,A \to G)/B^0(\Gamma,A \to G)\to Z^1(\Gamma,A \to G).
\end{equation}

\begin{remark}
    \begin{enumerate}
    \item Under the group structure on $C^0(\Gamma,A \to G)$ in Remark \ref{rmk:grp_C0_ordinary}, $d$ is not a homomorphism.

    \item
    It is easy to show that the kernel of $d$ is $Z^0(\Gamma,A \to G)$.
    In particular, the kernel of $\bar{d} \colon C^0(\Gamma,A \to G)/B^0(\Gamma,A \to G) \to Z^1(\Gamma,A \to G)$ becomes
    $H^0(\Gamma,A \to G) \coloneqq  Z^0(\Gamma,A \to G)/B^0(\Gamma,A \to G)$.
    This yields a canonical group structure on $H^0(\Gamma,A \to G)$.
    Furthermore, $H^0(\Gamma,A \to G)$ is abelian in the presence of the braiding $\br$; see \cite[Section 3.1]{Noohi}.
    \end{enumerate}
\end{remark}

We define a left action of $Z^1(\Gamma,A \to G)$ on $C^0(\Gamma,A \to G)$ by
\begin{align*}
    {\ha}^{(u,\psi)}(\varphi,g)
    = (\varphi',g)
\end{align*}
where
\begin{align}
    \varphi'_\sigma
       = \{\psi_\sigma,g\}^{-1} \cdot {\ha}^{\psi_\sigma} \varphi_\sigma \cdot \{{\ha}^\upsig\hmm g,\psi_\sigma\}^{-1}
    \label{eq:act_Z1}
\end{align}
for $\sigma \in \Gamma$.
Since this action preserves $B^0(\Gamma,A \to G)$,
we obtain an induced action of $Z^1(\Gamma,A \to G)$ on $C^0(\Gamma,A \to G)/B^0(\Gamma,A \to G)$.

\begin{lemma}\label{l:crossed-module}
    The homomorphism $\bar{d}$ of \eqref{e:bar-d} together with the above action
    of $Z^1(\Gamma,A \to G)$ on $C^0(\Gamma,A \to G)/B^0(\Gamma,A \to G)$
    is a (left) crossed module.
\end{lemma}

\begin{proof}
    We need to show that this action satisfies  \eqref{e:CM1} and \eqref{e:CM2}.
    For $(\varphi^1,g_1), (\varphi^2,g_2) \in C^0(\Gamma,A \to G)$, we denote by
    $(u^1,\psi^1) = d((\varphi^1,g_1))$,
    $(\varphi^{\prime\hs 2},g_2) = {\ha}^{(u^1,\psi^1)}(\varphi^2,g_2)$,
    $(\varphi^{\prime\hs 2,1},g_1g_2) = (\varphi^{\prime\hs 2},g_2) \cdot (\varphi^1,g_1)$ and $(\varphi^{1,2},g_1g_2) = (\varphi^1,g_1) \cdot (\varphi^2,g_2)$.
    To prove \eqref{e:CM1}, we need to show that
    \[
    [\varphi^{1,2},g_1g_2] = [\varphi^{\prime\hs 2,1},g_2g_1].
    \]
    It suffices to show that
    \[
    (\varphi^{\{g_1,g_2\}},\rho(\{g_1,g_2\}^{-1})) \cdot (\varphi^{1,2},g_1g_2)
    =
    (\varphi^{\prime\hs 2,1},g_2 g_1).
    \]
    where $(\varphi^{\{g_1,g_2\}},\rho(\{g_1,g_2\}^{-1})) \in B^0(\Gamma,A \to G)$, $\varphi^{\{g_1,g_2\}} \colon \sigma \mapsto \{g_1,g_2\}^{-1} \cdot \{{\ha}^\upsig\hmm g_1,{\ha}^\upsig\hmm g_2\}$.
    We denote the left-hand side by $(\varphi,g)$.
    The desired equality is proved as follows:
    \begin{align*}
        g
        = \rho(\{g_1,g_2\}^{-1}) \cdot g_1g_2 = [g_1,g_2] g_1 g_2 = g_2 g_1,
    \end{align*}
    \begin{align*}
        \varphi_\sigma
        &=
        {\ha}^{\rho(\{g_1,g_2\}^{-1})}
        \{g_1 g_2,1\}
        \cdot
        \varphi^{\{g_1,g_2\}}_\sigma
        \cdot {\ha}^{\rho({\ha}^\upsig\{g_1,g_2\}^{-1})}
        (
        {\ha}^{g_1}\{\psi^1_\sigma,g_2\}^{-1} \varphi^1_\sigma {\ha}^{{\ha}^\upsig\hmm g_1} \varphi^2_\sigma
        )
        &&\textit{ by } \eqref{eq:product_C0}\\
        &= \{g_1,g_2\}^{-1}
        \cdot \{{\ha}^\upsig\hmm g_1, {\ha}^\upsig\hmm g_2\}
        \cdot {\ha}^{\rho({\ha}^\upsig\{g_1,g_2\}^{-1})}
        (
        {\ha}^{g_1}\{\psi^1_\sigma,g_2\}^{-1} \varphi^1_\sigma {\ha}^{{\ha}^\upsig\hmm g_1} \varphi^2_\sigma
        )
        \\
        &= \{g_1,g_2\}^{-1}
        \cdot
        (
        {\ha}^{g_1}\{\psi^1_\sigma,g_2\}^{-1} \varphi^1_\sigma {\ha}^{{\ha}^\upsig\hmm g_1} \varphi^2_\sigma
        )
        \cdot \{{\ha}^\upsig\hmm g_1,{\ha}^{\upsig} g_2\}
        &&\textit{ by } \eqref{e:CM1} \\
        &= \{g_1,g_2\}^{-1}
        \cdot
        {\ha}^{g_1}(\{\psi^1_\sigma,g_2\}^{-1} {\ha}^{g_1^{-1}\rho(\varphi^1_\sigma){\ha}^\upsig\hmm g_1} \varphi^2_\sigma)
        \cdot \varphi^1_\sigma
        \cdot \{{\ha}^\upsig\hmm g_1,{\ha}^{\upsig} g_2\}
        &&\textit{ by } \eqref{e:CM1} \\
        &= \{g_2,g_1\}
        \cdot
        {\ha}^{g_1}(\{\psi^1_\sigma,g_2\}^{-1} {\ha}^{\psi^1_\sigma} \varphi^2_\sigma)
        \cdot \varphi^1_\sigma
        \cdot \{{\ha}^{\upsig} g_2,{\ha}^\upsig\hmm g_1\}^{-1}
        &&\textit{ by } \eqref{e:HM2} \\
        &= \{g_2,g_1\}
        \cdot
        {\ha}^{g_1}
        \varphi^{\prime\hs 2}_\sigma
        \cdot
        {\ha}^{g_1} \{{\ha}^\upsig\hmm g_2,\psi^1_\sigma\}
        \cdot \varphi^1_\sigma
        \cdot \{{\ha}^{\upsig} g_2,{\ha}^\upsig\hmm g_1\}^{-1}
        &&\textit{ by } \eqref{eq:act_Z1} \\
        &= \{g_2,g_1\}
        \cdot {\ha}^{g_1} \varphi^{\prime\hs 2}_\sigma
        \cdot \{{\ha}^\upsig\hmm g_2,g_1\}^{-1}
        \cdot {\ha}^{{\ha}^\upsig\hmm g_2} \varphi^1_\sigma
        &&\textit{ by } \eqref{eq:lem_tech-2} \\
        &= {\ha}^{g_2}\{g_2^{-1}\rho(\varphi^{\prime\hs 2}_\sigma){\ha}^\upsig\hmm g_2,g_1\}^{-1} \cdot \varphi^{\prime\hs 2}_\sigma \cdot {\ha}^{{\ha}^\upsig\hmm g_2} \varphi^1_\sigma
        &&\textit{ by } \eqref{eq:lem_tech-1} \\
        &=\varphi^{\prime\hs 2,1}_\sigma
        &&\textit{ by } \eqref{eq:product_C0}.
    \end{align*}
    We denote by $(\varphi',g) = {\ha}^{(u,\psi)}(\varphi,g)$, $(u',\psi') = d(\varphi',g) \cdot (u,\psi)$,
    and $(u'',\psi'') = (u,\psi) \cdot d((\varphi,g))$.
    We prove \eqref{e:CM2} by showing that $(u',\psi') = (u'',\psi'')$ as follows:
    \begin{align*}
        \psi'_\sigma
        &=g^{-1}\cdot \rho(\varphi'_\sigma)\cdot {\ha}^\upsig\hmm g \cdot \psi_\sigma
        &&\textit{ by } \eqref{e:HM2} \\
        &=g^{-1}\cdot \rho(\{\psi_\sigma,g\}^{-1}{\ha}^{\psi_\sigma}\varphi_\sigma\{{\ha}^\upsig\hmm g,\psi_\sigma\}^{-1})\cdot {\ha}^\upsig\hmm g \cdot \psi_\sigma
        &&\textit{ by } \eqref{eq:act_Z1} \\
        &= \psi_\sigma \cdot (g^{-1} \cdot \rho(\varphi_\sigma)
        \cdot {\ha}^\upsig\hmm g)
        &&\textit{ by } \eqref{e:Br1}, \eqref{e:CM1}\\
        &= \psi''_\sigma
        &&\textit{ by } \eqref{e:HM2},
    \end{align*}
    \begin{align*}
        &{\ha}^\upg u'_{\sigma,\tau} = \\
        &= \varphi'_{\sigma \tau}
        \cdot {\ha}^{{\ha}^{\upst} g}u_{\sigma,\tau}
        \cdot
        \{{\ha}^{\upst} g,\psi_{\sigma\tau}\} \cdot {\ha}^{\psi_\sigma \sigma} \varphi^{\prime\,-1}_\tau
        \cdot
        \{{\ha}^\upsig\hmm g,\psi_\sigma\}^{-1} \cdot \varphi^{\prime\,-1}_\sigma
        &&\textit{ by } \eqref{eq:lem_tech-3}\\
        &= \{\psi_{\sigma \tau},g\}^{-1}
        \cdot {\ha}^{\psi_{\sigma \tau}}\varphi_{\sigma \tau}
        \cdot
        \{{\ha}^{\upst} g,\psi_{\sigma \tau}\}^{-1}
        \cdot {\ha}^{{\ha}^{\upst} g}u_{\sigma,\tau}
        \cdot
        \{{\ha}^{\upst} g,\psi_\sigma\} \cdot {\ha}^{\psi_\sigma \sigma} \varphi^{\prime\,-1}_\tau
        \cdot
        \{{\ha}^\upsig\hmm g,\psi_\sigma\}^{-1} \cdot \varphi^{\prime\,-1}_\sigma
        &&\textit{ by } \eqref{eq:act_Z1} \\
        &=
        \{\rho(u_{\sigma,\tau})\psi_\sigma{\ha}^{\upst} \psi_\tau,g\}^{-1}
        \cdot
        {\ha}^{\psi_{\sigma\tau}}\varphi_{\sigma \tau}
        \cdot
        \{{\ha}^{\upst} g,\rho(u_{\sigma,\tau})\psi_\sigma {\ha}^\upsig \psi_\tau\}^{-1}
        \cdot {\ha}^{{\ha}^{\upst} g}u_{\sigma,\tau}
        \\
        &\quad
        \cdot
        \{{\ha}^{\upst} g,\psi_\sigma\} \cdot {\ha}^{\psi_\sigma \sigma} \varphi^{\prime\,-1}_\tau
        \cdot
        \{{\ha}^\upsig\hmm g,\psi_\sigma\}^{-1} \cdot \varphi^{\prime\,-1}_\sigma
        \\
        &= {\ha}^\upg u_{\sigma,\tau}
        \cdot
        \{\psi_\sigma{\ha}^\upsig \psi_\tau,g\}^{-1}
        \cdot
        u_{\sigma,\tau}^{-1}
        \cdot
        {\ha}^{\psi_{\sigma\tau}}\varphi_{\sigma \tau}
        \cdot
        u_{\sigma,\tau}
        \cdot
        \{{\ha}^{\upst} g,\psi_\sigma {\ha}^\upsig \psi_\tau\}^{-1}
        &&\textit{ by } \eqref{e:Br5} \\
        &\quad
        \cdot
        \{{\ha}^{\upst} g,\psi_\sigma\} \cdot {\ha}^{\psi_\sigma \sigma} \varphi^{\prime\,-1}_\tau
        \cdot
        \{{\ha}^\upsig\hmm g,\psi_\sigma\}^{-1} \cdot \varphi^{\prime\,-1}_\sigma \\
        &= {\ha}^\upg u_{\sigma,\tau}
        \cdot
        \{\psi_\sigma{\ha}^\upsig \psi_\tau,g\}^{-1}
        \cdot
        u_{\sigma,\tau}^{-1}
        \cdot
        {\ha}^{\psi_{\sigma\tau}}\varphi_{\sigma \tau}
        \cdot
        u_{\sigma,\tau}
        \cdot
        {\ha}^{\psi_\sigma}\{{\ha}^{\upst} g,{\ha}^\upsig \psi_\tau\}^{-1}
        \cdot
        {\ha}^{\psi_\sigma \sigma} \varphi^{\prime\,-1}_\tau
        \cdot
        \{{\ha}^\upsig\hmm g,\psi_\sigma\}^{-1} \cdot \varphi^{\prime\,-1}_\sigma
        &&\textit{ by }\eqref{e:Br4}\\
        &= {\ha}^\upg u_{\sigma,\tau} \cdot
        {\ha}^{\rho(\{\psi_\sigma{\ha}^\upsig \psi_\tau,g\})^{-1}\rho(u_{\sigma,\tau}^{-1})\psi_{\sigma\tau}}\varphi_{\sigma \tau}
        \cdot
        \{\psi_\sigma{\ha}^\upsig \psi_\tau,g\}^{-1}
        \cdot
        {\ha}^{\psi_\sigma}\{{\ha}^{\upst} g,{\ha}^\upsig \psi_\tau\}^{-1}
        \cdot
        {\ha}^{\psi_\sigma \sigma} \varphi^{\prime\,-1}_\tau
        \cdot
        \{{\ha}^\upsig\hmm g,\psi_\sigma\}^{-1} \cdot \varphi^{\prime\,-1}_\sigma
        &&\textit{ by }\eqref{e:CM1} \\
        &= {\ha}^\upg u_{\sigma,\tau} \cdot
        {\ha}^{g \psi_\sigma {\ha}^\upsig \psi_\tau g^{-1}}\varphi_{\sigma \tau}
        \cdot
        \{\psi_\sigma{\ha}^\upsig \psi_\tau,g\}^{-1}
        \cdot
        {\ha}^{\psi_\sigma}\{{\ha}^{\upst} g,{\ha}^\upsig \psi_\tau\}^{-1}
        \cdot
        {\ha}^{\psi_\sigma \sigma} \varphi^{\prime\,-1}_\tau
        \cdot
        \{{\ha}^\upsig\hmm g,\psi_\sigma\}^{-1} \cdot \varphi^{\prime\,-1}_\sigma
        &&\textit{ by }\eqref{eq:def1_1-cocycle} \\
        &= {\ha}^\upg u_{\sigma,\tau} \cdot
        {\ha}^{g \psi_\sigma {\ha}^\upsig \psi_\tau g^{-1}}\varphi_{\sigma \tau}
        \cdot
        \{\psi_\sigma{\ha}^\upsig \psi_\tau,g\}^{-1}
        \cdot
        {\ha}^{\psi_\sigma}\{{\ha}^{\upst} g,{\ha}^\upsig \psi_\tau\}^{-1}
        &&\textit{ by }\eqref{eq:act_Z1} \\
        &\quad
        \cdot
        {\ha}^{\psi_\sigma} (\{{\ha}^{\upst} g,{\ha}^\upsig \psi_\tau\}
         \cdot {\ha}^{{\ha}^\upsig \psi_\tau} {\ha}^\upsig \varphi_\tau^{-1}
         \cdot \{{\ha}^\upsig \psi_\tau,{\ha}^\upsig\hmm g\} )
        \cdot
        \{{\ha}^\upsig\hmm g,\psi_\sigma\}^{-1} \cdot
        (\{{\ha}^\upsig\hmm g,\psi_\sigma\}
        {\ha}^{\psi_\sigma}\varphi_\sigma^{-1}
        \{\psi_\sigma,g\}) \\
       &= {\ha}^\upg u_{\sigma,\tau} \cdot
        {\ha}^{g \psi_\sigma {\ha}^\upsig \psi_\tau g^{-1}}\varphi_{\sigma \tau}
        \cdot
        \{\psi_\sigma{\ha}^\upsig \psi_\tau,g\}^{-1}
        \cdot
        {\ha}^{\psi_\sigma} (
         {\ha}^{{\ha}^\upsig \psi_\tau}{\ha}^\upsig \varphi_\tau^{-1}
         \cdot \{{\ha}^\upsig \psi_\tau,{\ha}^\upsig\hmm g\} )
        \cdot
        {\ha}^{\psi_\sigma }\varphi_\sigma^{-1}
        \{\psi_\sigma,g\} \\
        &= {\ha}^\upg u_{\sigma,\tau} \cdot
        {\ha}^{g \psi_\sigma {\ha}^\upsig \psi_\tau g^{-1}}\varphi_{\sigma \tau}
        \cdot
         {\ha}^{\rho(\{\psi_\sigma{\ha}^\upsig\psi_\tau,g\}^{-1})\psi_\sigma{\ha}^\upsig\psi_\tau}{\ha}^\upsig \varphi_\tau^{-1}
         \cdot
         \{\psi_\sigma{\ha}^\upsig \psi_\tau,g\}^{-1}
         \cdot {\ha}^{\psi_\sigma} \{{\ha}^\upsig \psi_\tau,{\ha}^\upsig\hmm g\}
        \cdot
        {\ha}^{\psi_\sigma }\varphi_\sigma^{-1}
        \{\psi_\sigma,g\}
        &&\textit{ by } \eqref{e:CM1}\\
        &= {\ha}^\upg u_{\sigma,\tau} \cdot
        {\ha}^{g \psi_\sigma {\ha}^\upsig \psi_\tau g^{-1}}\varphi_{\sigma \tau}
        \cdot
         {\ha}^{g \psi_\sigma {\ha}^\upsig \psi_\tau g^{-1}}{\ha}^\upsig \varphi_\tau^{-1}
         \cdot
         \{\psi_\sigma{\ha}^\upsig \psi_\tau,g\}^{-1}
         \cdot {\ha}^{\psi_\sigma} \{{\ha}^\upsig \psi_\tau,{\ha}^\upsig\hmm g\}
        \cdot
        {\ha}^{\psi_\sigma }\varphi_\sigma^{-1}
        \{\psi_\sigma,g\}
        &&\textit{ by } \eqref{e:Br1}\\
        &= {\ha}^\upg u_{\sigma,\tau} \cdot
        {\ha}^{g \psi_\sigma {\ha}^\upsig \psi_\tau g^{-1}}(\varphi_{\sigma \tau}
        \cdot
         {\ha}^\upsig \varphi_\tau^{-1})
         \cdot
         \{\psi_\sigma,g\}^{-1}
         \cdot
         {\ha}^{\psi_\sigma}\{{\ha}^\upsig \psi_\tau,g\}^{-1}
         \cdot {\ha}^{\psi_\sigma} \{{\ha}^\upsig \psi_\tau,{\ha}^\upsig\hmm g\}
        \cdot
        {\ha}^{\psi_\sigma }\varphi_\sigma^{-1}
        \{\psi_\sigma,g\}
        &&\textit{ by } \eqref{e:Br5}
        \\
        &= {\ha}^\upg u_{\sigma,\tau} \cdot
        {\ha}^{g \psi_\sigma {\ha}^\upsig \psi_\tau g^{-1}}(\varphi_{\sigma \tau}
        \cdot
         {\ha}^\upsig \varphi_\tau^{-1})
         \cdot
         {\ha}^{\rho(\{\psi_\sigma{\ha},g\})^{-1}}
         {\ha}^{\psi_\sigma}(\{{\ha}^\upsig \psi_\tau,g\}^{-1}
         \cdot \{{\ha}^\upsig \psi_\tau,{\ha}^\upsig\hmm g\}
        \cdot
        \varphi_\sigma^{-1})
        &&\textit{ by } \eqref{e:CM1}\\
        &= {\ha}^\upg u_{\sigma,\tau} \cdot
        {\ha}^{g \psi_\sigma {\ha}^\upsig \psi_\tau g^{-1}}(\varphi_{\sigma \tau}
        \cdot
         {\ha}^\upsig \varphi_\tau^{-1})
         \cdot
         {\ha}^{g \psi_\sigma g^{-1}}(\{{\ha}^\upsig \psi_\tau,g\}^{-1}
         \cdot \{{\ha}^\upsig \psi_\tau,{\ha}^\upsig\hmm g\}
        \cdot
        \varphi_\sigma^{-1})
        &&\textit{ by } \eqref{e:Br1}\\
        &= {\ha}^\upg u_{\sigma,\tau} \cdot
        {\ha}^{g \psi_\sigma {\ha}^\upsig \psi_\tau g^{-1}}(\varphi_{\sigma \tau}
        \cdot
         {\ha}^\upsig \varphi_\tau^{-1} \cdot \varphi_\sigma^{-1})
         \cdot
         {\ha}^{(g \psi_\sigma g^{-1})(g {\ha}^\upsig \psi_\tau g^{-1})}\varphi_\sigma
         \cdot
         {\ha}^{g \psi_\sigma g^{-1}}(\{{\ha}^\upsig \psi_\tau,g\}^{-1}
         \cdot \{{\ha}^\upsig \psi_\tau,{\ha}^\upsig\hmm g\}
        \cdot
        \varphi_\sigma^{-1})
        \\
        &= {\ha}^\upg u_{\sigma,\tau} \cdot
        {\ha}^{g \psi_\sigma {\ha}^\upsig \psi_\tau g^{-1}}(\varphi_{\sigma \tau}
        \cdot
         {\ha}^\upsig \varphi_\tau^{-1} \cdot \varphi_\sigma^{-1})
         \cdot
         {\ha}^{g \psi_\sigma g^{-1}}({\ha}^{g {\ha}^\upsig \psi_\tau g^{-1}}
         \varphi_\sigma
         \cdot \{{\ha}^\upsig \psi_\tau,g\}^{-1}
         \cdot \{{\ha}^\upsig \psi_\tau,{\ha}^\upsig\hmm g\}
        \cdot
        \varphi_\sigma^{-1})
        \\
        &= {\ha}^\upg u_{\sigma,\tau} \cdot
        {\ha}^{g \psi_\sigma {\ha}^\upsig \psi_\tau g^{-1}}(\varphi_{\sigma \tau}
        \cdot
         {\ha}^\upsig \varphi_\tau^{-1} \cdot \varphi_\sigma^{-1})
         \cdot
         {\ha}^{g \psi_\sigma g^{-1}}(
         \{{\ha}^\upsig \psi_\tau,g\}^{-1}
         \cdot {\ha}^{\rho(\{{\ha}^\upsig \psi_\tau,g\})g {\ha}^\upsig \psi_\tau g^{-1}}
         \varphi_\sigma
         \cdot \{{\ha}^\upsig \psi_\tau,{\ha}^\upsig\hmm g\}
        \cdot
        \varphi_\sigma^{-1})
        &&\textit{ by } \eqref{e:CM1}\\
        &= {\ha}^\upg u_{\sigma,\tau} \cdot
        {\ha}^{g \psi_\sigma {\ha}^\upsig \psi_\tau g^{-1}}(\varphi_{\sigma \tau}
        \cdot
         {\ha}^\upsig \varphi_\tau^{-1} \cdot \varphi_\sigma^{-1})
         \cdot
         {\ha}^{g \psi_\sigma g^{-1}}(
         \{{\ha}^\upsig \psi_\tau,g\}^{-1}
         \cdot {\ha}^{{\ha}^\upsig \psi_\tau}
         \varphi_\sigma
         \cdot \{{\ha}^\upsig \psi_\tau,{\ha}^\upsig\hmm g\}
        \cdot
        \varphi_\sigma^{-1})
        &&\textit{ by } \eqref{e:Br1}\\
        &= {\ha}^\upg u_{\sigma,\tau} \cdot
        {\ha}^{g \psi_\sigma {\ha}^\upsig \psi_\tau g^{-1}}(\varphi_{\sigma \tau}
        \cdot
         {\ha}^\upsig \varphi_\tau^{-1} \cdot \varphi_\sigma^{-1})
         \cdot
         {\ha}^{g \psi_\sigma g^{-1}}{\ha}^\upg \{{\ha}^\upsig \psi_\tau,g^{-1} \rho(\varphi_\sigma) {\ha}^\upsig\hmm g\}
        &&\textit{ by } \eqref{eq:lem_tech-1}\\
        &= {\ha}^\upg(u_{\sigma,\tau} \cdot
        {\ha}^{\psi_\sigma {\ha}^\upsig \psi_\tau}({\ha}^{g^{-1}}(\varphi_{\sigma \tau}
        \cdot
         {\ha}^\upsig \varphi_\tau^{-1} \cdot \varphi_\sigma^{-1}))
         \cdot
         {\ha}^{\psi_\sigma} \{{\ha}^\upsig \psi_\tau,g^{-1} \rho(\varphi_\sigma) {\ha}^\upsig\hmm g\})
        \\
        &= {\ha}^\upg u''_{\sigma,\tau}
        &&\textit{ by }\eqref{eq:def_prod_1}.
    \end{align*}
    This completes the proof of Lemma \ref{l:crossed-module}.
\end{proof}

Since \eqref{e:bar-d} is a crossed module, the image of $\bar d$
is a normal subgroup of $Z^1(\Gamma, A\to G)$, and therefore
$\coker\,\bar d$ is a  well-defined group;
see \cite[Lemma 3.3.2]{Borovoi-Memoir}.
We wish to obtain a group structure on $H^1(\Gamma,A \to G)$ by identifying it with $\operatorname{coker} \bar{d}$.

Recall that $H^1(\Gamma,A \to G)$ is the set of orbits of $Z^1(\Gamma,A \to G)$ under the right $C^0(\Gamma,A \to G)$-action denoted by $(u,\psi) * (\varphi,g)$.
This action is defined by \eqref{eq:action_C0-1}-\eqref{eq:action_C0-2}
with respect to the group structure on $C^0(\Gamma,A \to G)$ given by \eqref{e:grp_C0_Bor};
see \cite[Section 3.3]{Borovoi-Memoir}.

We relate this action to the image of $\bar{d}$.

\begin{lemma}\label{l:noname3}
    For any $(u,\psi) \in Z^1(\Gamma,A \to G)$ and $(\varphi,g) \in C^0(\Gamma,A \to G)$, we have
    \begin{align*}
        (u,\psi) * (\varphi,g)
        = d((\varphi \cdot \delta(\psi,g),g)) \cdot (u,\psi),
    \end{align*}
    where $\delta(\psi,g) \colon \Gamma \to A$ is defined by $\delta(\psi,g)_\sigma = \{{\ha}^\upsig\hmm g,\psi_\sigma\}^{-1}$.
\end{lemma}

\begin{proof}
    We denote $(u',\psi') = (u,\psi) * (\varphi,g)$ and $(u'',\psi''') = d\big(\varphi \cdot \delta(\psi,g)\big) \cdot (u,\psi)$.
    We show that $(u'',\psi'') = (u',\psi')$ as follows:
    \begin{align*}
        \psi''_\sigma
        &=(g^{-1} \cdot\rho(\varphi_\sigma \cdot \delta(\psi,g)_\sigma) \cdot {\ha}^\upsig\hmm g) \cdot \psi_\sigma
        &&\textit{ by } \eqref{e:HM1} \\
        &=
        g^{-1} \cdot
        \rho(\varphi_\sigma) \cdot
        [{\ha}^\upsig\hmm g,\psi_\sigma]^{-1}
        \cdot {\ha}^\upsig\hmm g \cdot \psi_\sigma &&\textit{ by } \eqref{e:Br1} \\
        &=g^{-1} \cdot \rho(\varphi_\sigma) \cdot \psi_\sigma \cdot {\ha}^\upsig\hmm g \\
        &= \psi'_\sigma &&\textit{ by } \eqref{eq:action_C0-2},
    \end{align*}
    \begin{align*}
        &{\ha}^\upg u''_{\sigma,\tau}= \\
        &=\varphi_{\sigma \tau} \cdot \delta(\psi,g)_{\sigma\tau}
        \cdot {\ha}^{{\ha}^{\upst} g} u_{\sigma,\tau}
        \cdot \{{\ha}^{\upst} g,\psi_\sigma\}
        \cdot
        {\ha}^{\psi_\sigma \sigma}(\varphi_\tau \cdot \delta(\psi,g)_\tau)^{-1}
        \cdot
        \{{\ha}^\upsig\hmm g,\psi_\sigma\}^{-1}
        \cdot
        (\varphi_\sigma \cdot \delta(\psi,g)_\sigma)^{-1}
        &&\textit{ by } \eqref{eq:lem_tech-3} \\
        &=\varphi_{\sigma \tau} \cdot \{{\ha}^{\upst} g,\psi_{\sigma \tau}\}^{-1}
        \cdot {\ha}^{{\ha}^{\upst} g} u_{\sigma,\tau}
        \cdot \{{\ha}^{\upst} g,\psi_\sigma\}
        \cdot
        {\ha}^{\psi_\sigma}\{{\ha}^{\upst} g,{\ha}^\upsig \psi_\tau\}
        \cdot
        {\ha}^{\psi_\sigma \sigma}\varphi_\tau \cdot
        \varphi_\sigma^{-1}
        \\
        &=\varphi_{\sigma \tau} \cdot \{{\ha}^{\upst} g,\psi_{\sigma \tau}\}^{-1}
        \cdot {\ha}^{{\ha}^{\upst} g} u_{\sigma,\tau}
        \cdot \{{\ha}^{\upst} g,\psi_\sigma {\ha}^\upsig \psi_\tau\}
        \cdot
        {\ha}^{\psi_\sigma \sigma}\varphi_\tau \cdot
        \varphi_\sigma^{-1}
        &&\textit{ by }\eqref{e:Br4}\\
        &=\varphi_{\sigma \tau} \cdot \{{\ha}^{\upst} g,\psi_{\sigma \tau}\}^{-1}
        \cdot \{{\ha}^{\upst} g,\rho(u_{\sigma,\tau})\psi_\sigma {\ha}^\upsig \psi_\tau\}
        \cdot u_{\sigma,\tau}
        \cdot
        {\ha}^{\psi_\sigma \sigma}\varphi_\tau \cdot
        \varphi_\sigma^{-1}
        &&\textit{ by }\eqref{e:B.1}\\
        &=\varphi_{\sigma \tau} \cdot \{{\ha}^{\upst} g,\psi_{\sigma \tau}\}^{-1}
        \cdot \{{\ha}^{\upst} g,\psi_{\sigma \tau}\}
        \cdot u_{\sigma,\tau}
        \cdot
        {\ha}^{\psi_\sigma \sigma}\varphi_\tau \cdot
        \varphi_\sigma^{-1}
        &&\textit{ by }\eqref{eq:def1_1-cocycle}\\
        &=\varphi_{\sigma \tau} \cdot
        \cdot u_{\sigma,\tau}
        \cdot
        {\ha}^{\psi_\sigma \sigma}\varphi_\tau \cdot
        \varphi_\sigma^{-1}
        \\
        &= {\ha}^\upg u'_{\sigma,\tau} &&\textit{ by } \eqref{eq:action_C0-1}.
    \end{align*}
    This completes the proof of Lemma \ref{l:noname3}.
\end{proof}

Lemma \ref{l:noname3} implies the following corollary:

\begin{corollary}\label{cor:H1_coker}
    The canonical map $Z^1(\Gamma,A \to G)\to H^1(\Gamma,A \to G)$ induces
    a  canonical bijection $H^1(\Gamma,A \to G) \cong \operatorname{coker} \bar{d}$.
    In particular, $H^1(\Gamma,A \to G)$ inherits a canonical group structure.
\end{corollary}

In the case that the braiding $\br$ is symmetric,
the group structure on $H^1(\Gamma,A \to G)$ becomes abelian.
We first construct the braiding on the crossed module \eqref{e:bar-d}
defined as follows:

\begin{lemma}\label{lem:braided_cross_mod_C0_Z1}
    The crossed module $\bar{d} \colon C^0(\Gamma,A \to G)/B^0(\Gamma,A \to G)\to Z^1(\Gamma,A \to G)$
    of \eqref{e:bar-d} is a braided crossed module with the braiding defined by
    \begin{align}
        \{(u^1,\psi^1),(u^2,\psi^2)\} \coloneqq
        \big[\{\psi^1,\psi^2\},1\big]
        \label{eq:br_Z1}
    \end{align}
    where $\{\psi^1,\psi^2\} \colon \Gamma \to A,\ \, \sigma \mapsto \{\psi^1_\sigma,\psi^2_\sigma\}$.
    Here $\big[\{\psi^1,\psi^2\},1\big]\in  C^0(\Gamma,A \to G)/B^0(\Gamma,A \to G)$  denotes the class
    of the 0-cochain $\big(\{\psi^1,\psi^2\},1\big)\in C^0(\Gamma,A \to G)$.
\end{lemma}

\begin{proof}
    We need to prove \eqref{e:Br1}-\eqref{e:Br5}.
    Since $\eqref{e:Sym}$ holds by its construction, the condition \eqref{e:Br2}
    (resp. \eqref{e:Br4}) is equivalent to \eqref{e:Br3} (resp. \eqref{e:Br5}).
    Thus it suffices to show that $\eqref{e:Br1}$, $\eqref{e:Br2}$, and $\eqref{e:Br4}$ hold.
    To prove \eqref{e:Br1}, we wish to show that
    \[
    \bar{d}\big(\{(u^1,\psi^1),(u^2,\psi^2)\}\big) \cdot (u^2,\psi^2) \cdot (u^1,\psi^1)
    =
    (u^1,\psi^1) \cdot (u^2,\psi^2)
    \]
    for $(u^1,\psi^1), (u^2,\psi^2) \in Z^1(\Gamma,A \to G)$.
    We denote
    \[(u^{\prime\,2,1},\psi^{\prime\,2,1}) = \bar{d}(\{(u^1,\psi^1),(u^2,\psi^2)\}) \cdot (u^2,\psi^2) \cdot (u^1,\psi^1)\]
    and
    \[(u^{1,2},\psi^{1,2}) = (u^1,\psi^1) \cdot (u^2,\psi^2).\]
    The desired equality is proved as follows:
    \begin{align*}
        \psi^{\prime\,2,1}_{\hs\sigma}
        &= \rho(\{\psi^2_\sigma,\psi^1_\sigma\}) \cdot \psi^2_\sigma \cdot \psi^1_\sigma
        &&\textit{ by } \eqref{e:HM2}, \eqref{eq:br_Z1}, \eqref{eq:def_prod_2} \\
        &= \psi^1_\sigma \cdot \psi^2_\sigma &&\textit{ by }\eqref{e:Br1},
    \end{align*}
    \begin{align*}
        &u^{\prime\,2,1}_{\,\sigma,\tau} =\\
        &= \{\psi^1_{\sigma \tau}, \psi^2_{\sigma \tau}\}
        \cdot
        u^2_{\sigma,\tau}
        \cdot {\ha}^{\psi^2_\sigma} \{{\ha}^\upsig \psi^1_\tau,{\ha}^\upsig \psi^2_\tau\}^{-1}
        \cdot
        \{\psi^1_\sigma,\psi^2_\sigma\}^{-1}
        &&\textit{ by } \eqref{eq:lem_tech-3}, \eqref{eq:def_prod_1}\\
        &\quad
        \cdot\big(
        {\ha}^{\rho(\{\psi^1_\sigma,\psi^2_\sigma\})\psi^2_\sigma\rho(\{{\ha}^\upsig\psi^1_\tau,{\ha}^\upsig\psi^2_\tau\}) {\ha}^\upsig\psi^2_\tau}
        u^1_{\sigma,\tau}\big)
        \cdot\big(
        {\ha}^{\rho(\{\psi^1_\sigma,\psi^2_\sigma\})\psi^2_\sigma}
        \{\rho(\{{\ha}^\upsig\psi^1_\tau,{\ha}^\upsig\psi^2_\tau\}) {\ha}^\upsig\psi^2_\tau,\psi^1_\sigma\}\big) \\
        &= \{\psi^1_{\sigma \tau}, \psi^2_{\sigma \tau}\}
        \cdot
        u^2_{\sigma,\tau}
        \cdot {\ha}^{\psi^2_\sigma}({\ha}^{{\ha}^{\upsig} \psi^2_\tau}u^1_{\sigma,\tau}
        \cdot \{{\ha}^\upsig \psi^1_\tau,{\ha}^\upsig \psi^2_\tau\}^{-1}
        \cdot \{\rho(\{{\ha}^\upsig\psi^1_\tau,{\ha}^\upsig\psi^2_\tau\}) {\ha}^\upsig\psi^2_\tau,\psi^1_\sigma\}
        )
        \cdot \{\psi^1_\sigma,\psi^2_\sigma\}^{-1}
        &&\textit{ by } \eqref{e:CM1} \\
        &= \{\psi^1_{\sigma \tau}, \psi^2_{\sigma \tau}\}
        \cdot
        u^2_{\sigma,\tau}
        \cdot {\ha}^{\psi^2_\sigma}({\ha}^{{\ha}^{\upsig} \psi^2_\tau}u^1_{\sigma,\tau}
        \cdot \{{\ha}^\upsig\psi^2_\tau,\psi^1_\sigma\}
        \cdot {\ha}^{\psi^1_\sigma}\{{\ha}^\upsig \psi^1_\tau,{\ha}^\upsig \psi^2_\tau\}^{-1}
        )
        \cdot \{\psi^1_\sigma,\psi^2_\sigma\}^{-1}
        &&\textit{ by } \eqref{e:B.2} \\
        &= \{\psi^1_{\sigma \tau}, \psi^2_{\sigma \tau}\}
        \cdot
        u^2_{\sigma,\tau}
        \cdot {\ha}^{\psi^2_\sigma}({\ha}^{{\ha}^{\upsig} \psi^2_\tau}u^1_{\sigma,\tau}
        \cdot \{{\ha}^\upsig\psi^2_\tau,\psi^1_\sigma\}
        \cdot {\ha}^{\psi^1_\sigma}\{{\ha}^\upsig \psi^2_\tau,{\ha}^\upsig \psi^1_\tau\}
        )
        \cdot \{\psi^2_\sigma,\psi^1_\sigma\}
        &&\textit{ by } \eqref{e:Sym}
        \\
        &= \{\psi^1_{\sigma \tau}, \psi^2_{\sigma \tau}\}
        \cdot
        u^2_{\sigma,\tau}
        \cdot {\ha}^{\psi^2_\sigma}({\ha}^{{\ha}^{\upsig} \psi^2_\tau}u^1_{\sigma,\tau}
        \cdot \{{\ha}^\upsig\psi^2_\tau,\psi^1_\sigma {\ha}^\upsig \psi^2_\tau\}
        )
        \cdot \{\psi^2_\sigma,\psi^1_\sigma\}
        &&\textit{ by } \eqref{e:Br4} \\
        &= \{\psi^1_{\sigma \tau}, \psi^2_{\sigma \tau}\}
        \cdot
        u^2_{\sigma,\tau}
        \cdot {\ha}^{\psi^2_\sigma}(\{{\ha}^\upsig\psi^2_\tau,\rho(u^1_{\sigma,\tau})\psi^1_\sigma {\ha}^\upsig \psi^2_\tau\}
        \cdot u^1_{\sigma,\tau}
        )
        \cdot \{\psi^2_\sigma,\psi^1_\sigma\}
        &&\textit{ by } \eqref{e:B.1}
        \\
        &= \{\psi^1_{\sigma \tau}, \rho(u^2_{\sigma,\tau}) \psi^2_{\sigma} {\ha}^\upsig \psi^2_\tau\}
        \cdot
        u^2_{\sigma,\tau}
        \cdot {\ha}^{\psi^2_\sigma}(\{{\ha}^\upsig\psi^2_\tau,\psi^1_{\sigma \tau}\}
        \cdot u^1_{\sigma,\tau}
        )
        \cdot \{\psi^2_\sigma,\psi^1_\sigma\}
        &&\textit{ by } \eqref{eq:def1_1-cocycle}
         \\
        &= {\ha}^{\psi^1_{\sigma \tau}} u^2_{\sigma,\tau}
        \cdot \{\psi^1_{\sigma \tau}, \psi^2_{\sigma} {\ha}^\upsig \psi^2_\tau\}
        \cdot {\ha}^{\psi^2_\sigma}(\{{\ha}^\upsig\psi^2_\tau,\psi^1_{\sigma \tau}\}
        \cdot u^1_{\sigma,\tau}
        )
        \cdot \{\psi^2_\sigma,\psi^1_\sigma\}
        &&\textit{ by } \eqref{e:B.1}
        \\
        &= {\ha}^{\psi^1_{\sigma \tau}} u^2_{\sigma,\tau}
        \cdot \{\psi^1_{\sigma \tau}, \psi^2_{\sigma}\}
        \cdot {\ha}^{\psi^2_\sigma} \{\psi^1_{\sigma \tau}, {\ha}^\upsig \psi^2_\tau\}
        \cdot {\ha}^{\psi^2_\sigma}(\{{\ha}^\upsig\psi^2_\tau,\psi^1_{\sigma \tau}\}
        \cdot u^1_{\sigma,\tau}
        )
        \cdot \{\psi^2_\sigma,\psi^1_\sigma\}
        &&\textit{ by } \eqref{e:Br4}
        \\
        &= {\ha}^{\psi^1_{\sigma \tau}} u^2_{\sigma,\tau}
        \cdot \{\psi^1_{\sigma \tau}, \psi^2_{\sigma}\}
        \cdot {\ha}^{\psi^2_\sigma} u^1_{\sigma,\tau}
        \cdot \{\psi^2_\sigma,\psi^1_\sigma\}
        &&\textit{ by } \eqref{e:Sym}
        \\
        &= {\ha}^{\psi^1_{\sigma \tau}} u^2_{\sigma,\tau}
        \cdot \{\rho(u^1_{\sigma,\tau})\psi^1_\sigma {\ha}^\upsig \psi^2_\tau, \psi^2_{\sigma}\}
        \cdot {\ha}^{\psi^2_\sigma} u^1_{\sigma,\tau}
        \cdot \{\psi^2_\sigma,\psi^1_\sigma\}
        &&\textit{ by } \eqref{eq:def1_1-cocycle}
        \\
        &= {\ha}^{\psi^1_{\sigma \tau}} u^2_{\sigma,\tau}
        \cdot u^1_{\sigma,\tau}
        \cdot \{\psi^1_\sigma {\ha}^\upsig \psi^2_\tau, \psi^2_{\sigma}\}
        \cdot \{\psi^2_\sigma,\psi^1_\sigma\}
        &&\textit{ by } \eqref{e:B.2}
        \\
        &= {\ha}^{\psi^1_{\sigma \tau}} u^2_{\sigma,\tau}
        \cdot u^1_{\sigma,\tau}
        \cdot {\ha}^{\psi^1_\sigma} \{{\ha}^\upsig \psi^2_\tau, \psi^2_{\sigma}\}
        \cdot \{\psi^2_\sigma,\psi^1_\sigma\}
        &&\textit{ by } \eqref{e:Br5}, \eqref{e:Sym}
        \\
        &= u^1_{\sigma,\tau} \cdot {\ha}^{\psi^1_{\sigma}{\ha}^\upsig \psi^1_\tau} u^1_{\sigma,\tau}
        \cdot
        \cdot {\ha}^{\psi^1_\sigma} \{{\ha}^\upsig \psi^2_\tau, \psi^2_{\sigma}\}
        \cdot \{\psi^2_\sigma,\psi^1_\sigma\}
        &&\textit{ by } \eqref{e:CM1}, \eqref{eq:def1_1-cocycle}
        \\
        &= u^{1,2}_{\sigma,\tau} &&\textit{ by }\eqref{eq:def_prod_1}.
    \end{align*}
    To prove \eqref{e:Br2}, we wish to show that
    \[
    \{d((\varphi,g)),(u,\psi)\} \cdot {\ha}^{(u,\psi)} (\varphi,g) = (\varphi,g)
    \]
    for any $(\varphi,g) \in C^0(\Gamma,A \to G)$ and $(u,\psi) \in Z^1(\Gamma,A \to G)$.
    We denote the left-hand side by $(\varphi',g') = \{d((\varphi,g)),(u,\psi)\} \cdot {\ha}^{(u,\psi)} (\varphi,g)$.
    The desired equality is proved as follows:
    \begin{align*}
        g' = 1 \cdot g = g \quad \textit{ by } \eqref{e:HM2}, \eqref{eq:act_Z1}, \eqref{eq:br_Z1},
    \end{align*}
    \begin{align*}
        \varphi'_\sigma
        &=\{\rho(\{g^{-1} \rho(\varphi_\sigma) {\ha}^\upsig\hmm g,\psi_\sigma\}),g\}^{-1}
        \cdot \{g^{-1}\rho(\varphi_\sigma){\ha}^\upsig\hmm g,\psi_\sigma\}
        \cdot\{\psi_\sigma,g\}^{-1}
        {\ha\ha}^{\psi_\sigma}\hm \varphi_\sigma
        \{{\ha}^\upsig\hmm g,\psi_\sigma\}^{-1}
        &&\textit{ by } \eqref{e:HM2}, \eqref{eq:product_C0}, \\
        & && \,\,\,\,\quad \eqref{eq:act_Z1}, \eqref{eq:br_Z1} \\
        &=\{\rho(\{g^{-1} \rho(\varphi_\sigma) {\ha}^\upsig\hmm g,\psi_\sigma\}),g\}^{-1}
        \cdot \{g^{-1}\rho(\varphi_\sigma){\ha}^\upsig\hmm g,\psi_\sigma\}
        \cdot\{\psi_\sigma,g\}^{-1}
        {\ha\ha}^{\psi_\sigma}\hm \varphi_\sigma
        \{\psi_\sigma,{\ha}^\upsig\hmm g\}
        &&\textit{ by } \eqref{e:Sym} \\
        &=\{\rho(\{g^{-1} \rho(\varphi_\sigma) {\ha}^\upsig\hmm g,\psi_\sigma\}),g\}^{-1}
        \cdot \{g^{-1}\rho(\varphi_\sigma){\ha}^\upsig\hmm g,\psi_\sigma\}
        \cdot
        {\ha}^\upg\{\psi_\sigma,g^{-1}\rho(\varphi_\sigma){\ha}^\upsig\hmm g\} \varphi_\sigma
        &&\textit{ by } \eqref{eq:lem_tech-1} \\
        &={\ha}^\upg\{g^{-1}\rho(\varphi_\sigma){\ha}^\upsig\hmm g,\psi_\sigma\}
        \cdot
        {\ha}^\upg\{\psi_\sigma,g^{-1}\rho(\varphi_\sigma){\ha}^\upsig\hmm g\} \varphi_\sigma
        &&\textit{ by } \eqref{e:Br2} \\
        &= \varphi_\sigma &&\textit{ by }\eqref{e:Sym}.
    \end{align*}
    To prove $\eqref{e:Br4}$, we wish to show that
    \[
    \{(u^1,\psi^1) , (u^2,\psi^2) \cdot (u^3,\psi^3)\}
    =
    \{(u^1,\psi^1) , (u^2,\psi^2)\} \cdot {\ha}^{(u^2,\psi^2)} \{(u^1,\psi^1) , (u^3,\psi^3)\},
    \]
    which is proved as follows:
    \begin{align*}
        \{(u^1,\psi^1) , (u^2,\psi^2) \cdot (u^3,\psi^3)\}
        &= (\{\psi^1,\psi^2 \psi^3\},1)
        &&\textit{ by } \eqref{eq:br_Z1} \\
        &= (\{\psi^1,\psi^2\} {\ha\ha}^{\psi^2}\!\{\psi^1,\psi^3\},1)
        &&\textit{ by } \eqref{e:Br4} \\
        &= (\{\psi^1,\psi^2\},1) \cdot ({\ha\ha}^{\psi^2}\!\{\psi^1,\psi^3\},1)
        &&\textit{ by } \eqref{eq:product_C0} \\
        &= (\{\psi^1,\psi^2\},1) \cdot {\ha}^{(u^2,\psi^2)}(\{\psi^1,\psi^3\},1)
        &&\textit{ by } \eqref{eq:act_Z1} \\
        &= \{(u^1,\psi^1) , (u^2,\psi^2)\} \cdot {\ha}^{(u^2,\psi^2)} \{(u^1,\psi^1) , (u^3,\psi^3)\}
        && \textit{ by } \eqref{eq:br_Z1}
    \end{align*}
    where ${\ha}^{\psi^2}\!\{\psi^1,\psi^3\} \colon \sigma \mapsto {\ha}^{\psi^2_\sigma}\{\psi^1_\sigma,\psi^3_\sigma\}$.
    This completes the proof of Lemma \ref{lem:braided_cross_mod_C0_Z1}.
\end{proof}

\begin{remark}
    It is easy to see that the braiding defined in Lemma \ref{lem:braided_cross_mod_C0_Z1} is a Picard braiding.
\end{remark}

\begin{corollary}
    The group structure on $H^1(\Gamma,A \to G)$ is abelian.
\end{corollary}
\begin{proof}
    In the presence of the braiding in Lemma \ref{lem:braided_cross_mod_C0_Z1}, we have
    \[
    \bar{d}\big(\{(u^1,\psi^1),(u^2,\psi^2)\}\big) = \Big[(u^1,\psi^1),(u^2,\psi^2)\Big]
    \]
    for any $(u^1,\psi^1), (u^2,\psi^2) \in Z^1(\Gamma,A \to G)$,
    where the square brackets denote the commutator; see \eqref{e:Br1} in Proposition \ref{p:br}.
    This shows that the image of $\bar{d}$ contains the commutator subgroup of $Z^1(\Gamma,A \to G)$, which yields that $\operatorname{coker}\bar{d}$ is abelian.
    Now Corollary \ref{cor:H1_coker} completes the proof.
\end{proof}

\end{document}